\documentclass[11pt]{article}
\usepackage{amsmath,amsthm,mathrsfs,bm}

\usepackage[initials,alphabetic]{amsrefs}
\usepackage[utf8]{inputenc}

\usepackage[dvipsnames]{xcolor}

\usepackage{amsfonts,amstext,amssymb,verbatim,epsfig,psfrag,stmaryrd,extarrows}
\usepackage{pgf,tikz}
\usepackage{float}
\usetikzlibrary{arrows}
\usepackage{hyperref}
\usepackage{color}
\usepackage{transparent}
\usepackage{subfigure}
\usepackage{prettyref}
\newrefformat{pr}{Proposition \ref{#1}}
\newrefformat{co}{Corollary \ref{#1}}
\newrefformat{le}{Lemma \ref{#1}}
\newrefformat{th}{Theorem \ref{#1}}
\newrefformat{ex}{Example \ref{#1}}
\newrefformat{re}{Remark \ref{#1}}
\newrefformat{se}{Section \ref{#1}}

\usepackage{esint}

\usepackage{bbm,mathrsfs}

\usepackage[normalem]{ulem}

\definecolor{forestgreen}{rgb}{0.1333,0.5451,0.1333}
\definecolor{navyblue}{rgb}{0,0,0.5}
\definecolor{darkgreen}{rgb}{0,0.3922,0}

\hypersetup{colorlinks=true, citecolor  = teal, linkcolor=MidnightBlue}
\makeatletter
\let\reftagform@=\tagform@
\def\tagform@#1{\maketag@@@{(\ignorespaces\textcolor{black}{#1}\unskip\@@italiccorr)}}
\renewcommand{\eqref}[1]{\textup{\reftagform@{\ref{#1}}}}
\makeatother

\makeatletter

\allowdisplaybreaks

\def\lessim{\ \lower4pt\hbox{$
		\buildrel{\displaystyle <}\over\sim$}\ }
\def\gessim{\ \lower4pt\hbox{$\buildrel{\displaystyle >}
		\over\sim$}\ }

\def\si{\sigma}

\def\eps{{\varepsilon}}

\newcommand{\indi}{\ensuremath{\mathbbm{1}}}

\newcommand{\bt}{\boldsymbol{t}}

\newcommand{\pref}{\prettyref}

\newtheorem{lemma}{\bf Lemma}[section]

\newtheorem{theorem}[lemma]{\bf Theorem}

\newtheorem{proposition}[lemma]{\bf Proposition}

\theoremstyle{remark}

\newtheorem{remark}{Remark}[section]

\numberwithin{equation}{section}

\newcommand{\8}{\infty}

\newcommand{\px}{\mathcal{P}}

\newcommand{\rz}{\mathbb{R}}

\newcommand{\ez}{\mathbb{E}}

\newcommand{\pz}{\mathbb{P}}

\newcommand{\sfG}{\mathsf G}
\newcommand{\sfF}{\mathsf F}
\newcommand{\sfH}{\mathsf H}
\newcommand{\sfI}{\mathsf I}
\newcommand{\sfL}{\mathsf L}
\newcommand{\sfJ}{\mathsf J}
\newcommand{\sfK}{\mathsf K}
\newcommand{\sfT}{\mathsf T}

\newcommand{\sfa}{\mathsf a}
\newcommand{\sfb}{\mathsf b}

\newcommand{\al}{\alpha}
\newcommand{\de}{\delta}
\renewcommand{\si}{\sigma}

\newcommand{\la}{\lambda}
\renewcommand{\bt}{\beta}

\newcommand{\Crt}{\mathrm{Crt}}

\newcommand{\GOE}{\mathrm{GOE}}

\newcommand{\dd}{\mathrm{d}}

\begin{document}
\title{Hessian spectrum at the global minimum and topology trivialization of locally isotropic Gaussian random fields}
\author{Hao Xu\thanks{Department of Mathematics, University of Macau, yc17446@connect.um.edu.mo.} \and Qiang Zeng\thanks{Academy of Mathematics and Systems Science, Chinese Academy of Sciences, qzeng.math@gmail.com.}}
\date{October 13, 2024}
\maketitle
\begin{abstract}
We study the energy landscape near the ground state of a model of a single particle in a random potential with trivial topology. More precisely, we find the large dimensional limit of the Hessian spectrum at the global minimum of the Hamiltonian $X_N(x) +\frac\mu2 \|x\|^2, x\in\rz^N,$ when $\mu$ is above the phase transition threshold so that the system has only one critical point. Here $X_N$ is a locally isotropic Gaussian random field. The same idea is also applied to study the more general model of elastic manifold. In the replica symmetric regime, our results confirm the predictions of Fyodorov and Le Doussal made in 2018 and 2020 using the replica method.
\end{abstract}

\section{Introduction}
In 1941, Kolmogorov introduced locally isotropic fields \cite{Ko41} for the application in statistical theory of turbulence. Since then, this model has been applied to study various problems in statistical physics. In particular, for disordered systems, Mezard--Parisi \cite{MP91} and Engel \cite{En93} employed locally isotropic fields to model a classical particle confined to an impenetrable spherical box or to describe elastic manifolds propagating in a random potential. In particular, this model can be understood as the zero dimensional case of elastic manifold. For more background of this model, we refer the interested readers to the papers \cite{Ya57,Fy04, FS07, BD07, Kli12, AZ20} among many others. On the other hand, since its discovery, Parisi's replica trick has been applied to study plenty of models in physics and combinatorial optimization, and led to fascinating predictions and conjectures, while the replica method itself remains mysterious in mathematics. There has been lasting interest on investigating predictions made by the replica method for various models; see \cite{MPV86,Par23}. Here we continue this line of research and study the locally isotropic Gaussian random fields and, more generally, the elastic manifold.

The first model is defined as follows. Let $B_{N} \subset \mathbb{R}^{N}$ be a sequence of subsets and let $H_{N}: B_{N} \subset$ $\mathbb{R}^{N} \rightarrow \mathbb{R}$ be given by
\begin{align}
\label{eq:m}
H_{N}(x)=X_{N}(x)+\frac{\mu}{2}\|x\|^{2},
\end{align}
where $\mu >0,\|x\|$ is the Euclidean norm of $x$, and $X_{N}$ is a locally isotropic Gaussian random field, namely, a Gaussian process satisfying
$$
\mathbb{E}\left[\left(X_{N}(x)-X_{N}(y)\right)^{2}\right]=N D\left(\frac{1}{N}\|x-y\|^{2}\right), \quad x, y \in \mathbb{R}^{N} .
$$
Due to this defining property, $X_N$ is also known as a Gaussian random field with isotropic increments.
Here the function $D: \mathbb{R}_{+} \rightarrow \mathbb{R}_{+}$ is called the structure function and $\mathbb{R}_{+}=[0, \infty)$. It determines the law of $X_{N}$ up to an additive shift by a Gaussian random variable.
Complete characterization of all structure functions of locally isotropic Gaussian fields was given by Yaglom \cite{Ya57} (see also early work of Schoenberg in the context of embedding metric spaces into Hilbert spaces \cite{Sch38}). Following \cite[Section 25.3]{Ya87}, if $D$ is a structure function for all $N \in \mathbb{N}$, then $X_{N}$ must belong to one of the following two classes:

\begin{enumerate}
    \item \textbf{Isotropic fields.} There exists a function $B:\rz_+\to \rz$ such that
        \begin{align*}
          \ez[X_N(x)X_N(y)]=N B\Big(\frac1N\|x-y\|^2\Big),
        \end{align*}
        where $B$ has the representation
        \begin{align}\label{funB}
        B(r) = c_0+\int_{(0,\8)} e^{-r t^2} \nu(\dd t),
        \end{align}
        and $c_0\in \rz_+$ is a constant and $\nu$ is a finite measure on $(0,\8)$. In this case, $B$ is called a correlator function and
        \begin{align*}
          D(r)=2(B(0)-B(r)).
        \end{align*}
    \item \textbf{Non-isotropic fields with isotropic increments.} The structure function $D$ can be written as
        \begin{align}\label{eq:drep}
          D(r) = \int_{(0,\8)} (1-e^{-rt^2})\nu(\dd t) + Ar, \ \ r\in \rz_+,
        \end{align}
        where $A\in \rz_+$ is a constant and $\nu$ is a $\si$-finite (but not finite) measure satisfying
        \begin{align*}
          \int_{(0,\8)} \frac{t^2}{1+t^2}\nu(\dd t) <\8.
        \end{align*}
  \end{enumerate}
In the physics literature,  Case 1 is also known as short-range correlation (SRC)  fields and Case 2 as long-range correlation (LRC) fields.

We also consider the more general model of elastic manifold. Following \cite{fyodorov2020manifolds,FLD20,BABM21}, we denote by $\Omega$ the discrete lattice $\llbracket 1, L \rrbracket^d \subset \mathbb{Z}^d$, where the positive integers $L$ and $d$ are called the length and the internal dimension, respectively.
We parameterize the manifold by an $N$-component field $\mathbf{u}(x) \in \mathbb{R}^N$ for $x\in\Omega$. Let $X_{N}$ be a Gaussian random field on $\mathbb{R}^N \times \Omega$ satisfying  
$$
\mathbb{E}\left[X_N\left(y_1, x_1\right) X_N\left(y_2, x_2\right)\right]=N B\left(\frac{\left\|y_1-y_2\right\|^2}{N}\right) \delta_{x_1, x_2},
$$
where $B$ is the correlator function defined in (\ref{funB}) and $\de_{x,y}=\de_{xy}$ is the Kronecker delta function.
The Hamiltonian of elastic manifold is
\begin{equation}\label{modEM}
    \mathcal{H}(\mathbf{u})=\sum_{x, y \in \Omega}\left(\mu_0 I_{L^d}-t_0 \Delta\right)_{x y}\langle\mathbf{u}(x), \mathbf{u}(y)\rangle+\sum_{x \in \Omega} X_N(\mathbf{u}(x), x),
\end{equation}
for any deterministic function $
\mathbf{u}: \Omega \rightarrow \mathbb{R}^N$.
In the above definition, $\Delta \in \mathbb{R}^{L^d \times L^d}$ is the lattice Laplacian on $\Omega$ and its $(x, y)$ entry is given by
$\Delta_{x y}=\delta_{x \sim y}-2 d \delta_{xy}$, 
where $x \sim y$ means that $x$ and $y$ are lattice neighbors. The positive numbers $\mu_0$ and $t_0$ are called mass and interaction strength in this model, respectively. In this context, the model \pref{eq:m} is the $d=0$ case of elastic manifold.

Our study is motivated by the physics literature \cite{FLD18,fyodorov2020manifolds}. In these remarkable papers, Fyodorov and Le Doussal considered the energy landscape of the models \pref{eq:m} and \eqref{modEM} near their global minimum. They investigated the behavior of the lower (or left) edge of Hessian spectrum at the global minimum using the replica method, and discovered phase transitions in the thermodynamic limit according to different levels of replica symmetry breaking. For the model \pref{eq:m}, they showed that the limiting Hessian spectrum at the global minimum is the semicircle law with radius $2\sqrt{B''(0)}$ and center
\begin{align*}
\mu_{\mathrm{eff}}=\mu+\frac{B^{\prime \prime}(0)}{\mu}+v\left(B^{\prime}(\mathcal{Q})-B^{\prime}(0)-\mathcal{Q} B^{\prime \prime}(0)\right),
\end{align*}
where the values of parameters $v>0, \mathcal{Q} \ge 0$ can be solved by the following two equations
\begin{align*}
&\frac{\mu^{2} \mathcal{Q}}{1-\mu v \mathcal{Q}}=B^{\prime}(\mathcal{Q})-B^{\prime}(0), \\
&\frac{1}{v} \log (1-\mu v \mathcal{Q})=-\mu \mathcal{Q}+v\left[B(\mathcal{Q})-B(0)-\mathcal{Q} B^{\prime}(\mathcal{Q})\right].
\end{align*}
Here one should interpret the result using the relation $D'(r)=-2B'(r)$ for LRC fields. In the replica symmetric regime where $\mathcal{Q}=0$,  the lower edge of the semicircle law is given by
\begin{align*}
    \la_{-}^{\rm SRC} &= \Big(\sqrt{\mu}-\sqrt{\frac{B''}\mu}\Big)^2,\\
    \la_{-}^{\rm LRC} &= \Big(\sqrt{\mu}-\sqrt{\frac{-D''(0)}{2\mu}}\Big)^2;
\end{align*}
see the discussion after Equation (9) in \cite{FLD18}. Note that the scaling in the above formulas is inherited from \cite[Equation (2)]{FLD18} which is different from that in \eqref{eq:m}. With our scaling, we would have a factor 4 in front of $B''$ and $D''$.
Their formulas for the model \eqref{modEM} are more involved; but in the  replica symmetric regime the lower edge of limiting Hessian spectrum has a simple form. Indeed, writing $\hat\mu_A$ for the empirical spectral measure of the matrix $A$ and $I_N$ for the $N\times N$ identity matrix, we define 
\begin{align}\label{ustar}
    u^*= 4B''(0)\int \frac{\hat{\mu}_{-t_0 \Delta+\mu_0I_{L^d}}(\mathrm{d} \lambda)}{\lambda}.
\end{align}
Denoting by $m(z)$ the Stieltjes transform of limiting Hessian spectrum at the global minimum, Fyodorov and Le Doussal predicted the following self-consistency equation \cite[Equations (45) and (107)]{fyodorov2020manifolds} in the replica symmetric regime
\begin{equation}\label{Stie}
m(z)=\int \frac{\hat{\mu}_{-t_0 \Delta+\mu_0 I_{L^d}}(\mathrm{d} s)}{s+u^*-z-4B''(0)m(z)}.
\end{equation}
Note that $ip$ in \cite[Equation (45)] {fyodorov2020manifolds} should be understood as the Cauchy transform which has opposite sign with $m(z)$, and the parameter $\mu_{\mathrm{eff}}$ is given in \cite[Equation (107)] {fyodorov2020manifolds}.
It was shown in {\cite[Theorem 2.4]{BABM21}} that the following equation for $\mu_c=\mu_c\left(t_0, 4B''(0), L, d\right)$ has a unique positive solution, which is called the Larkin mass:
\begin{align}\label{larkin}
    \int_{\mathbb{R}} \frac{\hat{\mu}_{-t_0 \Delta}(\mathrm{d} \lambda)}{\left(\mu_c+\lambda\right)^2}=\frac{1}{4B''(0)}.
\end{align}
As a consequence, the lower edge of the limiting Hessian spectrum is given by \cite[Equation (20)]{fyodorov2020manifolds}
\begin{align}\label{eq:emledge}
    \ell=\mu_0-\mu_c+4B''(0)\int \frac{\hat{\mu}_{-t_0 \Delta}(\mathrm{d} s)}{s+\mu_0}-4B''(0)\int \frac{\hat{\mu}_{-t_0 \Delta}(\mathrm{d} s)}{s+\mu_c}.
\end{align}
In short, the goal of this paper is to prove all the above replica symmetric predictions rigorously.

To state our results precisely, we make the following hypothesis throughout the paper so that the random fields are twice differentiable; see e.g.~\cite{AT07}.

\textbf{Assumption I} (Smoothness). The function $D$ is four times differentiable at 0, and it satisfies
$$
0<\left|D^{(4)}(0)\right|<\infty .
$$

Given a (random) smooth function $f$, we write $\nabla f(x)$ and $\nabla^2 f(x)$ for its gradient and Hessian at $x$ in its domain, respectively.
Let $\mathcal{P}(\mathbb{R})$ denote the space of probability measures on $\rz$.
Recall the bounded Lipschitz metric between two probability measures $\mu, \nu\in \px(\rz)$ is given by
\begin{align}\label{eq:measd}
d(\mu,\nu)=\sup\Big\{\Big|\int f \dd \mu-\int f \dd\nu \Big|: \|f\|_\8\le1, \|f\|_L\le 1\Big\},
\end{align}
where $\|f\|_\8$ and $\|f\|_L$ denote the $L^\8$ norm and Lipschitz constant of the function $f$, respectively.  Let $B(\nu, \delta)$ denote the open ball in $\mathcal{P}(\mathbb{R})$ with center $\nu$ and radius $\delta$ with respect to the metric $d$ given by (\ref{eq:measd}). Let $\si_{c,r}$ denote the semicircle measure with center $c$ and radius $r$, i.e.,
\begin{align*}
    \si_{c,r}(\dd x) = \frac{2}{\pi r^2}\sqrt{r^2-(x-c)^2} \dd x.
\end{align*}
In particular, we write $\si_{\rm sc} = \si_{0,\sqrt2}$. The empirical spectral measure of an $N\times N$ matrix $M$ is denoted by $\hat\mu_M=L_{N}(M)=\frac{1}{N} \sum_{i=1}^{N} \delta_{\lambda_{i}(M)}$, where $\la_i(M)$ are the eigenvalues of $M$. Let $\lambda_{\min}(M)$ and $\lambda_{\max}(M)$ denote the smallest and largest eigenvalue of $M$, respectively.  We use the shorthand notations $B=B(0), B'=B'(0), B''=B''(0)$ in what follows.
\begin{theorem}[SRC fields]
    \label{h-3}
    Assume Assumption I. If $\mu>\sqrt{4 B^{\prime \prime}}$, then
    \begin{align}
    \label{xx}
    \lim _{N \rightarrow \infty} \mathbb{P}\left(\text {The only critical point of } H_{N}(x) \text{ is global minimum} \right)=1.
    \end{align}
    Let $x^{*}$ denote the global minimum of $H_{N}(x)$. Then we have convergence in probability
    \begin{align}
    &\label{hh3}\lim _{N \rightarrow \infty} \frac{1}{N} H_{N}\left(x^{*}\right) =\frac{B^{\prime}}{\mu},\\
    &\label{q10}\lim _{N \rightarrow \infty} \frac{\|x^{*}\|}{\sqrt{N}} =\frac{\sqrt{-2 B^{\prime}}}{\mu},\\
    &\label{qq10}\lim_{N\to \8} d(L_N(\nabla^2 H_N(x^*)), \si_{c_s,r_s}) =0,\\
    &\label{xx8}\lim _{N \rightarrow \infty}  \lambda_{\min }\left(\nabla^{2} H_{N}\left(x^{*}\right)\right)=c_s-r_s,
    \end{align}
    where $c_s=\mu+\frac{4B''}{\mu}$ and $r_s=4\sqrt{B''}$.
\end{theorem}
Here \eqref{xx} is essentially known from \cite{Fy04,Fy15} before. As mentioned earlier, the different coefficients of $B''$ compared with the predictions in \cite{FLD18} are due to the different scaling in the definition \pref{eq:m} here and in what follows. For LRC fields, we also need the following extra assumption, which is natural for stochastic processes with stationary increments. 

\textbf{Assumption II} (Pinning). We have
$
X_{N}(0)=0.
$

The following is our main results for LRC fields.
\begin{theorem}[LRC fields]
    \label{h-2}
    Assume Assumptions I and II. If $\mu>\sqrt{-2 D^{\prime \prime}(0)}$, then
    \begin{align*}
    \lim _{N \rightarrow \infty} \mathbb{P}\left(\text {The only critical point of } H_{N}(x) \text{ is global minimum} \right)=1.
    \end{align*}
    Let $x^{*}$ be the global minimum of $H_{N}(x)$. Then we have convergence in probability
    \begin{align}
    &\lim _{N \rightarrow \infty} \frac{1}{N} H_{N}\left(x^{*}\right) =\frac{D^{\prime}\left(\frac{D'(0)}{\mu^2}\right)-D^{\prime}(0)}{\mu}+\frac{D'(0)}{2\mu}-\frac{ D^{\prime}\left(\frac{D'(0)}{\mu^2}\right)}{\mu},\label{eq:energy}\\
    &\lim _{N \rightarrow \infty} \frac{\|x^{*}\|}{\sqrt{N}} =\frac{\sqrt{D^{\prime}(0)}}{\mu},\label{g1}\\
    &\lim_{N\to\8} d(L_N(\nabla^2 H_N(x^*)), \si_{c_l,r_l})= 0,\label{gg1}\\
    &\lim _{N \rightarrow \infty}  \lambda_{\min }\left(\nabla^{2} H_{N}\left(x^{*}\right)\right)=c_l-r_l,\label{zx}
    \end{align}
    where $c_l=\mu+\frac{-2D''(0)}{\mu}$ and $r_l=\sqrt{-8D''(0)}$.

\end{theorem}

The ideas of establishing the above results can be applied to the study of more realistic elastic manifold model. For the translation $\tau_u(x)=x+u$ on $\rz$, we denote by $\tau_u(\mu)$ the pushforward of a probability measure $\mu$ under $\tau_u$. The Pastur relation \cite{Pa72} reveals that the limiting measure in the following result exactly matches the Stieltjes transform in the prediction \eqref{Stie}.

\begin{theorem}[Elastic manifold]
    \label{Elas}
        Assume Assumption I.
       The global minimum of $\mathcal{H}(\mathbf{u})$ exists and is unique which we denote by $\mathbf{u}^{*}$.  If $\mu_0 >\mu_c\left(t_0, 4B''(0), L, d\right)$, then we have convergence in probability
        \begin{align}
        \label{meau}
        &\lim_{N\to \8} d\left(\hat{\mu}_{\nabla^2 \mathcal{H}(\mathbf{u}^*)}, \tau_{-u^*}(\sigma_{0, 4\sqrt{B''}} \boxplus \hat{\mu}_{-t_0 \Delta+\mu_0 I_{L^d}})\right) =0,
    \end{align}
    where $\boxplus$ denotes the free additive convolution of probability measures. Consequently, the lower edge of the limiting measure $\tau_{-u^*}(\sigma_{0, 4\sqrt{B''}} \boxplus \hat{\mu}_{-t_0 \Delta+\mu_0 I_{L^d}})$ is given by \pref{eq:emledge}.
    \end{theorem}


We briefly review previous results related to landscape complexity since it is the major tool for our proofs. Let $B_{N} \subset \mathbb{R}^{N}$ and $E \subset \mathbb{R}$ be Borel sets. We define
\begin{align}\label{eq:crtn}
\operatorname{Crt}_{N}\left(E, B_{N}\right) &=\#\left\{x \in B_{N}: \nabla H_{N}(x)=0, \frac{1}{N} H_{N}(x) \in E\right\}.
\end{align}
In the seminal work \cite{Fy04}, Fyodorov computed the large $N$ limit of $\ez[\Crt_N(\rz,\rz^N)]$ (which is known as complexity in the physics literature) for SRC fields and found a phase transition depending on the value of $\mu/\sqrt{B''(0)}$.
Since then, there has been a large body of work on this topic for isotropic fields on Euclidean spaces (see e.g.~\cite{FW07,FN12}) and on spheres (see e.g.~\cite{ABC13,ABA13,subag2017complexity,Mihai}). The latter is known as spherical mixed $p$-spin model in spin glasses. Recently, Auffinger and the second named author considered the complexity of  LRC fields \cite{AZ20,AZ22} assuming a technical condition which was removed in \cite{XYZ23}. Ben Arous, Bourgade and McKenna studied the complexity of elastic manifold and other models beyond invariance ensembles in \cite{BABM21,BABM22} using technology from universality theory of random matrices. In \cite{Fy15} Fyodorov showed that $\lim_{N\to\8}\ez[\Crt_N(\rz,\rz^N)]=1$ when $\mu>\sqrt{4B''(0)}$ (see also \cite{FN12} for the same conclusion for local minima).  This phenomenon is called topology trivialization when it was studied for another system by Fyodorov and Le Doussal in the pioneering work \cite{FLD14}.   For topology trivialization, there are recent works on the (variants of) spherical mixed $p$-spin model \cite{Be22, ABL22}. The work \cite{Be22} also discussed the Hessian spectrum at the global minimum of the spherical mixed $p$-spin model in the trivial regime, which is close to our investigation here.

\subsection*{Proof ideas, contributions, novelty and limitation}
In general, we follow the strategy of \cite{ABC13, ABA13, Fy15, Be22} where as many works in this direction the Kac--Rice representation of the mean number of critical points is employed and the convergence in probability is derived from vanishing complexity. Our argument relies crucially on the exact limit (not only the order) of mean number of critical points. To this end, we take advantage of various results and arguments developed in previous works for the models under consideration \cite{Fy15,AZ20,BABM21}. Typically, the landscape complexity is given as a variational problem over some parameters representing the location variable (e.g.~$\rho$ in Proposition \ref{x-10} and Theorem \ref{x-5}), the critical value (e.g.~$u$ in Proposition \ref{x-10} and Theorem \ref{x-5}) and other `physical' quantities that are related to constraints of critical points. Besides, there is a free parameter which is related to the random matrices used in the proof; see the variable $y$ in Lemma \ref{le:stech}, Proposition \ref{x-10} and Theorem \ref{x-5}, and the variable $u$ in Lemma \ref{unimax} and Equation \eqref{sups}. In the replica symmetric regime, the complexity function attains its maximum value zero at a unique point. The convergence in probability follows from avoiding this maximizer and using the Markov inequality. In this way, we prove limit theorems for the ground state energy (see Equations (\ref{hh3}) and (\ref{eq:energy}) for SRC fields and LRC fields respectively) and the minimizers (see Equations (\ref{q10}) and (\ref{g1}) for SRC fields and LRC fields respectively) like those in \cite{Be22}. However, for the convergence of empirical spectral measures we have to go one step further from the landscape complexity analysis, as we will now clarify. We need to additionally localize the free parameter via using the large deviation principle (LDP) with speed $N^2$ or concentration inequalities for the empirical spectral measures of Gaussian matrices.  As a technical contribution of our proofs, we conclude that the maximizer of this free parameter ($y^*$ or $u^*$) is responsible for the center of the limiting (deformed) semicircle distribution. This observation may be known to experts before, but it seems not explicitly emphasized to the best of our knowledge. 

For comparison, the argument of \cite{Be22} relies on the fact that the Hessian of the spherical mixed $p$-spin model can be written as a sum of two independent parts--a rescaled GOE matrix plus the radial derivative of the Hamiltonian. For SRC fields considered here, such a decomposition is no longer available; and for LRC fields, the Hessian has an even more involved structure due to its dependence on the location. As a result, some novel ideas must be developed on top of \cite{Fy15,AZ20,AZ22}. The role of radial derivative is replaced by the aforementioned free variable inside the integral (which is relatively direct for SRC fields but is in a rather `twisted' way for LRC fields). For elastic manifold, we simply apply the above ideas to the hard work already done in \cite{BABM21}.

For the convergence of the smallest eigenvalues, we adopt an idea of \cite{Be22} on bounding the second moment of the determinant of a shifted matrix from the Gaussian Orthogonal Ensemble (GOE) with the first moment squared; see Lemma \ref{x-13}. The proof of \eqref{zx} is the most difficult part of this paper. Morally speaking, the isotropy assumption is destroyed in the radial direction for the LRC fields, which results in one defective entry in the conditional distribution of the Hessian compared with a GOE matrix; see \pref{eq:gu}. The structure of this conditional distribution motivates us to utilize tools for rank one perturbations of GOE matrices; specifically, we use the Dumitriu--Edelman tridiagonal representation \cite{DE02} and the factorization of Bloemendal and Vir\'{a}g \cite{BV13} to get estimates for the smallest eigenvalue of the conditional distribution of Hessian (the matrix $G$) in Lemma \ref{51}. A pleasant fact we find is that the seemingly strange defective entry is within the BBP phase transition threshold and does not create an outlier for the limiting spectrum at a deterministic radial position. Together with complexity analysis, this allows us to show that the lower edge of the limiting empirical spectrum distribution matches the limit of the smallest eigenvalues of Hessians at the global minimum of the LRC fields. Note that the original prediction in \cite{FLD18} was about the limiting empirical spectral measures of Hessians. As Fyodorov pointed out to us (private communication), convergence of the smallest eigenvalues verifies  the replica symmetric prediction from \cite{FLD18} in a strong sense.

Outside the replica symmetric regime, the complexity is not vanishing and one may find exponentially many critical points. In this case, the method used here is not expected to give a solution to the conjecture of Fyodorov and Le Doussal. A natural approach to the problem is to identify when the complexity function of mean number of critical points considered here correctly gives the typical number of critical points, which requires a second moment computation as first carried out for the spherical pure $p$-spin model in \cite{subag2017complexity}. As observed in \cite{ABA13} for the spherical mixed $p$-spin model, concentration for complexity could fail, which may set a limit for the application of complexity analysis to the study of the limiting Hessian spectrum at the global minimum. In this sense, we may have pushed the first moment method to the extreme for studying the Hessian spectrum problems even though it is robust enough to handle several models in the replica symmetric regime. The more involved replica symmetry breaking regime needs more novel ideas and will be addressed in the future.

The paper is organized as follows. In Section \ref{se:euminimum}, we show that the global minimum of the model exists and is unique. Then we prove Theorem \ref{h-3} in Section \ref{se:src}. As the topology trivialization regime for LRC fields has not been determined yet, we identify this regime in Section \ref{se:toptriv} and then prove Theorem \ref{h-2} in Section \ref{se:lrc}. Finally, Theorem \ref{Elas} for elastic manifold is proved in Section \ref{se:em}.

\section{Existence and uniqueness of the global minimum}\label{se:euminimum}

Due to compactness, we know any smooth function on the sphere has at least two critical points, one global maximum and the other global minimum. Such a simple topological argument does not apply to our models. The next result shows that the Hamiltonian of \pref{eq:m} must attain its global minimum in $\rz^N$ almost surely.
\begin{proposition}\label{z-1}
    Let $X_N$ be a locally isotropic Gaussian random field. Then almost surely,
    \[
    \lim_{\|x\|\to+ \8} X_N(x)+\frac{\mu}{2}\|x\|^2 =+\8.
    \]
    Therefore, $H_N(x)$ attains its global minimum in $\rz^N$.
\end{proposition}
\begin{proof}
    Let $d_X(x,y)=(\ez[X_N(x)-X_N(y)]^2)^{1/2}$ for $x,y\in\rz^N$. Then by Assumption I,
    \[
        d_X(x,y)^2=N D(\frac1N\|x-y\|^2)\le  D'(0)\|x-y\|^2.
    \]
    Let $B_R=\{x\in\rz^N: \|x\|\le \sqrt{N} R\}$. It is well known that the covering number of $B_R$ with Euclidean $\eps$-ball is bounded above by $(1+\frac{2\sqrt{N}R}{\eps})^N$. It follows that the covering number of $B_R$ with $d_X$ $\eps$-ball is
    \[
    N(B_R,d_X,\eps)\le  (1+\frac{2\sqrt{N D'(0)} R}{\eps})^N.
    \]
    By Dudley's theorem \cite{AT09},
    \[
    \ez\Big[\sup_{x\in B_R} X_N(x)\Big]\le 24 \int_0^{\sqrt{N D'(0) }R} [\log N(B_R,d_X,\eps)]^{1/2} \dd \eps   \le 48\sqrt2 NR\sqrt{D'(0)}<\8.
    \]
    In particular, $X_N$ is a.s.~bounded. Let $\si_R^2=\sup_{x\in B_R} \ez[X_N(x)^2]=ND(R^2)$. Using the Borell--TIS inequality and writing $k=\ez[\sup_{x\in B_R} X_N(x)]$, we deduce
    \begin{align*}
        \pz\Big(\sup_{x\in B_R}X_N(x) \ge \sqrt{D'(0)}NR^{3/2}\Big) \le e^{-\frac{(\sqrt{D'(0)}NR^2-k)^2}{2ND(R^2)}}\le e^{-\frac{NR}{4}}
    \end{align*}
    for $R$ large. By symmetry, we have $\pz(\inf_{x\in B_R}X_N(x) \le -\sqrt{D'(0)}NR^{3/2})\le e^{-\frac{NR}{4}}$. Using the Borel--Cantelli lemma, almost surely we have
    \[
      \liminf_{R\to\8}   \inf_{x\in B_R}X_N(x)  +\sqrt{D'(0)}NR^{3/2}\ge0.
    \]
    This yields
    \[
    \liminf_{R\to\8}      \inf_{\sqrt{N}R/2\le  \|x\|\le \sqrt{N}R} X_N(x)  +\frac{\mu}{2}\|x\|^2\ge \liminf_{R\to\8}   \frac{\mu N R^2}{8}-\sqrt{D'(0)}NR^{3/2} =+\8.
    \]
    From here the assertion follows.
\end{proof}

For elastic manifold, we let $\mathbf{U}$ denote the vector space that consists of functions defined in (\ref{modEM}), which is equipped with the norm 
$$\|\mathbf{u}\|=\left(\sum_{x\in \Omega}\|\mathbf{u}(x)\|^2\right)^{1/2},\quad \text{for all} \quad\mathbf{u}\in\mathbf{U}.$$
Direct calculation yields
\begin{align}\label{mu0}
    \sum_{x, y \in \Omega}\left(\mu_0 I_{L^d}\right)_{x y}\langle\mathbf{u}(x), \mathbf{u}(y)\rangle=\mu_0\|\mathbf{u}\|^2.
\end{align}
On the other hand, it is well known that the lattice Laplacian $\Delta$ is semi-negative definite. By the elementary inequality $\langle\mathbf{u}(x), \mathbf{u}(x)\rangle+\langle\mathbf{u}(y), \mathbf{u}(y)\rangle\geq 2\langle\mathbf{u}(x), \mathbf{u}(y)\rangle$, for any function $\mathbf{u}$, we deduce 
\begin{align}\label{lap}
  \sum_{x, y \in \Omega}-t_0 \Delta_{x y}\langle\mathbf{u}(x), \mathbf{u}(y)\rangle \geq 0. 
\end{align}
For convenience, we set \begin{align}\label{redef}
    X_N(\mathbf{u})=\sum_{x \in \Omega} X_N(\mathbf{u}(x), x).
    \end{align} 
Therefore, combining (\ref{mu0}) with (\ref{lap}), we have
\begin{equation}\label{relarg}
    \mathcal{H}[\mathbf{u}]\geq \mu_0\|\mathbf{u}\|^2+X_N(\mathbf{u}).
\end{equation}
\begin{proposition}\label{minimum}
    With $X_N(\mathbf{u})$ defined in (\ref{redef}), we have almost surely,
    \[
    \lim_{\|\mathbf{u}\|\to+ \8} \mu_0\|\mathbf{u}\|^2+X_N(\mathbf{u})=+\8.
    \]
    Therefore, $\mathcal{H}(\mathbf{u})\to+ \8$ as $\|\mathbf{u}\|\to+ \8$, and thus $\mathcal{H}(\mathbf{u})$ attains its global minimum on $\mathbf{U}$ almost surely.
\end{proposition}
\begin{proof}
For $\mathbf{u}_1,\mathbf{u}_2\in\mathbf{U}$, we define 
\begin{align*}
d_X\left(\mathbf{u}_1,\mathbf{u}_2\right)&=\left(\ez[X_N(\mathbf{u}_1)-X_N(\mathbf{u}_2)]^2\right)^{1/2}\\&=\left(\sum_{x\in\Omega}\ez[X_N(\mathbf{u}_1(x),x)-X_N(\mathbf{u}_2(x),x)]^2\right)^{1/2}\\&
    =\left(\sum_{x\in\Omega}N D\left(\frac1N\|\mathbf{u}_1(x)-\mathbf{u}_2(x)\|^2\right)\right)^{1/2},
\end{align*} 
where the function $D(r)=2(B(0)-B(r))$. By Assumption I, we have
    \[
d_X(\mathbf{u}_1,\mathbf{u}_2)^2\leq\sum_{x\in\Omega} D'(0)\|\mathbf{u}_1(x)-\mathbf{u}_2(x)\|^2=  D'(0)\|\mathbf{u}_1-\mathbf{u}_2\|^2.
    \]
    Let $B_R=\{\mathbf{u}\in\mathbf{U}: \|\mathbf{u}\|\le \sqrt{NL^d} R\}$. Then the assertion follows from the same proof as for Proposition \ref{z-1} after replacing $N$ by $NL^d$ and observing \eqref{relarg}.
\end{proof}

The uniqueness of the global minimum can be derived using a version of Bulinskaya's lemma as in \cite[Proposition 3.1]{Bov99}. More directly, we may apply \cite[Lemma 2.6]{KP90} to get the almost sure uniqueness of the global minimum thanks to the continuity of sample paths of our models.

To end this section, we fix some notations for random matrices that will be used throughout. Let $M^N$ be an $N\times N$ matrix from the Gaussian Orthogonal Ensemble (GOE), i.e., $M^N$ is a real symmetric matrix whose entries $M_{i j}, i \leq j$ are independent centered Gaussian random variables with variance
$$
\mathbb{E} M_{i j}^{2}=\frac{1+\delta_{i j}}{2 N} .
$$
 We also write $\GOE_N$ for an $N\times N$ GOE matrix independent of all other randomness in question. Let $\rho_N(x)$ denote the density of expected empirical measure of $\GOE_N$, i.e., for a bounded continuous function $f$,
\[
\frac1N\ez\Big[\sum_{i=1}^Nf(\lambda_{i}\left(\GOE_N\right))\Big] =  \int f(x) \rho_N(x) \dd x,
 \]
where $\la_i(\GOE_N)$ are eigenvalues of $\GOE_N$. The empirical spectral measures of GOE matrices satisfy an LDP with speed $N^2$ \cite{benarous1997large}, which implies that for any $\delta>0$ there exists $c>0$ such that for large $N$
\begin{align}
\label{ac2}
\mathbb{P}\left(L_{N}\left(\GOE_N\right) \notin B\left(\sigma_{\mathrm{sc}}, \delta\right)\right) \leq e^{-c N^2},
\end{align}
where $B(\sigma_{\mathrm{sc}}, \delta)$ denotes the open ball in the space $\mathcal{P}(\mathbb{R})$ with center $\sigma_{\mathrm{sc}}$ and radius $\delta$ with respect to the metric $d$ given by \pref{eq:measd}.
By the LDP of the smallest eigenvalues of GOE matrices \cite{BDG01}, we have for all $\varepsilon>0$ there is a $\delta>0$ such that for $N$ large enough
\begin{align}
\label{pp6}
\mathbb{P}\left(\left|\lambda_{\min}\left(\GOE_N\right)+\sqrt{2}\right|>\varepsilon\right) \leq e^{-\delta N}.
\end{align}
We also have the large deviation estimate for the operator norm  \cite{BDG01}
\begin{align}\label{eq:opld}
\mathbb{P}\left(\lambda^*\left(\GOE_N\right)>K\right) \leq e^{-N K^{2} / 9}
\end{align}
for $K$ and  $N$ large enough, where $\la^*(\GOE_N)=\max_{1\le i\le N}|\la_i(\GOE_N)|$.  It follows that there exists a constant $C>0$ such that for any $k \geq 0$,
\begin{align}
\label{ac1}
\mathbb{E}\left[\lambda^{*}(\GOE_N)^k\right] \leq C^k.
\end{align}

\section{SRC fields}\label{se:src}

We first give the covariance structure when $X_N$ is an SRC field. We omit the proof as it is elementary and was given elsewhere; see e.g.~\cite[Lemma 3.2]{CS18}.

\begin{lemma}
\label{x-8}
Assume Assumption I.
 Consider the isotropic field model with correlator function $B$. Then for $x \in \mathbb{R}^{N}$,
$$
\begin{aligned}
\operatorname{Cov}\left[H_{N}(x), H_{N}(x)\right] &=NB(0), \\
\operatorname{Cov}\left[H_{N}(x), \partial_{i} H_{N}(x)\right]&=\operatorname{Cov}\left[\partial_{i} H_{N}(x), \partial_{j k} H_{N}(x)\right]=0, \\
\operatorname{Cov}\left[\partial_{i} H_{N}(x), \partial_{j} H_{N}(x)\right] &=-\operatorname{Cov}\left[H_{N}(x), \partial_{i j} H_{N}(x)\right]=-2 B^{\prime}(0) \delta_{i j}, \\
\operatorname{Cov}\left[\partial_{i j} H_{N}(x), \partial_{k l} H_{N}(x)\right] &=4 B^{\prime \prime}(0)\left[\delta_{i j} \delta_{k l}+\delta_{i k} \delta_{j l}+\delta_{i l} \delta_{j k}\right]/ N,
\end{aligned}
$$
where $\delta_{i j}$ are the Kronecker delta function.
\end{lemma}

From Lemma \ref{x-8}, it is straightforward to check that
\begin{align}
    \nabla^{2} H_{N}(x) &\stackrel{d}{=}\sqrt{8B^{\prime\prime}}\GOE_{N}+\sqrt{\frac{4B^{\prime\prime}}{N}}Z I_{N}+\mu I_{N} , \label{eq:goe1}\\
\Big(\nabla^{2} H_{N}(x) \Big| \frac{H_{N}(x)}{N}=u\Big) & \stackrel{d}{=}\sqrt{8B^{\prime\prime}}\GOE_{N}+\sqrt{\frac{4(BB^{\prime\prime}-B^{\prime 2})}{N B}}Z I_{N} \label{eq:goe2}\\ & \quad+\Big[\mu+\frac{2B^{\prime}}{B}(u-\frac{\mu\|x\|^{2}}{2N})\Big] I_{N}, \nonumber
\end{align}
 where $Z$ is a standard Gaussian random variable independent of $\GOE_{N}$, and
$ I_{N}$ is the $N \times N$ identity matrix. To simplify the notation, we define
\begin{align}
\label{wn}
W^{N}=\sqrt{8B^{\prime\prime}}(\GOE_{N}+\eta I_{N}),
\end{align}
where $\eta$ is a Gaussian random variable independent of $\GOE_{N}$ with mean $\frac{\mu+\frac{2B^{\prime}}{B}(u-\frac{\mu\|x\|^{2}}{2N})}{\sqrt{8B^{\prime\prime}}}$ and variance $\frac{BB^{\prime\prime}-B^{\prime2}}{2NBB^{\prime\prime}}$. By (\ref{eq:goe2}), we know $W^{N} \stackrel{d}{=} \left(\nabla^{2} H_{N}(x) | \frac{H_{N}(x)}{N}=u\right)$.

The following topology trivialization result is taken from \cite{Fy04,Fy15}. We put it here for later references. 
\begin{proposition}
\label{p-2}
 Recall from \pref{eq:crtn} that $\operatorname{Crt}_{N}\left(\mathbb{R}, \mathbb{R}^{N}\right)$ is the total number of critical points of $H_{N}$.
If $\mu>\sqrt{4 B^{\prime \prime}(0)}$, then
    \begin{align}
    \label{pp1}
    \lim _{N \rightarrow \infty} \mathbb{E}\left[\operatorname{Crt}_{N}\left(\mathbb{R}, \mathbb{R}^{N}\right)\right]=1.
    \end{align}
\end{proposition}

The next lemma is well known; see e.g.~\cite{Fy15,SZ22,Be22}. We add a proof for definiteness of the constants.
\begin{lemma}
\label{x-9}
Recall $\GOE_{N}$ is an $N \times N$ GOE matrix. Then for any $x \in \mathbb{R}$,
\begin{align*}
\mathbb{E}\left[\left|\operatorname{det}\left(\GOE_{N}\pm x  I_{N}\right)\right|\right]=\sqrt{2(N+1)} N^{-N / 2} \Gamma\left(\frac{N+1}{2}\right) e^{N x^{2} / 2} \rho_{N+1}\Big(\sqrt {\frac{N}{N+1}}x\Big) .
\end{align*}
\end{lemma}

\begin{proof}
Assume $X$ is a Gaussian random variable with mean $x$ and variance $\sigma^2$. Using  \cite[Lemma 3.3]{ABC13} with $m=x$, $t=\sigma$ and summing over the eigenvalues, we obtain
$$
\mathbb{E}\left[\left|\operatorname{det}\left(-X  I_{N}+\GOE_{N}\right)\right|\right]=\frac{(N+1)\Gamma\left(\frac{N+1}{2}\right)N^{-\frac{N+1}{2}}}{\sqrt{\pi \sigma^{2}}}\int_{\mathbb{R}} e^{\frac{(N+1)t^2}{2}-\frac{(\sqrt {\frac{N+1}{N}}t-x)^2}{2\sigma^2}}\rho_{N+1}(t)\dd t.
$$
Using the change of variable $y=\frac{\sqrt {\frac{N+1}{N}}t-x}{\sigma}$ yields
$$
\begin{aligned}
&\mathbb{E}\left[\left|\operatorname{det}\left(-X  I_{N}+\GOE_{N}\right)\right|\right]=\sqrt{2(N+1)} N^{-N / 2} \Gamma\left(\frac{N+1}{2}\right)\\& \quad\quad\quad\quad\quad\quad\quad\quad\quad\quad\quad\quad\quad\quad \times\int_{\mathbb{R}}\frac{1}{\sqrt{2 \pi}}e^{-\frac{y^2}{2}+ \frac{N(\sigma y+x)^2}{2}}\rho_{N+1}\Big(\sqrt {\frac{N}{N+1}}(\sigma y+x)\Big)\dd y.
\end{aligned}
$$
Sending $\sigma\rightarrow 0$ in the above equality, we get the desired result by symmetry.
\end{proof}

Let
$$
\Psi_{*}(x)=\int_{\mathbb{R}} \log |x-t|\sigma_{\mathrm{sc}} (\dd t).
$$
By calculation,
\begin{align*}
\Psi_{*}(x)= \begin{cases}\frac{1}{2} x^{2}-\frac{1}{2}-\frac{1}{2} \log 2, & |x| \leq \sqrt{2}, \\\frac{1}{2} x^{2}-\frac{1}{2}-\log 2-\frac{1}{2}|x| \sqrt{x^{2}-2}+\log \left(|x|+\sqrt{x^{2}-2}\right), & |x|>\sqrt{2}.\end{cases}
\end{align*}
For $(\rho,u,y)\in \rz_+\times \rz\times \rz$, let us define
\begin{align}\label{eq:psidef}
    \psi(\rho, u, y)=\Psi_{*}(y)-\frac{(u-\frac{\mu \rho^{2}}{2})^{2}}{2B}+\frac{\mu^{2} \rho^{2}}{4 B^{\prime}}+\log \rho-\frac{BB^{\prime\prime}}{BB^{\prime\prime}-B^{\prime2}}\Big(y+\frac{\mu+\frac{2B^{\prime}}{B}(u-\frac{\mu\rho^{2}}{2})}{\sqrt{8B^{\prime\prime}}}\Big)^2.
    \end{align}
\begin{lemma}
    \label{x-11}
    If $\mu>\sqrt{4B''}$, the function $\psi$ attains its unique maximum at $(\rho_*,u_*,y_*)$, where
    \begin{align}
    \label{219}
    \rho_*=\frac{\sqrt{-2 B^{\prime}}}{\mu},\quad u_{*}=\frac{B^{\prime}}{\mu},\quad y_{*}=-\frac{1}{\sqrt{2}}\left(\frac{\mu}{\sqrt{4 B^{\prime \prime}}}+\frac{\sqrt{4B^{\prime \prime}}}{\mu}\right).
    \end{align}
    Moreover, the maximum value is
    \begin{align*}
    \psi\left(\rho_{*}, u_{*}, y_{*}\right)=-\log \sqrt{4 B^{\prime \prime}}+\frac{1}{2} \log (-2B^{\prime})-\frac{1}{2}-\frac{1}{2}\log 2 .
    \end{align*}
    If $\mu\leq\sqrt{4B''}$, the function $\psi$ attains its unique maximum at $(\rho_*,u_*,y_*)$, where
    \begin{align}
    \rho_*=\frac{\sqrt{-2 B^{\prime}}}{\mu},\quad u_{*}=\frac{\mu B'}{2B''}-\frac{B^{\prime}}{\mu},\quad y_{*}=-\frac{\mu}{\sqrt{2 B^{\prime \prime}}}.
    \end{align}
    In this case, the maximum value is
    \begin{align*}
    \psi\left(\rho_{*}, u_{*}, y_{*}\right)=\frac{\mu^2}{8 B^{\prime \prime}}-1-\frac{1}{2} \log 2+\frac{1}{2} \log (-2B^{\prime})-\log \mu .
    \end{align*}
    \end{lemma}
    \begin{proof}
    Letting $v=\frac{u-\frac{\mu \rho^{2}}{2}}{\sqrt{B}}$, $J=\sqrt{4 B^{\prime \prime}}$ and $\beta=\frac{2B^{\prime}}{\sqrt{B}}$ gives
    \begin{align*}
    \psi(\rho, u, y)=\tilde\psi(\rho, v, y):=\Psi_{*}(y)-\frac{J^{2}}{J^{2}-\beta^{2}}\left(y+\frac{\mu}{\sqrt{2} J}+\frac{\beta v}{\sqrt{2} J}\right)^{2}-\frac{v^{2}}{2}+\frac{\mu^2\rho^2}{4B^{\prime}}+\log \rho.
    \end{align*}
    It is clear that $\Psi_{*}(y)-\frac{J^{2}}{J^{2}-\beta^{2}}\left(y+\frac{\mu}{\sqrt{2} J}+\frac{\beta v}{\sqrt{2} J}\right)^{2}-\frac{v^{2}}{2}$ does not depend on $\rho$ and $\rho_*=\frac{\sqrt{-2 B^{\prime}}}{\mu}$ is the maximizer of $\frac{\mu^2\rho^2}{4B^{\prime}}+\log \rho$. Following the same argument as that in \cite[Example 2]{AZ20}, we can easily get the desired result.
    \end{proof}

Hereafter we will denote by $(\rho_{*}, u_{*}, y_{*})$ the unique maximizer of $\psi$ for SRC fields. The following result is \cite[Lemma 2.2]{Be22}, which was partially credited to previous works \cite{Fy04, Fo12, Fy15}.
\begin{lemma}
\label{x-2}
(i) For any $\delta>0$,
\begin{align*}
\rho_{N}(x)=\frac{\exp (N \Phi(x))}{2 \sqrt{\pi N}\left(x^{2}-2\right)^{\frac{1}{4}}\left(|x|+\sqrt{x^{2}-2}\right)^{\frac{1}{2}+o(1)}}
\end{align*}
with the error term o(1) converging to zero uniformly for $|x|>\sqrt{2}(1+\delta)$, and where 
\begin{align}\label{eq:phi}
    \Phi(x)=\left(-\frac{|x| \sqrt{x^{2}-2}}{2}+\log \left(\frac{|x|+\sqrt{x^{2}-2}}{\sqrt{2}}\right)\right) \indi_{\{|x| \geq \sqrt{2}\}}\le0.
    \end{align}

(ii) For any $\varepsilon>0$ and large enough $N$, for all $x \in \mathbb{R}$,
\begin{align*}
e^{N \Phi(x)(1+\varepsilon)-N \varepsilon} \leq \rho_{N}(x) \leq e^{N \Phi(x)(1-\varepsilon)+N \varepsilon} .
\end{align*}
\end{lemma}
Note that
\begin{align}\label{eq:psiphi}
    \Psi_{*}(y)=\frac{1}{2}y^2-\frac{1}{2}-\frac{1}{2}\log 2 +\Phi(y).
\end{align}
Recalling Lemma \ref{x-8} and the representation \eqref{wn}, the density of $\nabla H_N(x)$ at 0, the density of $\frac1N H_N(x)$ and the density of $\eta$ are, respectively,
\begin{align}
     p_{\nabla H_{N}(x)}(0)&=\frac{1}{(2 \pi)^{N / 2} (-2B^{\prime})^{N / 2}} e^{-\frac{\mu^{2}\|x\|^{2}}{2 (-2B^{\prime})}} = \frac{1}{(2 \pi)^{N / 2} (-2B^{\prime})^{N / 2}} e^{-\frac{N \mu^{2}\rho^2}{2 (-2B^{\prime})}}, \label{eq:gradhnpdf}\\
     f_{\frac{1}{N}H_{N}(x)}(u)&=\sqrt{\frac{N}{2\pi B}}e^{\frac{-N(u-\frac{\mu\|x\|^{2}}{2N})^2}{2B}} = \sqrt{\frac{N}{2\pi B}}e^{\frac{-N(u-\frac{\mu \rho^2}{2})^2}{2B}}, \label{eq:hnpdf}\\
    g_{\eta}(t)&=\sqrt{\frac{NBB^{\prime\prime}}{\pi(BB^{\prime\prime}-B^{\prime2})}}e^{\frac{-NBB^{\prime\prime}(t-\frac{\mu+\frac{2B^{\prime}}{B}(u-\frac{\mu\rho^{2}}{2})}{\sqrt{8B^{\prime\prime}}})^2}{BB^{\prime\prime}-B^{\prime2}}}, \label{eq:etapdf}
\end{align}
where we have used the relation $\rho=\|x\|/\sqrt{N}$. Let $0\le R_1<R_2\le \8$, $E$ and $T$ be open subsets of $\rz$, and we define the integral
\begin{align}
\label{eq:jdef}
    \sfJ_N((R_1,R_2),E, T) &= \int_{R_1}^{R_2} \int_E \int_{T} \mathbb{E}[|\operatorname{det}(\GOE_{N}+t I_N)|] \\
    &\quad \times g_{\eta}(t)p_{\nabla H_{N}(x)}(0) f_{\frac{1}{N}H_{N}(x)}(u) \rho^{N-1}\dd  t \dd  u \dd  \rho \notag.
\end{align}
The following result reduces the analysis of Laplace asymptotics to the compact setting, which is commonly called exponential tightness.
For a subset $A\subset \rz$, we write $A^c$ for its complement in $\rz$.

\begin{lemma}\label{le:exptt}
    We have
    \begin{align*}
        &\limsup_{K\to\8}\limsup_{N\to\8} \frac1N \log \sfJ_N((K,\8),\rz, \rz) = -\8,\\
        &\limsup_{K\to\8}\limsup_{N\to\8} \frac1N \log \sfJ_N(\rz_+,[-K,K]^c, \rz) = -\8,\\
        &\limsup_{K\to\8}\limsup_{N\to\8} \frac1N \log \sfJ_N(\rz_+,\rz, [-K,K]^c) = -\8.
    \end{align*}
\end{lemma}

\begin{proof}
    We follow the strategy of \cite[Section 4]{AZ20}. For the first assertion, observe that
    \[
        \ez(|\eta|^N)\le C^N \Big[N^{N/2}\Big(\frac{BB''-{B'}^2}{2NBB''}\Big)^{N/2} +\Big|\frac{\mu+\frac{2B^{\prime}}{B}(u-\frac{\mu\rho^{2}}{2})}{\sqrt{8B^{\prime\prime}}}\Big|^N\Big]\le C_{\mu,B}^N\Big[1+\Big|u-\frac{\mu\rho^2}{2}\Big|^N\Big].
    \]
Here and in what follows, we use $C$ to denote an absolute constant and $C_{\mu,B}$ a constant depending only on $\mu$ and the function $B$ which may vary from line to line. Combining with \eqref{ac1}, we deduce
\begin{align*}
    &\sfJ_N((K,\8),\rz, \rz)= \int_K^\8 \int_\rz \ez[|\det(\GOE_N+\eta I_N) |]f_{\frac{1}{N}H_{N}(x)}(u)p_{\nabla H_{N}(x)}(0) \rho^{N-1}\dd  u \dd  \rho\\
    &\le C_{B}^N \int_K^\8 \int_\rz [\ez[|\la^*(\GOE_N)|^N]+\ez(|\eta|^N)] e^{-\frac{N(u-\frac{\mu \rho^2}{2})^2}{2B}} e^{-\frac{N \mu^{2}\rho^2}{2 (-2B^{\prime})}} \rho^{N-1}\dd  u \dd  \rho\\
    &\le C_{\mu,B}^N \int_K^\8 \int_\rz \Big[1+\Big|u-\frac{\mu\rho^2}{2}\Big|^N\Big] e^{-\frac{N(u-\frac{\mu \rho^2}{2})^2}{2B}} e^{-\frac{N \mu^{2}\rho^2}{2 (-2B^{\prime})}} \rho^{N-1}\dd  u \dd  \rho\\
    &\le C_{\mu,B}^N \int_K^\8 e^{-\frac{N \mu^{2}\rho^2}{2 (-2B^{\prime})}} \rho^{N-1} \dd  \rho \le  C_{\mu,B}^N e^{-\frac{N \mu^{2}K^2}{ (-8B^{\prime})}}.
\end{align*}
Similarly, for the second assertion, bounding the polynomials with exponential functions and using the change of variable $\rho=\sqrt{r},$ 
 we find
\begin{align*}
    &\sfJ_N(\rz_+,[-K,K]^c, \rz)\\&
    \le C_{\mu,B}^N   \int_0^\8 \int_{[-K,K]^c} \Big[1+\Big|u-\frac{\mu\rho^2}{2}\Big|^N\Big] e^{-\frac{N(u-\frac{\mu \rho^2}{2})^2}{2B}} e^{-\frac{N \mu^{2}\rho^2}{2 (-2B^{\prime})}} \rho^{N-1}\dd  u \dd  \rho\\
    &\le C_{\mu,B}^N\Big(\int_0^{K/\mu}+\int_{K/\mu}^\8\Big)\int_{[-K,K]^c}  e^{-\frac{N(u-\frac{\mu r}{2})^2}{4B}}e^{-\frac{N \mu^{2}r}{4 (-2B^{\prime})}}  \dd u \dd r\\
    &\le C_{\mu,B}^N \int_0^{K/\mu}\int_{[-K,K]^c}   e^{-\frac{N(|u|-\frac{ K}{2})^2}{4B}}e^{-\frac{N \mu^{2}r}{4 (-2B^{\prime})}}  \dd u \dd r +  C_{\mu,B}^N \int_{K/\mu}^\8 \int_{\rz} e^{-\frac{N(u-\frac{\mu r}{2})^2}{4B}}e^{-\frac{N \mu^{2}r}{4 (-2B^{\prime})}}  \dd u \dd r\\
    &\le    C_{\mu,B}^N (e^{-\frac{NK^2}{20B}} + e^{-\frac{N\mu K}{-8B'}}).
\end{align*}
For the third assertion, note that
\begin{align*}
    &\sfJ_N(\rz_+, \rz,[-K,K]^c)\le \sfJ_N((R,\8), \rz,\rz) +\sfJ_N(\rz_+, [-R,R]^c,\rz)\\&\quad\quad\quad\quad\quad\quad\quad\quad\quad\quad\quad\quad+ \sfJ_N((0,R),(-R,R),[-K,K]^c).
\end{align*}
The first two assertions imply that the first two terms on the right-hand side of the preceding display are exponentially negligible by choosing $R$ large enough. Using Lemmas \ref{x-9} and \ref{x-2} and choosing $K$ large
\begin{align*}
    &\sfJ_N((0,R),(-R,R),[-K,K]^c) \\
    &=\int_{0}^R \int_{-R}^R \int_{[-K,K]^c} \mathbb{E}[|\operatorname{det}(\GOE_{N}+t I_N)|] g_{\eta}(t)p_{\nabla H_{N}(x)}(0) f_{\frac{1}{N}H_{N}(x)}(u) \rho^{N-1}\dd  t \dd  u \dd  \rho\\
    &\le C_{\mu,B}^N \int_{0}^R \int_{-R}^R \int_{[-K,K]^c} e^{\frac{-NBB^{\prime\prime}(t-\frac{\mu+\frac{2B^{\prime}}{B}(u-\frac{\mu\rho^{2}}{2})}{\sqrt{8B^{\prime\prime}}})^2}{BB^{\prime\prime}-B^{\prime2}}}   e^{N t^{2} / 2} \rho_{N+1}(\sqrt {\frac{N}{N+1}}t) e^{N[\frac{-(u-\frac{\mu\rho^2}{2})^2}{2B}+\frac{\mu^2\rho^2}{4B^{\prime}}]} \\
    &\quad \times \rho^{N-1}\dd  t \dd  u \dd  \rho\\
    &\le C_{\mu,B,R}^N \int_{[-K,K]^c}  e^{\frac{-N(BB^{\prime\prime}+B^{\prime^2})t^2}{2(BB^{\prime\prime}-B^{\prime^2})}} e^{\frac{N\sqrt{B^{\prime\prime}}[\mu B+2|B^{\prime}|(R+\frac{\mu R^{2}}{2})]|t|}{\sqrt{2}(BB^{\prime\prime}-B^{\prime2})}} \dd t\\
    &\le C_{\mu,B,R}^N \int_{[-K,K]^c}  e^{\frac{-N(BB^{\prime\prime}+B^{\prime^2})t^2}{4(BB^{\prime\prime}-B^{\prime^2})}}  \dd t\le C_{\mu,B,R}^N e^{\frac{-N(BB^{\prime\prime}+B^{\prime^2})K^2}{4(BB^{\prime\prime}-B^{\prime^2})}}.
\end{align*}
The proof is complete.
\end{proof}

The next technical lemma is the cornerstone of various asymptotic analysis in this section.

\begin{lemma}\label{le:stech}
    Recall the notations \pref{eq:gradhnpdf}--\pref{eq:jdef}. Let $0\le R_1<R_2\le \8$, $E$ and $T$ be open subsets of $\rz$. Then
    \begin{align*}
        &\lim_{N\to\8}\frac1N\log \sfJ_N((R_1,R_2),E, -T)= -\frac12\log(2\pi) -\frac12\log(-2B') +\sup_{(\rho,u,y)\in F} \psi(\rho,u,y),
    \end{align*}
    where  $F=\left\{(\rho, u, y): \rho \in\left(R_{1}, R_{2}\right), u \in E,  y \in T\right\}$ and $\psi$ is given as in \pref{eq:psidef}. 
\end{lemma}

\begin{proof}
    Using Lemma \ref{x-9} with the change of variable $y=-\sqrt {\frac{N}{N+1}}t$, we have
    \begin{align*}
        &\int_{R_1}^{R_2} \int_E \int_{-T} \mathbb{E}[|\operatorname{det}(\GOE_{N}+t I_N)|] g_{\eta}(t)p_{\nabla H_{N}(x)}(0) f_{\frac{1}{N}H_{N}(x)}(u) \rho^{N-1}\dd  t \dd  u \dd  \rho\\
        &=\sqrt{\frac{N}{2\pi B}}\frac{\sqrt{2(N+1)} N^{-N / 2} \Gamma\left(\frac{N+1}{2}\right)}{(2 \pi)^{N / 2} (-2B^{\prime})^{N / 2}}\sqrt{\frac{NBB^{\prime\prime}}{\pi(BB^{\prime\prime}-B^{\prime2})}}\int_{R_{1}}^{R_{2}} \int_{E} \int_{-T} e^{N t^{2} / 2} \\
        &\quad \times e^{\frac{-NBB^{\prime\prime}(t-\frac{\mu+\frac{2B^{\prime}}{B}(u-\frac{\mu\rho^{2}}{2})}{\sqrt{8B^{\prime\prime}}})^2}{BB^{\prime\prime}-B^{\prime2}}}   \rho_{N+1}(\sqrt {\frac{N}{N+1}}t) e^{N[\frac{-(u-\frac{\mu\rho^2}{2})^2}{2B}+\frac{\mu^2\rho^2}{4B^{\prime}}]} \rho^{N-1}\dd  t \dd  u \dd  \rho\\
        &= \frac{(N+1)  N^{-\frac{N-1}{2}}\Gamma\left(\frac{N+1}{2}\right)\sqrt{4B^{\prime\prime}}}{(2\pi)^{\frac{N}2+1} \sqrt{BB^{\prime\prime}-B^{\prime2}}(-2B^{\prime})^{N/2}}  \int_{R_{1}}^{R_{2}} \int_{E} \int_{\sqrt{\frac{N}{N+1}}T} e^{\frac{-(N+1)(BB^{\prime\prime}+B^{\prime^2})y^2}{2(BB^{\prime\prime}-B^{\prime^2})}}\nonumber \\
        &\quad\times e^{-\frac{\sqrt{N(N+1)}\sqrt{B^{\prime\prime}}[\mu B+2B^{\prime}(u-\frac{\mu\rho^{2}}{2})]y}{\sqrt{2}(BB^{\prime\prime}-B^{\prime2})}} e^{N[\frac{-(u-\frac{\mu\rho^2}{2})^2}{2B}-\frac{[\mu B+2B^{\prime}(u-\frac{\mu\rho^{2}}{2})]^{2}}{8B(BB^{\prime\prime}-B^{\prime2})}+\frac{\mu^2\rho^2}{4B^{\prime}}]} \rho_{N+1}(y) \rho^{N-1}\dd  y \dd  u \dd  \rho\\
        &= C_N \sfI_N\Big((R_1,R_2),E,\sqrt{\frac{N}{N+1}}T\Big),
    \end{align*}
    where $C_N=\frac{(N+1)  N^{-\frac{N-1}{2}}\Gamma\left(\frac{N+1}{2}\right)\sqrt{4B^{\prime\prime}}}{(2\pi)^{\frac{N}2+1} \sqrt{BB^{\prime\prime}-B^{\prime2}}(-2B^{\prime})^{N/2}}$ and $\sfI_N((R_1,R_2),E,\sqrt{\frac{N}{N+1}}T)$ is the remaining integral.  The Stirling formula implies that $\lim_{N\to\8} \frac1N\log[\Gamma(\frac{N+1}{2})/N^{\frac{N-1}{2}}]= -\frac12\log2-\frac12$ and thus
    \[
    \lim_{N\to\8} \frac1N\log C_N=     -\frac12\log2-\frac12 -\frac12\log(2\pi) -\frac12\log(-2B').
    \]
    It remains to show that
    \begin{align}\label{eq:inlim}
        \lim_{N\to\8} \frac1N\log \sfI_N\Big((R_1,R_2),E,\sqrt{\frac{N}{N+1}}T\Big)=\frac12\log2+\frac12 +\sup_{(\rho,u,y)\in F} \psi(\rho,u,y).
    \end{align}
 Thanks to Lemma \ref{x-11}, we see that $\sup_{(\rho,u,y)\in F} \psi(\rho,u,y) \le \psi(\rho_*,u_*,y_*)<\8$. By continuity, $\psi$ attains its maximum on $\bar F$, the closure of $F$ in the Euclidean topology. Together with a density argument, we may find a maximizer $(\rho^*,u^*,y^*)\in \bar F$ so that $\psi(\rho^*,u^*,y^*)=\sup_{(\rho,u,y)\in F} \psi(\rho,u,y)$. Note that the two points  $(\rho_*,u_*,y_*)$ and $(\rho^*,u^*,y^*)$ may not be equal. Thus we use different notations here to distinguish them. Since $E$ and $T$ are open, for any $\de>0$, we have for $N$ large enough
    \[
    (R_1,R_2)\cap U_\de(\rho^*) \neq \emptyset, \ \     E\cap U_\de(u^*)\neq \emptyset , \ \ \sqrt{\frac{N}{N+1}} T\cap U_\de(y^*)\neq \emptyset,
    \]
    where $U_\de(x)=(x-\de,x+\de)$. Taking inf of the integrand of the integral
    $$ \sfI_N\Big((R_1,R_2)\cap U_\de(\rho^*),E \cap U_\de(u^*),\sqrt{\frac{N}{N+1}}T\cap U_\de(y^*)\Big),$$
    using Lemma \ref{x-2}, we find
    \begin{align*}
        &\liminf_{N\to\8} \frac1N\log \sfI_N\Big((R_1,R_2),E,\sqrt{\frac{N}{N+1}}T\Big)\\
        &\ge \liminf_{N\to\8} \frac1N\log \sfI_N\Big((R_1,R_2)\cap U_\de(\rho^*),E \cap U_\de(u^*),\sqrt{\frac{N}{N+1}}T\cap U_\de(y^*)\Big)\\
        &\ge \inf_{y\in U_\de (y^*)} \Big[- \frac{(BB^{\prime\prime}+B^{\prime^2})y^2}{2(BB^{\prime\prime}-B^{\prime^2})}+\Phi(y)(1+\eps)\Big]-\eps  \\
    &\quad-\sup_{y\in U_\de(y^*), u\in U_\de(u^*)} \frac{\sqrt{B^{\prime\prime}}[\mu B+2B^{\prime}(u-\frac{\mu\rho^{2}}{2})]y}{\sqrt{2}(BB^{\prime\prime}-B^{\prime2})}+\inf_{\rho\in (R_1,R_2)\cap U_\de(\rho^*)}\log \rho \\
        &\quad - \sup_{u\in U_\de(u^*),\rho\in (R_1,R_2)\cap U_\de(\rho^*)}\Big[ \frac{(u-\frac{\mu\rho^2}{2})^2}{2B}+\frac{[\mu B+2B^{\prime}(u-\frac{\mu\rho^{2}}{2})]^{2}}{8B(BB^{\prime\prime}-B^{\prime2})}-\frac{\mu^2\rho^2}{4B^{\prime}}\Big].
    \end{align*}
    Sending $\eps\downarrow 0$ and $\de\downarrow 0$, the lower bound of \pref{eq:inlim} follows from continuity, \pref{eq:psidef} and \pref{eq:psiphi}.

    For the upper bound, due to \pref{le:exptt} we may and do assume $\bar E$ and $\bar T$ are compact and $R_2<\8$. By Lemma \ref{x-2} and continuity, we have
    \begin{align*}
        &\limsup_{N\to\8} \frac1N\log \sfI_N\Big((R_1,R_2),E,\sqrt{\frac{N}{N+1}}T\Big) \\
        &\le \limsup_{N\to\8} \frac1N \sup_{(\rho,u,y)\in F} \Big[ \frac{-(N+1)(BB^{\prime\prime}+B^{\prime^2})y^2}{2(BB^{\prime\prime}-B^{\prime^2})}\\&\quad-\frac{\sqrt{N(N+1)}\sqrt{B^{\prime\prime}}[\mu B+2B^{\prime}(u-\frac{\mu\rho^{2}}{2})]y}{\sqrt{2}(BB^{\prime\prime}-B^{\prime2})}\\
        &\quad +N[\frac{-(u-\frac{\mu\rho^2}{2})^2}{2B}-\frac{[\mu B+2B^{\prime}(u-\frac{\mu\rho^{2}}{2})]^{2}}{8B(BB^{\prime\prime}-B^{\prime2})}+\frac{\mu^2\rho^2}{4B^{\prime}}]\\&\quad+(N-1) \log \rho  +N\Phi(y)(1-\eps)+N\eps\Big]\\
        &\le \sup_{(\rho,u,y)\in F} \Big[ \frac{- (BB^{\prime\prime}+B^{\prime^2})y^2}{2(BB^{\prime\prime}-B^{\prime^2})}-\frac{\sqrt{B^{\prime\prime}}[\mu B+2B^{\prime}(u-\frac{\mu\rho^{2}}{2})]y}{\sqrt{2}(BB^{\prime\prime}-B^{\prime2})}\\
        &\quad -\frac{(u-\frac{\mu\rho^2}{2})^2}{2B}-\frac{[\mu B+2B^{\prime}(u-\frac{\mu\rho^{2}}{2})]^{2}}{8B(BB^{\prime\prime}-B^{\prime2})}+\frac{\mu^2\rho^2}{4B^{\prime}} +\log \rho +\Phi(y)(1-\eps)\Big]+\eps.
    \end{align*}
    Sending $\eps\downarrow0$, combining \pref{eq:psiphi} with the definition \pref{eq:psidef}, we have proved the upper bound of \pref{eq:inlim}.
\end{proof}
In the following, we consider a shell domain $B_N(R_1,R_2)=\{x\in\rz^N: R_1< \|x\|/\sqrt{N}< R_2\}$ and write $\Crt_N(E,(R_1,R_2))=\Crt_N(E,B_N(R_1,R_2)).$ Since $\|x\|=\sqrt{N}R$ has Lebesgue measure zero in $\rz^N$ for $0<R<\8$, and we will only consider $\ez[\Crt_N(E,(R_1,R_2))]$, we may alternatively define $B_N(R_1,R_2)$ as a closed subset of $\rz^N$. Consequently, by abuse of notation, from time to time we also use $\Crt_N(E,(R_1,R_2))$ to denote the number of critical points when (part of) the boundary of $B_N(R_1,R_2)$ is included. 

\begin{proposition}
\label{x-10}
Let $0 \leq R_{1}<R_{2} \leq \infty$ and $E$ be an open set of $\mathbb{R}$. Assume Assumption I. 
Then
\begin{align}\label{eq:cpx1}
\lim _{N \rightarrow \infty} \frac{1}{N} \log \mathbb{E} [\operatorname{Crt}_{N}\left(E, \left(R_{1}, R_{2}\right)\right)]=\frac{1}{2} \log \left(8 B^{\prime \prime}\right)-\frac{1}{2} \log \left(-2B^{\prime}\right)+\frac{1}{2}+\sup _{(\rho, u, y) \in F} \psi(\rho, u, y),
\end{align}
where $F=\left\{(\rho, u, y): y \in \mathbb{R}, \rho \in\left(R_{1}, R_{2}\right), u \in E\right\}$, and the function $\psi$ is given as in \pref{eq:psidef}.

\end{proposition}
\begin{proof}
 Using the spherical coordinates with $\rho=\frac{\|x\|}{\sqrt{N}}$, by the Kac--Rice formula,
\begin{align}\label{eq:kr1}
&\mathbb{E} [\operatorname{Crt}_{N}\left(E,\left(R_{1}, R_{2}\right)\right)]=\int_{B_{N}} \int_{E} \mathbb{E}\Big[|\operatorname{det}\nabla^{2} H_{N}(x) | \Big| \frac{H_{N}(x)}{N}=u\Big]\\
&\quad\quad\quad\quad\quad\quad\quad\quad\quad\quad\quad\quad\times f_{\frac{1}{N}H_{N}(x)}(u) p_{\nabla H_{N}(x)}(0) \dd  u \dd  x\nonumber \\
&=S_{N-1} N^{(N-1) / 2} \int_{R_{1}}^{R_{2}} \int_{E} \mathbb{E}[|\operatorname{det} W^{N}|] f_{\frac{1}{N}H_{N}(x)}(u) p_{\nabla H_{N}(x)}(0) \rho^{N-1} \dd  u \dd  \rho \nonumber.
\end{align}
Here $S_{N-1}=\frac{2 \pi^{N / 2}}{\Gamma(N / 2)}$ is the area of the $N-1$ dimensional unit sphere, $W^N, f_{\frac{1}{N}H_{N}(x)}(u)$ and $p_{\nabla H_{N}(x)}(0)$ are given as in (\ref{wn}), \pref{eq:hnpdf} and \pref{eq:gradhnpdf}, respectively. Conditioning on $\eta=t$, using Lemma \ref{x-9} and the change of variable $y=-\sqrt {\frac{N}{N+1}}t$, we deduce that
\begin{align}
\label{pp3}
\mathbb{E} [\operatorname{Crt}_{N}\left(E,\left(R_{1}, R_{2}\right)\right)]
&= S_{N-1}N^{(N-1)/2}(8B^{\prime\prime})^{N/2}\int_{R_{1}}^{R_{2}} \int_{E} \int_{\mathbb{R}} \mathbb{E}[|\operatorname{det}(\GOE_{N}+t I_N)|]\\
&\qquad \times g_{\eta}(t)p_{\nabla H_{N}(x)}(0) f_{\frac{1}{N}H_{N}(x)}(u) \rho^{N-1}\dd  t \dd  u \dd  \rho\notag  \notag\\
&=S_{N-1}N^{(N-1)/2}(8B^{\prime\prime})^{N/2} \sfJ_N((R_1,R_2),E,\rz) \notag,
\end{align}
where $g_{\eta}(t)$ is given as in \pref{eq:etapdf}, and $\sfJ_N$ is defined in \pref{eq:jdef}. Since
\begin{align}\label{eq:snlim}
    \lim_{N\to\8} \frac1N \log (S_{N-1}N^{\frac{N-1}{2}})= \frac12\log(2\pi)+\frac12,
\end{align}
applying \pref{le:stech} completes the proof.
\end{proof}

We remark that as observed before $\sup _{(\rho, u, y) \in F} \psi(\rho, u, y) = \sup _{(\rho, u, y) \in \bar F} \psi(\rho, u, y)$ by density and continuity, where $\bar F$ is the closure of $F$ in the Euclidean topology.

\begin{lemma}
\label{x-12}
Assume Assumption I and $\mu>\sqrt{4 B^{\prime \prime}}$. Then for any $0\le R_{1} <R_{2}\le\8$, and open set $E\subset \mathbb{R}$,  if $\left(\rho_{*}, u_{*}\right) \notin  \overline{(R_1,R_2)} \times \bar E$, we have
\begin{align*}
\lim _{N \rightarrow \infty}  \mathbb{E} [\operatorname{Crt}_{N}\left(E, B_{N}\left(R_{1}, R_{2}\right)\right)]=0.\end{align*}
\end{lemma}

\begin{proof}
By Lemma \ref{x-11}, if $\left(\rho_{*}, u_{*}\right) \notin  \overline{(R_1,R_2)} \times \bar E$, letting $F=\overline{(R_1,R_2)} \times \bar E\times \rz$, the value of $\sup _{(\rho, u, y) \in F} \psi(\rho, u, y)$ is  strictly smaller than $\psi\left(\rho_{*}, u_{*}, y_{*}\right)=-\log \sqrt{4 B^{\prime \prime}}+\frac{1}{2} \log (-2B^{\prime})-\frac{1}{2}-\frac{1}{2}\log 2$ since $\overline{(R_1,R_2)} \times \bar E$ is a closed set. Therefore, the right-hand side of \pref{eq:cpx1} is strictly negative, which directly yields the assertion. 
\end{proof}

We are now ready to prove the main results of this section.
\begin{proof}[Proof of \eqref{xx}--\eqref{qq10}]
Since Proposition \ref{z-1} shows that this model has at least  one global minimum in $\rz^N$ almost surely, (\ref{xx}) follows from (\ref{pp1}) and Markov's inequality.

For any $\varepsilon>0$, observe that
$$(\rho_{*}, u_{*})\notin [0,\8)\times (u_{*}-\varepsilon,u_*+\eps)^c, \quad (\rho_{*}, u_{*})\notin [\rz_+\setminus (\rho_*-\eps,\rho_*+\eps) ]\times \rz.$$
Using Markov's inequality and Lemma \ref{x-12}, we find 
\begin{align*}
&\mathbb{P}\left(\left|\frac{1}{N} H_{N}\left(x^{*}\right)-u_{*}\right|>\varepsilon\right)=\mathbb{P}\left(\frac{1}{N} H_{N}\left(x^{*}\right) \in [u_{*}-\varepsilon,u_*+\eps]^c \right) \\
&\leq\mathbb{E} [\operatorname{Crt}_{N}\left([u_{*}-\varepsilon,u_*+\eps]^c, \left(0, \infty\right)\right)] \stackrel{N\to\8}\longrightarrow 0,\\
&\mathbb{P}\left(\left|\frac{\|x^{*}\|}{\sqrt{N}}-\rho_{*}\right|> \varepsilon\right) =\mathbb{P}\left(\frac{\|x^{*}\|}{\sqrt{N}}\in  [\rho_*-\eps, \rho_*+\eps]^c \right) \\
&\leq\mathbb{E} [\operatorname{Crt}_{N}\left(\rz,  (0,\rho_*-\eps)  \right)]+\mathbb{E} [\operatorname{Crt}_{N}\left(\rz,  (\rho_*+\eps,\8)]  \right)]
\stackrel{N\to\8}\longrightarrow 0.
\end{align*}
We have proved \eqref{hh3} and (\ref{q10}).


Let us now turn to (\ref{qq10}). Define
$$
 \Lambda_{\varepsilon}=\left\{x \in \mathbb{R}^{N}:L_{N}\left(\nabla^{2}H_{N}(x)\right)\notin B\left(\si_{c_s,r_s}, \varepsilon\right)\right\}, $$
where we recall $c_s=\mu+\frac{4B''}{\mu}$ and $r_s=4\sqrt{B''}$.
Arguing as in \pref{eq:kr1} and using the same notations as there, we find 
\begin{align*}
&\mathbb{E}\left[\#\left\{x \in  \Lambda_{\varepsilon}: \nabla H_{N}(x)=0\right\}\right]
=\int_{\mathbb{R}^{N}} p_{\nabla H_{N}(x)}(0) \mathbb{E}\left[\left|\operatorname{det}\left(\nabla^{2} H_{N}(x)\right)\right| \indi_{\Lambda_{\varepsilon}}(x)\right]{\dd} x\nonumber\\
&=S_{N-1} N^{(N-1) / 2} \int_{0}^{\infty} \int_{-\infty}^{\infty} \mathbb{E}[|\operatorname{det} W^{N}|\indi_{\left\{L_{N}\left(W^{N}\right)\notin B\left(\si_{c_s,r_s}, \varepsilon\right)\right\}}]\nonumber\\
&\quad\quad\quad\quad\quad\quad\quad\quad\quad\quad\quad\quad\times p_{\nabla H_{N}(x)}(0) f_{\frac{1}{N}H_{N}(x)}(u) \rho^{N-1} \dd  u \dd  \rho.
\end{align*}
Recalling $W^{N}\stackrel{d}{=}\sqrt{8B^{\prime\prime}}(\GOE_{N}+\eta I_{N})$ and conditioning on $\eta=t$, rescaling the semicircle measure yields
\begin{align*}
    &\ez[\#\{ x \in  \Lambda_{\varepsilon}: \nabla H_{N}(x)=0\}] \\
    &= S_{N-1} N^{(N-1) / 2} (8B^{\prime\prime})^{N/2} \int_{0}^{\infty} \int_{-\infty}^{\infty} \int_{-\8}^\8 g_{\eta}(t) f_{\frac{1}{N}H_{N}(x)}(u) p_{\nabla H_{N}(x)}(0) \rho^{N-1}\\
    &\qquad \times\ez[|\det(\GOE_N+ tI_N) \indi_{\{ d(L_N(\GOE_N+ tI_N), ~\si_{c_s',r_s'})>\eps' \}} |]  \dd  t \dd  u \dd  \rho \\
    &=J_{N}^{1}+J_{N}^{2},
\end{align*}
where the metric $d$ is the bounded Lipschitz metric given as in \pref{eq:measd}, $g_\eta(t)$ is given as in \pref{eq:etapdf}, $\si_{c_s^{\prime},r_s^{\prime}}$ is  the semicircle measure with center $c_s^{\prime}=\frac{1}{\sqrt{2}}\left(\frac{\mu}{\sqrt{4 B^{\prime \prime}}}+\frac{\sqrt{4B^{\prime \prime}}}{\mu}\right)=-y_*$ and radius $r_s^{\prime}=\sqrt{2}$, $\eps'=\frac{\eps}{\sqrt{8B''}}$ and
\begin{align}
\label{qq2}
J_{N}^{1}&=S_{N-1} N^{(N-1)/2}(8B^{\prime\prime})^{N/2}\int_{0}^{\infty} \int_{-\infty}^{\infty}\int_{y_{*}-\frac{\varepsilon'}{2}}^{y_{*}+\frac{\varepsilon'}{2}} g_{\eta}(t)p_{\nabla H_{N}(x)}(0) f_{\frac{1}{N}H_{N}(x)}(u) \rho^{N-1}\\
 &\quad\times\mathbb{E}[|\operatorname{det}(\GOE_{N}+t I_{N})|\indi_{\{d(L_N(\GOE_N+ tI_N), ~\si_{c_s',r_s'})>\eps'\}}]\dd  t \dd  u \dd  \rho,\nonumber\\
J_{N}^{2}&=S_{N-1}N^{(N-1)/2}(8B^{\prime\prime})^{N/2}\int_{0}^{\infty}\int_{-\infty}^{\infty}
\int_{[y_{*}-\frac{\varepsilon'}{2}, y_{*}+\frac{\varepsilon'}{2}]^c} g_{\eta}(t)p_{\nabla H_{N}(x)}(0)f_{\frac{1}{N}H_{N}(x)}(u) \rho^{N-1}\nonumber\\
 &\quad\times\mathbb{E}[|\operatorname{det}(\GOE_{N}+t I_{N})|\indi_{\{d(L_N(\GOE_N+ tI_N), ~\si_{c_s',r_s'})>\eps'\}}] \dd  t \dd  u \dd  \rho \notag.
\end{align}
We claim that both $J_{N}^{1}$ and $J_{N}^{2}$ converge to 0 as $N$ goes to $\infty$, which together with Markov's inequality results in
\begin{align}
    \label{qz8}
    \mathbb{P}\left(x^{*} \in  \Lambda_{\varepsilon}\right) \leq \mathbb{E}\left[\# \left\{x \in  \Lambda_{\varepsilon}: \nabla H_{N}(x)=0\right\} \right] \stackrel{N \rightarrow \infty}{\longrightarrow} 0,
\end{align}
and \eqref{qq10} follows.

To prove the claim, we first observe that
\begin{align}
\label{mn}
J_{N}^{2}&\leq S_{N-1}N^{(N-1)/2}(8B^{\prime\prime})^{N/2}\int_{0}^{\infty}\int_{-\infty}^{\infty}
\int_{[y_{*}-\frac{\varepsilon'}{2}, y_{*}+\frac{\varepsilon'}{2}]^c} \mathbb{E}[|\operatorname{det}(\GOE_{N}+t I_{N})|]\\
 &\quad\times g_{\eta}(t)p_{\nabla H_{N}(x)}(0) f_{\frac{1}{N}H_{N}(x)}(u) \rho^{N-1}\dd  t \dd  u \dd  \rho \notag.
\end{align}
By Lemma \ref{le:stech} and \pref{eq:snlim}, we get
\begin{align}
\label{pp4}
\limsup_{N \rightarrow \infty} \frac{1}{N} \log J_{N}^{2}\leq\frac{1}{2} \log \left(8 B^{\prime \prime}\right)-\frac{1}{2} \log \left(-2B^{\prime}\right)+\frac{1}{2}+\sup _{(\rho, u, y) \in F} \psi(\rho, u, y),
\end{align}
where $F=\left\{(\rho, u, y): y \in \left[y_*-\frac{\varepsilon'}{2},y_*+\frac{\varepsilon'}{2}\right]^{c}, \rho \in\left(0, \infty\right), u \in \mathbb{R}\right\}$.
Then Lemma \ref{x-11} implies that the value of $\sup _{(\rho, u, y) \in F} \psi(\rho, u, y)$ is  strictly smaller than
$$\psi\left(\rho_{*}, u_{*}, y_{*}\right)=-\log \sqrt{4 B^{\prime \prime}}+\frac{1}{2} \log (-2B^{\prime})-\frac{1}{2}-\frac{1}{2}\log 2.$$
Combining this with (\ref{pp4}) yields that
$
\limsup_{N \rightarrow \infty} \frac{1}{N} \log J_{N}^{2}<0.
$
It follows that
\begin{align}
\label{hh2}
\lim _{N \rightarrow \infty} J_{N}^{2}=0.
\end{align}

Let us turn to $J_{N}^{1}$. Note that for $|t-y_*|\leq\frac{\varepsilon'}{2}$ and a function $h$ with $\|h\|_{L} \leq 1$,
\begin{align*}
\Big|\frac{1}{N} \sum_{i=1}^{N} h(\lambda_{i}(\GOE_{N})+t-y_*)-\frac{1}{N} \sum_{i=1}^{N} h(\lambda_{i}(\GOE_{N}))\Big|\leq\frac{\varepsilon'}{2}.
\end{align*}
By the triangle inequality,
\begin{align*}
    \eps'&< d(L_N(\GOE_N+ tI_N), \si_{c_s',r_s'}) = d(L_N(\GOE_N+ (t-y_*) I_N), \si_{c_s'+y_*,r_s'})\\
    &\le d(L_N(\GOE_N+ (t-y_*) I_N), L_N(\GOE_N)) +d(L_N(\GOE_N), \si_{\rm sc}),
\end{align*}
which implies $d(L_N(\GOE_N), \si_{\rm sc})>\frac{\eps'}2$. Combining these with the Cauchy--Schwarz inequality, we deduce
\begin{align}
\label{qq7}
J_{N}^{1}&\leq S_{N-1} N^{(N-1)/2}(8B^{\prime\prime})^{N/2}\int_{0}^{\infty} \int_{-\infty}^{\infty}\int_{y_*-\frac{\varepsilon'}{2}}^{y_*+\frac{\varepsilon'}{2}}g_{\eta}(t)p_{\nabla H_{N}(x)}(0) f_{\frac{1}{N}H_{N}(x)}(u) \rho^{N-1}\\
 &\quad\times\mathbb{E}[|\operatorname{det}(\GOE_{N}+t I_{N})|\indi_{\left\{(L_{N}\left(\GOE_{N}\right)\notin B\left(\si_{\rm{sc}}, \frac{\varepsilon'}{2}\right)\right\}}]\dd  t \dd  u \dd  \rho
 \nonumber\\
&\leq S_{N-1} N^{(N-1)/2}(8B^{\prime\prime})^{N/2}\int_{0}^{\infty} \int_{-\infty}^{\infty}\int_{y_*-\frac{\varepsilon'}{2}}^{y_*+\frac{\varepsilon'}{2}}g_{\eta}(t)p_{\nabla H_{N}(x)}(0) f_{\frac{1}{N}H_{N}(x)}(u) \rho^{N-1}\nonumber\\
 &\quad\times\mathbb{E}\left[|\operatorname{det}(\GOE_{N}+t I_{N})|^2\right]^{\frac{1}{2}}\mathbb{P}\left(L_{N}\left(\GOE_{N}\right)\notin B\left(\si_{\rm{sc}}, \frac{\varepsilon'}{2}\right)\right)^{\frac{1}{2}}\dd  t \dd  u \dd  \rho \notag.
\end{align}
Using (\ref{ac1}) and  (\ref{ac2}), and recalling the value of $y_*$ as in \eqref{219}, we have for large $N$
\begin{align}
\label{ac5}
&\mathbb{E}\left[|\operatorname{det}(\GOE_{N}+t I_{N})|^2\right]^{\frac{1}{2}}\mathbb{P}\left(L_{N}\left(\GOE_{N}\right)\notin B\left(\si_{\rm{sc}}, \frac{\varepsilon'}{2}\right)\right)^{\frac{1}{2}}\\
&\leq 2^{N}\mathbb{E}\left[\lambda^{*}\left(\GOE_N\right)^{2N}+t^{2N}\right]^{\frac{1}{2}}e^{-\frac{1}{2}c N^{2}}\nonumber\\
&\leq C_{\mu, B}^Ne^{-\frac{1}{2}c N^{2}} \notag.
\end{align}
Plugging this into (\ref{qq7}) and recalling \pref{eq:gradhnpdf}--\pref{eq:etapdf}, when $N$ is large, we find
\begin{align}
\label{ac3}
J_{N}^{1}&\leq S_{N-1} N^{(N-1)/2}(8B^{\prime\prime})^{N/2}C_{\mu, B}^Ne^{-\frac{1}{2}c N^{2}}\int_{0}^{\infty} \int_{-\infty}^{\infty}\int_{y_*-\frac{\varepsilon'}{2}}^{y_*+\frac{\varepsilon'}{2}}\sqrt{\frac{N}{2\pi B}}e^{\frac{-N(u-\frac{\mu\rho^{2}}{2})^2}{2B}} \rho^{N-1}\\
 &\quad\times\sqrt{\frac{NBB^{\prime\prime}}{\pi(BB^{\prime\prime}-B^{\prime2})}}e^{\frac{-NBB^{\prime\prime}(t-\frac{\mu+\frac{2B^{\prime}}{B}(u-\frac{\mu\rho^{2}}{2})}{\sqrt{8B^{\prime\prime}}})^2}{BB^{\prime\prime}-B^{\prime2}}}\frac{1}{(2 \pi)^{N/2} (-2B^{\prime})^{N / 2}} e^{-\frac{N\mu^{2}\rho^{2}}{2 (-2B^{\prime})}} \dd  t \dd  u \dd  \rho\nonumber\\
&\leq S_{N-1} N^{(N-1)/2} C_{\mu, B}^Ne^{-\frac{1}{2}c N^{2}}\int_{0}^{\infty} \int_{-\infty}^{\infty} e^{\frac{-N(u-\frac{\mu\rho^{2}}{2})^2}{2B}} \rho^{N-1} e^{-\frac{N\mu^{2}\rho^{2}}{2 (-2B^{\prime})}}  \dd  u \dd  \rho\nonumber\\
&\le S_{N-1} N^{(N-1)/2} C_{\mu, B}^Ne^{-\frac{1}{2}c N^{2}} \int_{0}^{\infty}   e^{-\frac{N\mu^{2}\rho^{2}}{ -8 B^{\prime}}}  \dd  \rho  \notag,
\end{align}
where the second inequality follows from the fact $e^{\frac{-NBB^{\prime\prime}(t-\frac{\mu+\frac{2B^{\prime}}{B}(u-\frac{\mu\rho^{2}}{2})}{\sqrt{8B^{\prime\prime}}})^2}{BB^{\prime\prime}-B^{\prime2}}}\leq 1$.
Recalling \pref{eq:snlim}, we arrive at
$$
\limsup_{N \rightarrow \infty} \frac{1}{N} \log J_{N}^{1}=-\infty,
$$
which proves that $\lim_{N\to \8}J_{N}^{1}=0$. 
\end{proof}

To prove (\ref{xx8}), we need the following result due to Belius, \v{C}ern\'{y}, Nakajima and Schmidt.
\begin{lemma}[{\cite[Lemma 6.4]{Be22}}]
\label{x-13}
For all $\delta>0$ we have for $N$ large enough and $\sqrt{2}+\delta \leq x \leq \delta^{-1}$,
\begin{align*}
\mathbb{E}\left[\left|\operatorname{det}\left(x  I_{N}+\GOE_N\right)\right|^{2}\right] \leq e^{\delta N}\mathbb{E}\left[\left|\operatorname{det}\left(x  I_{N}+\GOE_N\right)\right|\right]^{2} .
\end{align*}
\end{lemma}
\begin{proof}[Proof of (\ref{xx8})]
Recall that $c_s=\mu+\frac{4B''}{\mu}$ and $r_s=4\sqrt{B''}$. For any $\varepsilon>0$, we define
$$
\Omega_{\varepsilon}=\left\{x \in \mathbb{R}^{N}:\left|\lambda_{\min }\left(\nabla^{2} H_{N}(x)\right)+r_s-c_s\right|>\sqrt{8B^{\prime\prime}}\varepsilon\right\},
$$
which is a random subset of $\mathbb{R}^{N}$.
Arguing as in \pref{eq:kr1} and \eqref{pp3}, and using the same notations as there, we get
\begin{align}
&\mathbb{E}\left[\#\left\{x \in \Omega_{\varepsilon}: \nabla H_{N}(x)=0\right\} \right]\label{xx6} \\
&=S_{N-1} N^{(N-1) / 2} \int_{0}^{\infty} \int_{-\infty}^{\infty} \mathbb{E}[|\operatorname{det} W^{N}|\indi_{\left\{\left|\lambda_{\min }\left(W^{N}\right)+r_s-c_s\right|>\sqrt{8B^{\prime\prime}}\varepsilon\right\}}]\nonumber \\
&\quad\quad\quad\quad\quad\quad\quad\quad\quad \times p_{\nabla H_{N}(x)}(0) f_{\frac{1}{N}H_{N}(x)}(u) \rho^{N-1} \dd  u \dd  \rho\notag\\
&=S_{N-1} N^{(N-1)/2}(8B^{\prime\prime})^{N/2}\int_{0}^{\infty} \int_{-\infty}^{\infty}\int_{y_*-\frac{\varepsilon}{2}}^{y_*+\frac{\varepsilon}{2}}g_{\eta}(t)p_{\nabla H_{N}(x)}(0) f_{\frac{1}{N}H_{N}(x)}(u) \rho^{N-1}\notag\\
 &\quad\times\mathbb{E}[|\operatorname{det}(\GOE_{N}+t I_{N})|\indi_{\left\{\left|\lambda_{\min }\left(\GOE_{N}\right)+t+\sqrt{2}-y_*\right|>\varepsilon\right\}}]\dd  t \dd  u \dd  \rho\nonumber\\
 &\quad+S_{N-1}N^{(N-1)/2}(8B^{\prime\prime})^{N/2}\int_{0}^{\infty}\int_{-\infty}^{\infty}
 \left(\int_{-\infty}^{y_*-\frac{\varepsilon}{2}}+\int_{y_*+\frac{\varepsilon}{2}}^{\infty}\right)g_{\eta}(t)p_{\nabla H_{N}(x)}(0) f_{\frac{1}{N}H_{N}(x)}(u) \nonumber\\
 &\quad\times\mathbb{E}[|\operatorname{det}(\GOE_{N}+t I_{N})|\indi_{\left\{\left|\lambda_{\min }\left(\GOE_{N}\right)+t+\sqrt{2}-y_*\right|>\varepsilon\right\}}]\rho^{N-1} \dd  t \dd  u \dd  \rho\nonumber\\
 &=: I_{N}^{1}+I_{N}^{2} \notag,
\end{align}
where $y_*=-\frac{1}{\sqrt{2}}\left(\frac{\mu}{\sqrt{4 B^{\prime \prime}}}+\frac{\sqrt{4B^{\prime \prime}}}{\mu}\right)$ is defined in (\ref{219}). We claim that both $I_{N}^{1}$ and $I_{N}^{2}$ converge to 0 as $N\to\8$, which implies that
\begin{align}
\label{hh1}
\lim _{N \rightarrow \infty}\mathbb{E}\left[\#\left\{x \in \Omega_{\varepsilon}: \nabla H_{N}(x)=0\right\}\right]=0,
\end{align}
and \eqref{xx8} follows from Markov's inequality.

We first consider $I_{N}^{2}$.  It is clear that
\begin{align}
\label{qq4}
I_{N}^{2}&\leq S_{N-1}N^{(N-1)/2}(8B^{\prime\prime})^{N/2}\int_{0}^{\infty}\int_{-\infty}^{\infty}
 \left(\int_{-\infty}^{y_*-\frac{\varepsilon}{2}}+\int_{y_*+\frac{\varepsilon}{2}}^{\infty}\right) \mathbb{E}[|\operatorname{det}(\GOE_{N}+t I_{N})|]\\
 &\quad\times g_{\eta}(t)p_{\nabla H_{N}(x)}(0) f_{\frac{1}{N}H_{N}(x)}(u) \rho^{N-1}\dd  t \dd  u \dd  \rho \nonumber.
\end{align}
Following the same argument as for $J_{N}^{2}$ in (\ref{mn}), we can show that $\lim_{N\to\8} I_{N}^{2}= 0$ with the help of \pref{le:stech}.

For the integral $I_{N}^{1}$, since for any $t\in\left[y_*-\frac{\varepsilon}{2},y_*+\frac{\varepsilon}{2}\right]$, by the triangle inequality $\left|\lambda_{\min }\left(\GOE_{N}\right)+t+\sqrt{2}-y_*\right|>\varepsilon$ implies that $\left|\lambda_{\min }\left(\GOE_{N}\right)+\sqrt{2}\right|>\frac{\varepsilon}{2}$, it follows that
\begin{align}
\label{pp8}
I_{N}^{1}&\leq S_{N-1} N^{(N-1)/2}(8B^{\prime\prime})^{N/2}\int_{0}^{\infty} \int_{-\infty}^{\infty}\int_{y_*-\frac{\varepsilon}{2}}^{y_*+\frac{\varepsilon}{2}}g_{\eta}(t)p_{\nabla H_{N}(x)}(0) f_{\frac{1}{N}H_{N}(x)}(u) \rho^{N-1}\\
 &\quad\times\mathbb{E}[|\operatorname{det}(\GOE_{N}+t I_{N})|\indi_{\left\{\left|\lambda_{\min }\left(\GOE_{N}\right)+\sqrt{2}\right|>\frac{\varepsilon}{2}\right\}}]\dd  t \dd  u \dd  \rho
 \nonumber\\
&\leq S_{N-1} N^{(N-1)/2}(8B^{\prime\prime})^{N/2}\int_{0}^{\infty} \int_{-\infty}^{\infty}\int_{y_*-\frac{\varepsilon}{2}}^{y_*+\frac{\varepsilon}{2}}g_{\eta}(t)p_{\nabla H_{N}(x)}(0) f_{\frac{1}{N}H_{N}(x)}(u) \rho^{N-1}\nonumber\\
 &\quad\times\mathbb{E}\left[|\operatorname{det}(\GOE_{N}+t I_{N})|^2\right]^{\frac{1}{2}}\mathbb{P}\left(\left|\lambda_{\min }\left(\GOE_{N}\right)+\sqrt{2}\right|>\frac{\varepsilon}{2}\right)^{\frac{1}{2}}\dd  t \dd  u \dd  \rho \nonumber,
\end{align}
where the last inequality follows from the Cauchy--Schwarz inequality. Since $\mu>\sqrt{4 B^{\prime \prime}}$, we can find $\varepsilon$ small enough such that $-\frac{1}{\sqrt{2}}\left(\frac{\mu}{\sqrt{4 B^{\prime \prime}}}+\frac{\sqrt{4B^{\prime \prime}}}{\mu}\right)+\frac{\varepsilon}{2}<-\sqrt{2}$. Thus, by Lemma \ref{x-13}, for all $0<\frac{\de'}2<-\sqrt2- (y_*+\frac\eps2)$ we have for $N$ large enough and all $t \in\left[y_*-\frac{\varepsilon}{2},y_*+\frac{\varepsilon}{2}\right]$,
\begin{align}
\label{pp5}
\mathbb{E}\left[\left|\operatorname{det}\left(t  I_{N}+\GOE_{N}\right)\right|^{2}\right] \leq e^{\frac{\delta' N}{2}} \mathbb{E}\left[\left|\operatorname{det}\left(t  I_{N}+\GOE_{N}\right)\right|\right]^{2}.
\end{align}
On the other hand, by (\ref{pp6}), for large enough $N$ we can find some $\delta>0$ such that
\begin{align}
\label{pp7}
\mathbb{P}\left(\left|\lambda_{\min }\left(\GOE_N\right)+\sqrt{2}\right|>\frac{\varepsilon}{2}\right) \leq e^{-\delta N}.
\end{align}
Plugging (\ref{pp5}) and (\ref{pp7}) into (\ref{pp8}), since we can always choose $\de'\le \de$, we deduce that for all $N$ large enough,
\begin{align}
\label{pp9}
I_{N}^{1}&\leq e^{\frac{-N \delta} {4}} S_{N-1} N^{(N-1)/2}(8B^{\prime\prime})^{N/2}\int_{0}^{\infty} \int_{-\infty}^{\infty}\int_{y_*-\frac{\varepsilon}{2}}^{y_*+\frac{\varepsilon}{2}}\mathbb{E}\left[|\operatorname{det}(\GOE_{N}+t I_{N})|\right]\\
 &\quad\times g_{\eta}(t)p_{\nabla H_{N}(x)}(0) f_{\frac{1}{N}H_{N}(x)}(u) \rho^{N-1}\dd  t \dd  u \dd  \rho
 \nonumber\\
&\leq e^{\frac{-N \delta} {4}} S_{N-1} N^{(N-1)/2}(8B^{\prime\prime})^{N/2}\int_{0}^{\infty} \int_{-\infty}^{\infty}\int_{-\infty}^{\infty}\nonumber\\
 &\quad\times\mathbb{E}\left[|\operatorname{det}(\GOE_{N}+t I_{N})|\right]g_{\eta}(t)p_{\nabla H_{N}(x)}(0) f_{\frac{1}{N}H_{N}(x)}(u) \rho^{N-1}\dd  t \dd  u \dd  \rho
 \nonumber\\
&=e^{\frac{-N \delta} {4}}\mathbb{E} [\operatorname{Crt}_{N}\left(\mathbb{R}, \mathbb{R}^{N}\right)]\nonumber,
\end{align}
where the last equality can be derived from (\ref{pp3}). Together with Proposition \ref{p-2} (i), we arrive at
$\lim _{N \rightarrow \infty} I_{N}^{1}=0,$  which completes the proof of the claim.
\end{proof}

\section{Topology trivialization regime for LRC fields}\label{se:toptriv}

In this section, we determine the topology trivialization regime for LRC fields. First, we recall the covariance structure of derivatives of $H_N$. 
\begin{lemma}[{\cite[Lemma A.1]{AZ20}}]
\label{x-1}
 Assume Assumptions I and II. Then for $x \in \mathbb{R}^{N}$,
$$
\begin{aligned}
\operatorname{Cov}\left[H_{N}(x), H_{N}(x)\right] &=ND\left(\frac{\|x\|^{2}}{N}\right), \\
\operatorname{Cov}\left[H_{N}(x), \partial_{i} H_{N}(x)\right] &=D^{\prime}\left(\frac{\|x\|^{2}}{N}\right) x_{i}, \\
\operatorname{Cov}\left[\partial_{i} H_{N}(x), \partial_{j} H_{N}(x)\right] &=D^{\prime}(0) \delta_{i j}, \\
\operatorname{Cov}\left[H_{N}(x), \partial_{i j} H_{N}(x)\right] &=2 D^{\prime \prime}\left(\frac{\|x\|^{2}}{N}\right) \frac{x_{i} x_{j}}{N}+\left[D^{\prime}\left(\frac{\|x\|^{2}}{N}\right)-D^{\prime}(0)\right] \delta_{i j}, \\
\operatorname{Cov}\left[\partial_{k} H_{N}(x), \partial_{i j} H_{N}(x)\right] &=0, \\
\operatorname{Cov}\left[\partial_{l k} H_{N}(x), \partial_{i j} H_{N}(x)\right] &=-2 D^{\prime \prime}(0)\left[\delta_{j l} \delta_{i k}+\delta_{i l} \delta_{k j}+\delta_{k l} \delta_{i j}\right] / N.
\end{aligned}
$$
\end{lemma}
The following lemma is a slight variant of the classical Laplace method.
\begin{lemma}
\label{x-3}
Suppose $f(x)$ and $g(x)$ are both twice continuously differentiable functions on $[a, b], a,b\in\rz$. Let $h(x)=f(x)+g(x)$. If  there exists a unique point $x_{0} \in(a, b)$ such that $h\left(x_{0}\right)=\max _{x \in[a, b]} h(x)$ and $ h^{\prime \prime}\left(x_{0}\right)<0$,
then
\begin{align*}
\lim _{n \rightarrow \infty} \frac{\int_{a}^{b} e^{n f(x)+\sqrt{n(n-1)}g(x)} \dd x}{e^{ \left(nf(x_{0})+\sqrt{n(n-1)}g(x_0)\right)} \sqrt{\frac{2 \pi}{n\left(-h^{\prime \prime}\left(x_{0}\right)\right)}}}=1.
\end{align*}
\end{lemma}
\begin{proof}
We first prove the lower bound. For any $\varepsilon>0$, since $h^{\prime \prime}$ and $g$  are both continuous, there     exists $\delta>0$ such that if $\left|x-x_{0}\right|<\delta$ then $h^{\prime \prime}(x) \geq h^{\prime \prime}\left(x_{0}\right)-\varepsilon$ and $\left|g(x)-g(x_{0})\right|<\varepsilon$. By Taylor's Theorem, for any $x \in\left(x_{0}-\delta, x_{0}+\delta\right)$, we have $h(x) \geq h\left(x_{0}\right)+\frac{1}{2}\left(h^{\prime \prime}\left(x_{0}\right)-\varepsilon\right)\left(x-x_{0}\right)^{2}$. Note that
\begin{align}
\label{34}
&n f(x)+\sqrt{n(n-1)}g(x)-(n f(x_0)+\sqrt{n(n-1)}g(x_0))\\&=n(h(x)-h(x_0))-\frac{\sqrt{n}}{\sqrt{n}+\sqrt{n-1}}(g(x)-g(x_0)) \notag.
\end{align}
Then we have the following lower bound
$$
\begin{aligned}
&\int_{a}^{b} e^{n f(x)+\sqrt{n(n-1)}g(x)} \dd x  \geq \int_{x_{0}-\delta}^{x_{0}+\delta}e^{n f(x)+\sqrt{n(n-1)}g(x)} \dd  x \\
& = e^{ nf(x_{0})+\sqrt{n(n-1)}g(x_0)} \int_{x_{0}-\delta}^{x_{0}+\delta} e^{n(h(x)-h(x_0))-\frac{\sqrt{n}}{\sqrt{n}+\sqrt{n-1}}(g(x)-g(x_0))} \dd  x.
\end{aligned}
$$
Since $\frac{\sqrt{n}}{\sqrt{n}+\sqrt{n-1}}$ converges to $\frac{1}{2}$ as $n$ tends to infinity, there exists $N$ such that if $n>N$, $|\frac{\sqrt{n}}{\sqrt{n}+\sqrt{n-1}}-\frac{1}{2}|<\varepsilon$. Hence when $n>N$, we have the lower bound of the above integral
$$
\begin{aligned}
&e^{ nf(x_{0})+\sqrt{n(n-1)}g(x_0)} \int_{x_{0}-\delta}^{x_{0}+\delta} e^{n(h(x)-h(x_0))-\frac{\sqrt{n}}{\sqrt{n}+\sqrt{n-1}}(g(x)-g(x_0))} \dd  x\\
& \geq e^{ nf(x_{0})+\sqrt{n(n-1)}g(x_0)}e^{-\varepsilon(\frac{1}{2}+\varepsilon)} \int_{x_{0}-\delta}^{x_{0}+\delta} e^{\frac{n}{2}(\left(h^{\prime \prime}\left(x_{0}\right)-\varepsilon\right)\left(x-x_{0}\right)^{2})} \dd  x\\
& = e^{ nf(x_{0})+\sqrt{n(n-1)}g(x_0)}e^{-\varepsilon(\frac{1}{2}+\varepsilon)}\frac{1}{\sqrt{n\left(-h^{\prime \prime}\left(x_{0}\right)+\varepsilon\right)}} \int_{-\delta \sqrt{n\left(-h^{\prime \prime}\left(x_{0}\right)+\varepsilon\right)}}^{\delta \sqrt{n\left(-h^{\prime \prime}\left(x_{0}\right)+\varepsilon\right)}} e^{-\frac{1}{2} y^{2}}\dd  y,
\end{aligned}
$$
where in the last step we have used a change of variable
$$
y=\sqrt{n\left(-h^{\prime \prime}\left(x_{0}\right)+\varepsilon\right)}\left(x-x_{0}\right) .
$$
Note that $h^{\prime \prime}\left(x_{0}\right)<0$. Dividing both sides of the above inequality by
$$
e^{ nf(x_{0})+\sqrt{n(n-1)}g(x_0)} \sqrt{\frac{2 \pi}{n\left(-h^{\prime \prime}\left(x_{0}\right)\right)}}
,$$
and taking the limits yield
\begin{align}
\label{1}
&\liminf_{n \rightarrow \infty} \frac{\int_{a}^{b} e^{n f(x)+\sqrt{n(n-1)}g(x)} \dd  x}{e^{ \left(nf(x_{0})+\sqrt{n(n-1)}g(x_0)\right)} \sqrt{\frac{2 \pi}{n\left(-h^{\prime \prime}\left(x_{0}\right)\right)}}} \\
&\geq \liminf _{n \rightarrow \infty} \frac{1}{\sqrt{2 \pi}} \int_{-\delta \sqrt{n\left(-h^{\prime \prime}\left(x_{0}\right)+\varepsilon\right)}}^{\delta \sqrt{n\left(-h^{\prime \prime}\left(x_{0}\right)+\varepsilon\right)}} e^{-\frac{1}{2} y^{2}} \dd  y \cdot e^{-\varepsilon(\frac{1}{2}+\varepsilon)} \sqrt{\frac{-h^{\prime \prime}\left(x_{0}\right)}{-h^{\prime \prime}\left(x_{0}\right)+\varepsilon}}\nonumber\\
&=e^{-\varepsilon(\frac{1}{2}+\varepsilon)} \sqrt{\frac{-h^{\prime \prime}\left(x_{0}\right)}{-h^{\prime \prime}\left(x_{0}\right)+\varepsilon}} \notag.
\end{align}
Since the inequality (\ref{1}) is true for arbitrary $\varepsilon$, letting $\varepsilon\rightarrow 0$, we get the lower bound
$$
\liminf_{n \rightarrow \infty} \frac{\int_{a}^{b} e^{n f(x)+\sqrt{n(n-1)}g(x)} \dd  x}{e^{ \left(nf(x_{0})+\sqrt{n(n-1)}g(x_0)\right)} \sqrt{\frac{2 \pi}{n\left(-h^{\prime \prime}\left(x_{0}\right)\right)}}} \geq 1.
$$

 Now we prove the upper bound. Let $\varepsilon$ be small enough so that
$h^{\prime \prime}\left(x_{0}\right)+\varepsilon<0$.  By the continuity of $h^{\prime \prime}$ and $g$, using Taylor's Theorem we can find $\delta>0$ so that if $\left|x-x_{0}\right|<\delta$, then $h(x) \leq h\left(x_{0}\right)+\frac{1}{2}\left(h^{\prime \prime}\left(x_{0}\right)+\varepsilon\right)\left(x-x_{0}\right)^{2}$  and $\left|g(x)-g(x_{0})\right|<\varepsilon$.
By our assumptions $[a, b]$ is a compact set, so there exists an $\eta>0$ and large enough $M$, such that if $\left|x-x_{0}\right| \geq \delta$, then $h(x) \leq h\left(x_{0}\right)-\eta$ and   $\left|g(x)-g(x_{0})\right|\leq M$.
Together with (\ref{34}), we have for $n$ large enough
\begin{align}
\label{2}
&\int_a^b e^{n f(x)+\sqrt{n(n-1)}g(x)} \dd  x \\ & = \int_{a}^{x_{0}-\delta} e^{n f(x)+\sqrt{n(n-1)}g(x)} \dd  x+\int_{x_{0}-\delta}^{x_{0}+\delta} e^{n f(x)+\sqrt{n(n-1)}g(x)} \dd  x\nonumber\\
&\quad+\int_{x_{0}+\delta}^{b} e^{n f(x)+\sqrt{n(n-1)}g(x)} \dd  x \nonumber\\
& \leq e^{ nf(x_{0})+\sqrt{n(n-1)}g(x_0)}(b-a)e^{-n\eta+(\frac{1}{2}+\varepsilon)M}\nonumber
\\&\quad +e^{ nf(x_{0})+\sqrt{n(n-1)}g(x_0)}\int_{x_{0}-\delta}^{x_{0}+\delta}e^{\varepsilon(\frac{1}{2}+\varepsilon)} e^{\frac{n}{2}(\left(h^{\prime \prime}\left(x_{0}\right)+\varepsilon\right)\left(x-x_{0}\right)^{2})}\dd  x \nonumber \\
& \leq e^{ nf(x_{0})+\sqrt{n(n-1)}g(x_0)}(b-a)e^{-n\eta+(\frac{1}{2}+\varepsilon)M}\nonumber
\\&\quad +e^{ nf(x_{0})+\sqrt{n(n-1)}g(x_0)}\int_{-\infty}^{+\infty}e^{\varepsilon(\frac{1}{2}+\varepsilon)} e^{\frac{n}{2}(\left(h^{\prime \prime}\left(x_{0}\right)+\varepsilon\right)\left(x-x_{0}\right)^{2})}\dd  x \nonumber \\
& \leq e^{ nf(x_{0})+\sqrt{n(n-1)}g(x_0)}(b-a)e^{-n\eta+(\frac{1}{2}+\varepsilon)M}\nonumber\\
&\quad+e^{ nf(x_{0})+\sqrt{n(n-1)}g(x_0)}e^{\varepsilon(\frac{1}{2}+\varepsilon)}  \sqrt{\frac{2 \pi}{n\left(-h^{\prime \prime}\left(x_{0}\right)-\varepsilon\right)}} \notag.
\end{align}
If we divide both sides of the inequality ({\ref{2}}) by
$$
e^{ nf(x_{0})+\sqrt{n(n-1)}g(x_0)} \sqrt{\frac{2 \pi}{n\left(-h^{\prime \prime}\left(x_{0}\right)\right)}}
,$$
and take limits on the above expression, then
$$
\limsup_{n \rightarrow \infty} \frac{\int_{a}^{b} e^{n f(x)+\sqrt{n(n-1)}g(x)} \dd  x}{e^{ \left(nf(x_{0})+\sqrt{n(n-1)}g(x_0)\right)} \sqrt{\frac{2 \pi}{n\left(-h^{\prime \prime}\left(x_{0}\right)\right)}}}\leq e^{\varepsilon(\frac{1}{2}+\varepsilon)} \sqrt{\frac{-h^{\prime \prime}\left(x_{0}\right)}{-h^{\prime \prime}\left(x_{0}\right)-\varepsilon}}.
$$
Letting $\varepsilon\rightarrow 0$, we get the upper bound
$$
\limsup_{n \rightarrow \infty} \frac{\int_{a}^{b} e^{n f(x)+\sqrt{n(n-1)}g(x)} \dd  x}{e^{ \left(nf(x_{0})+\sqrt{n(n-1)}g(x_0)\right)} \sqrt{\frac{2 \pi}{n\left(-h^{\prime \prime}\left(x_{0}\right)\right)}}} \leq 1.
$$
Combining this with the lower bound gives the desired result.
\end{proof}
To simplify the notation, we set
\begin{align*}
m=-\mu / \sqrt{-4 D^{\prime \prime}(0)}.
\end{align*}
For $x\in \rz$, define \begin{align*}
\sfG(x)=-\frac{1}{2}x^2+2mx,
\end{align*}
and
\begin{align*}
    \sfF(x)=\sfG(x)+\Phi(x)=-\frac{1}{2}x^2+2mx+\Phi(x),
\end{align*}
where $\Phi(x)$ is defined in (\ref{eq:phi}).

The following results will be used to reduce the Laplace asymptotics to compact sets. For this purpose, in general one can employ the exponential tightness results from \cite[{Section 4}]{AZ20}. Here we only consider the special case of all critical points, and we follow an idea of \cite{Be22} to give an alternative treatment based on asymptotic properties of the one point correlation function $\rho_N$, from which the Laplace asymptotics follows directly.
\begin{lemma}
\label{x-4}
The following properties hold:

(i) For any large $M>0$,
\begin{align}
\label{9}
e^{-\frac{1}{2}N x^{2}+2 \sqrt{N(N-1)} m x}=e^{N \sfG(x)+o(N)},
\end{align}
where the error term is uniform in $|x|\leq M $.

(ii) Moreover,
\begin{align}
\label{10}
e^{N \sfG(x)+o(N)} \rho_{N}(x)=e^{N \sfF(x)[1+o(1)]+o(N)},
\end{align}
where the error terms are both uniform in $x \in \mathbb{R} $.
\end{lemma}
\begin{proof}
 Taking the logarithm of the left-hand side of (\ref{9}) gives $-\frac{1}{2}N x^{2}+2 \sqrt{N(N-1)} m x$, which is equal to $N\sfG(x)-(2Nmx-2 \sqrt{N(N-1)} m x)$. Furthermore, for all $|x|\leq M $, we have \begin{align*}
 &|2Nmx-2 \sqrt{N(N-1)} m x|\leq|2mM|(N-\sqrt{N(N-1)})\\
&=|2mM|\frac{\sqrt{N}}{\sqrt{N}+\sqrt{N-1}} \stackrel{N \rightarrow \infty}{\longrightarrow}  |mM|,
\end{align*}
 which completes the proof of the first part.

By Lemma \ref{x-2} (ii), taking the logarithm of the left-hand side of (\ref{10}), we get
$$
N \sfF(x)+o\left(N\left(\left|\Phi(x)\right|+1\right)\right).
$$
The error term can be bounded by $o(N)$ on any compact subset of $\mathbb{R}$. Since for all  $x \in \mathbb{R}$, there exist some constants $c, c^{\prime}$, such that
\begin{align}
\label{312}
-\frac{x^{2}}{2} \leq \Phi(x) \leq-\frac{x^{2}}{2}+c+c^{\prime} x ,
\end{align}
we find that for large enough $x$ the error term can be bounded by $o(N|\sfF(x)|)$.
\end{proof}

\begin{proposition}
\label{p-1}
\begin{itemize}
    \item[(i)]  If $\mu>\sqrt{-2 D^{\prime \prime}(0)}$, the unique maximizer of $\sfF(x)$ is at $x_\ast=m +\frac{1}{2m}$. Furthermore, $\sfF^{\prime\prime}(x_\ast)=-\frac{4m^2}{2m^2-1}$ and $\sfF(x_\ast)=m^2+\log |m|+\frac{1}{2}(1+\log 2)$.
    \item[(ii)] If $0<\mu \leq \sqrt{-2 D^{\prime \prime}(0)}$, the unique maximizer of $\sfF(x)$ is at $x_0=2m$. Furthermore, $\sfF^{\prime\prime}(x_0)=-1$ and $\sfF(x_0)=2m^2$.
\end{itemize}
\end{proposition}

\begin{proof}
By the definition in (\ref{eq:phi}), we know $\Phi$ is a continuous and differentiable function on $\mathbb{R}$. Thus, function $\sfF(x)=-\frac{1}{2}x^2+2mx+\Phi(x)$ is also continuous and differentiable. Moreover, $\sfF(x)$ tends to $-\infty$ as $|x| \rightarrow \infty$. Therefore, ${\sfF}(x)$ must have a maximizer, which should be a critical point. Since $m=-\frac{\mu}{\sqrt{-4 D^{\prime \prime}(0)}}<0$, we know $\mu>\sqrt{-2 D^{\prime \prime}(0)}$ is equivalent to $m<-\frac{\sqrt{2}}{2}$ and $0<\mu \leq \sqrt{-2 D^{\prime \prime}(0)}$ is equivalent to $-\frac{\sqrt{2}}{2}\leq m<0$, respectively. Due to the expression of $\sfF(x)$ and $m<0$, we can easily see that $\sfF(x)\leq \sfF(-x)$ for all $x \in [0,+\infty)$, where the equality holds only when $x=0$. Therefore, in order to find a maximizer of $\sfF(x)$, it suffices
to consider the critical points in $(-\infty,0]$. Taking the derivative of $\sfF(x)$, we obtain
\begin{align*}
\sfF^{\prime}(x)=-x+2m+\sqrt{x^{2}-2}\indi_{\{x \leq -\sqrt{2}\}}.
\end{align*}
When $-\frac{\sqrt{2}}{2}\leq m<0$, letting $\sfF^{\prime}(x)=0$, we can get the unique solution $x_0=2m$. Moreover, $\sfF^{\prime}(x)>0$ for $x\in(-\infty,x_0)$ and $\sfF^{\prime}(x)<0$ for $x\in(x_0,0)$. Thus $\sfF(x)$ attains the unique maximizer  at $x_0=2m$. By straightforward calculation, we conclude that $\sfF^{\prime\prime}(x_0)=-1$ and $\sfF(x_0)=2m^2$.
When $m<-\frac{\sqrt{2}}{2}$, by the same argument as above, we see that $x_\ast=m +\frac{1}{2m}$ is the unique solution of the equation $\sfF^{\prime}(x)=0$. Similarly, since $\sfF^{\prime}(x)>0$ for $x\in(-\infty,x_\ast)$ and $\sfF^{\prime}(x)<0$ for $x\in(x_\ast,0)$, $\sfF(x)$ attains the unique maximizer  at $x_\ast$.  After a little algebra, we get $\sfF^{\prime\prime}(x_\ast)=-\frac{4m^2}{2m^2-1}$ and $\sfF(x_\ast)=m^2+\log |m|+\frac{1}{2}(1+\log 2)$.
\end{proof}

Now we can prove the main result of this section.
\begin{theorem}
\label{h-1}
 Recall $\operatorname{Crt}_{N}\left(\mathbb{R}, \mathbb{R}^{N}\right)$ is the total number of critical points of $H_{N}(x)$.
\begin{itemize}
    \item[(i)] If $\mu>\sqrt{-2 D^{\prime \prime}(0)}$, then
\begin{align}\label{R1}
\lim _{N \rightarrow \infty} \mathbb{E}\left[\operatorname{Crt}_{N}\left(\mathbb{R}, \mathbb{R}^{N}\right)\right]=1 ,
\end{align}
\item[(ii)] If $0<\mu \leq \sqrt{-2 D^{\prime \prime}(0)}$,  then
\begin{align}
\label{by}
\lim _{N \rightarrow \infty} \frac{1}{N} \log \mathbb{E}\left[\operatorname{Crt}_{N}\left(\mathbb{R}, \mathbb{R}^{N}\right)\right]= -\log \frac{\mu}{\sqrt{-2 D^{\prime \prime}(0)}}+\frac{\mu^{2}}{-4 D^{\prime \prime}(0)}-\frac{1}{2}.
\end{align}
\end{itemize}
\end{theorem}
\begin{proof}

Using the Kac--Rice formula and Lemma \ref{x-1}, we find
\begin{align*}
\mathbb{E} [\operatorname{Crt}_{N}\left(\mathbb{R}, \mathbb{R}^{N}\right)]=\int_{\mathbb{R}^{N}} \mathbb{E}\left[\left|\operatorname{det} \nabla^{2} H_{N}(x)\right|\right] p_{\nabla H_{N}(x)}(0) \dd  x,
\end{align*}
where $p_{\nabla H_{N}(x)}(0)=\frac{1}{(2 \pi)^{N / 2} D^{\prime}(0)^{N / 2}} \exp \left(-\frac{\mu^{2}\|x\|^{2}}{2 D^{\prime}(0)}\right)$.
Moreover,
$$\nabla^{2} H_{N}(x) \stackrel{d}= \sqrt{-4 D^{\prime \prime}(0)} \GOE_N-\left(\sqrt{\frac{-2 D^{\prime \prime}(0)}{N}} Z-\mu\right)  I_{N},
$$
where $Z$ is a standard Gaussian random variable independent of $\GOE_N$. Since
$
m=-\mu / \sqrt{-4 D^{\prime \prime}(0)},
$
it follows from  \cite[Lemma 2.1]{AZ20} with $a=\sqrt{-4 D^{\prime \prime}(0)}, b=-\mu, \sigma=\sqrt{-2 D^{\prime \prime}(0)}$ that
\begin{align*}
\mathbb{E}\left[\left|\operatorname{det} \nabla^{2} H_{N}(x)\right|\right]&=\frac{\sqrt{2}\left[-4 D^{\prime \prime}(0)\right]^{N / 2} \Gamma\left(\frac{N+1}{2}\right)(N+1)}{\sqrt{\pi} N^{N / 2} e^{N m^{2}}}\\&\quad\times \int_{\mathbb{R}} e^{-\frac{1}{2}(N+1) w^{2}+2 \sqrt{N(N+1)} m w} \rho_{N+1}(w)\dd w.
\end{align*}
Straightforward calculation gives
$
\int_{\mathbb{R}^{N}} p_{\nabla H_{N}(x)}(0) \mathrm{d} x=\frac{1}{\mu^N}.
$
Therefore, we get
\begin{align}
\label{318}
\mathbb{E} [\operatorname{Crt}_{N}\left(\mathbb{R}, \mathbb{R}^{N}\right)]=\frac{\sqrt{2}\Gamma\left(\frac{N+1}{2}\right)(N+1)}{\sqrt{\pi}|m|^N N^{N / 2} e^{N m^{2}}} \int_{\mathbb{R}} e^{-\frac{1}{2}(N+1) w^{2}+2 \sqrt{N(N+1)} m w} \rho_{N+1}(w)\dd w.
\end{align}
Using Lemma \ref{x-4} and the Stirling formula, the above integral can be written as
\begin{align*}
\mathbb{E} [\operatorname{Crt}_{N}\left(\mathbb{R}, \mathbb{R}^{N}\right)]&=\exp \left(-N\left[m^2+\log |m|+\frac{1}{2}(1+\log 2)\right]+o(N)\right)\\&\quad\times \int_{\mathbb{R}} e^{(N+1) \sfF(x)(1+o(1))} \dd  x .
\end{align*}
By the expression of $\sfF(x)$, there exists a large $M$ such that the above integral is equal to
$$
\begin{aligned}
&\mathbb{E} [\operatorname{Crt}_{N}\left(\mathbb{R}, \mathbb{R}^{N}\right)]=\exp \left(-N\left[m^2+\log |m|+\frac{1}{2}(1+\log 2)\right]+o(N)\right) \\&\quad\quad\quad\quad\quad\quad\quad\quad\times\left(\int_{-M}^{M} e^{(N+1) \sfF(x)(1+o(1))} \dd  x+o(1)\right).
\end{aligned}
$$
Then using the Laplace method yields that
\begin{align*}
&\mathbb{E} [\operatorname{Crt}_{N}\left(\mathbb{R}, \mathbb{R}^{N}\right)]\\&=\exp \left(\left[-N[m^2+\log |m|+\frac{1}{2}(1+\log 2)]+(N+1)\max _{x \in[-M,M]} \sfF(x)\right]+o(N)\right). \end{align*}
Thus, applying Proposition \ref{p-1} (ii) immediately gives Theorem \ref{h-1} (ii).

By the same argument as above, using  Proposition \ref{p-1} (i) we obtain
$$
\lim _{N \rightarrow \infty} \frac{1}{N} \log \mathbb{E} [\operatorname{Crt}_{N}\left(\mathbb{R}, \mathbb{R}^{N}\right)]=0 \text { if } \mu>\sqrt{-2 D^{\prime \prime}(0)},
$$
which is a weaker form of the triviality claimed in item (i).

Now we turn to prove the case (i). By Lemma \ref{x-4}, there exists a sufficiently large $M$ such that (\ref{318}) is equal to
\begin{align*}
\mathbb{E} [\operatorname{Crt}_{N}\left(\mathbb{R}, \mathbb{R}^{N}\right)]&=\frac{\Gamma\left(\frac{N+1}{2}\right)\sqrt{N+1}}{\sqrt{2}\pi|m|^N N^{N / 2} e^{N m^{2}}}\\&\quad\times\int_{-M}^{M}e^{-\frac{1}{2}(N+1) w^{2}+2 \sqrt{N(N+1)} m w} \rho_{N+1}(w)\dd w+o(1) .
\end{align*}
Combining Proposition \ref{p-1} (i) with the Stirling formula, using Laplace principle, we can further restrict to any fixed neighborhood $\left[x_{*}-\varepsilon, x_{*}+\varepsilon\right] $ of the maximizer and  yield only $o(1)$ error. Note that we have assumed $\mu>\sqrt{-2 D^{\prime \prime}(0)}$, which means $x_{*}<-\sqrt{2}$ by  Proposition \ref{p-1}.  Hence choosing $\varepsilon> 0$ small enough allows the use of Lemma \ref{x-2}  (i), which implies
\begin{align*}
\mathbb{E} [\operatorname{Crt}_{N}\left(\mathbb{R}, \mathbb{R}^{N}\right)]&=\frac{\Gamma\left(\frac{N+1}{2}\right)\sqrt{N+1}}{\sqrt{2}\pi|m|^N N^{N / 2} e^{N m^{2}}}\\&\quad\times\int_{x_{*}-\varepsilon}^{x_{*}+\varepsilon}\frac{e^{(N+1)(-\frac{1}{2}x^{2}+\Phi(x))+2 \sqrt{N(N+1)}m x}}{\left(x^{2}-2\right)^{\frac{1}{4}}\left(|x|+\sqrt{x^{2}-2}\right)^{\frac{1}{2}+o(1)}} \dd  x+o(1) .
\end{align*}
Using Lemma \ref{x-3} with $f(x)=-\frac{1}{2}x^{2}+\Phi(x)$ and $g(x)=2mx$, by the mean value theorem, we have as $N\to\8$
\begin{align*}
     \mathbb{E}\left[\operatorname{Crt}_{N}\left(\mathbb{R}, \mathbb{R}^{N}\right)\right]&\sim \frac{\Gamma\left(\frac{N+1}{2}\right)\sqrt{N+1}}{\sqrt{2}\pi|m|^N N^{N / 2} e^{N m^{2}}}\frac{e^{(N+1)(-\frac{1}{2}x_{*}^{2}+\Phi(x_{*}))+2 \sqrt{N(N+1)}m x_{*}}}{\left(\tilde x^{2}-2\right)^{\frac{1}{4}}\left(|\tilde x|+\sqrt{\tilde x^{2}-2}\right)^{\frac{1}{2}+o(1)}}\\&\quad\times
\sqrt{\frac{2\pi}{(N+1)(-\sfF^{\prime\prime}(x_\ast))}},
\end{align*}
where $\tilde x\in (x_*-\eps,x_*+\eps)$ may depend on $N$. Using Proposition \ref{p-1} (i), plugging the corresponding values of $\Phi(x_{*})$ and $\sfF^{\prime\prime}(x_\ast)$ into the above limit, and replacing the $\Gamma$ function  with its asymptotics $\Gamma(N/2)\sim 2 \sqrt{\pi/N} e^{-N/2}(N/2)^{N/2}$,  sending $N\to\8$ and then $\eps\downarrow0$, after some tedious but easy calculation, we eventually get
$$\lim _{N \rightarrow \infty} \mathbb{E}\left[\operatorname{Crt}_{N}\left(\mathbb{R}, \mathbb{R}^{N}\right)\right]=1 ,$$
which is the desired result of triviality claimed in this theorem.
\end{proof}

\begin{remark}
Although the result (\ref{by}) has been proved in \cite[Theorem 1]{AZ20} with $\Xi=0$, the method here is slightly different, and it can be viewed as a by-product of our argument.
\end{remark}

\section{Proofs of main results for LRC fields}\label{se:lrc}
To state our results in this section, we first recall some useful constants as below, which are taken from \cite[Equation (3.15)]{AZ20}:
\begin{align}
\label{eq:msialbt}
  m_1 & =m_1(\rho,u)= \mu + \frac{(u-\frac{\mu\rho^2}{2} +\frac{\mu D'(\rho^2)\rho^2}{D'(0)}) ( 2D''(\rho^2)\rho^2+D'(\rho^2) -D'(0) )}{D(\rho^2) -\frac{D'(\rho^2)^2 \rho^2}{D'(0) }}, \\
  m_2&=m_2(\rho,u)= \mu + \frac{(u-\frac{\mu \rho^2}{2} +\frac{\mu D'(\rho^2) \rho^2}{D'(0)}) (D'(\rho^2) -D'(0) )}{D( \rho^2) -\frac{D'(\rho^2)^2 \rho^2}{D'(0) }}, \notag\\
\si_1 & =\si_1(\rho)= \sqrt{\frac{-4D''(0)-(\al \rho^2 +\bt)\al\rho^2}{N}}, \ \ \
  \si_2  =\si_2(\rho)= \sqrt{\frac{-2D''(0)-(\al\rho^2 +\bt)\bt}{N}}, \notag \\
  m_Y&=m_Y(\rho)= \frac{\mu\rho^2}{2}-\frac{\mu D'(\rho^2) \rho^2}{D'(0) }, \ \ \ \si_Y =\si_Y(\rho) =\sqrt{\frac1N\Big(D(\rho^2)-\frac{D'(\rho^2)^2\rho^2}{D'(0)} \Big)},\notag \\
  \alpha &=\alpha(\rho^2)= \frac{2D''(\rho^2)}{ \sqrt{ D(\rho^2)-\frac{D'(\rho^2)^2 \rho^2}{D'(0)}}},  \ \ \
   \beta =\beta(\rho^2)=\frac{D'(\rho^2 )-D'(0)}{\sqrt{ D(\rho^2 )-\frac{D'(\rho^2)^2 \rho^2}{D'(0)}}}\nonumber. 
\end{align}
Before proving Theorem \ref{h-2}, we also recall some facts from \cite{AZ20}. Let $E \subset \mathbb{R}$ be a Borel set and $B_{N}\left(R_{1}, R_{2}\right)=\{x \in \mathbb{R}^{N}: R_{1}<\frac{\|x\|}{\sqrt{N}}<R_{2}\}$ be a shell where $0 \leq R_{1}<R_{2} \leq \infty$. In this case we write $\operatorname{Crt}_{N}\left(E,\left(R_{1}, R_{2}\right)\right)=$ $\operatorname{Crt}_{N}\left(E, B_{N}\left(R_{1}, R_{2}\right)\right)$.
Using the spherical coordinates and writing $\rho=\frac{\|x\|}{\sqrt{N}}$, by the Kac--Rice formula and \cite[Proposition 3.3]{AZ20},
$$
\begin{aligned}
&\mathbb{E} \operatorname{Crt}_{N}\left(E,\left(R_{1}, R_{2}\right)\right)\\ 
&=S_{N-1} N^{(N-1) / 2} \int_{R_{1}}^{R_{2}} \int_{E} \mathbb{E}[|\operatorname{det} G|] \frac{1}{\sqrt{2 \pi} \sigma_{Y}} e^{-\frac{\left(u-m_{Y}\right)^{2}}{2 \sigma_{Y}^{2}}}\\&\quad\times \frac{1}{(2 \pi)^{N / 2} D^{\prime}(0)^{N / 2}} e^{-\frac{N \mu^{2} \rho^{2}}{2 D^{\prime}(0)}} \rho^{N-1} \dd  u \dd  \rho;
\end{aligned}
$$
see \cite[{(4.1)}]{AZ20} for details. In the above formula,  $Y=\frac{H_{N}(x)}{N}-\frac{D^{\prime}\left(\frac{\|x\|^{2}}{N}\right) \sum_{i=1}^{N} x_{i} \partial_{i} H_{N}(x)}{N D^{\prime}(0)}$, $m_Y$ and $\sigma_{Y}^{2}$ are the mean and the variance of $Y$, $S_{N-1}=\frac{2 \pi^{N / 2}}{\Gamma(N / 2)}$ is the area of $N-1$ dimensional unit sphere, and
\begin{align}\label{eq:gu}
    G= G(u) =
    \begin{pmatrix}
      z_1'& \xi^\mathsf T \\
       \xi & \sqrt{-4D''(0)} (\sqrt{\frac{N-1}{N}}\GOE_{N-1}-z_3'I_{N-1})
    \end{pmatrix}=: \begin{pmatrix}
        z_1'& \xi^\mathsf T \\
         \xi & G_{**}
      \end{pmatrix}  ,
  \end{align}
  where with $z_1,z_2,z_3$ being independent standard Gaussian random variables,
  \begin{align*}
    z_1'&=\si_1 z_1 - \si_2  z_2 + m_1, \quad
z_3'=\frac1{\sqrt{-4D''(0)}}\Big(\si_2 z_2+ \frac{  \sqrt{\al\bt}\rho }{\sqrt N} z_3 - m_2\Big),
  \end{align*}
  and $\xi$ is a centered column Gaussian vector with covariance matrix $\frac{-2D''(0)}{N}I_{N-1}$ which is independent of $z_1,z_2,z_3$ and the GOE matrix $\GOE_{N-1}$. Here the conditional distribution of Hessian was identified with $G$ under a technical condition (Assumption IV) in \cite{AZ20}, which was removed recently in \cite{XYZ23}.
  From the above relations,  conditioning on $z_3'=y$ will only affect the first entry $z_1'$ but not the structure of the matrix $G$. By abuse of notation we will write $G$ for $(G | z_3'=y)$ to save space from time to time, which should not cause ambiguity from context since $G$ is a function of $y$ only if conditioning on $z_3'=y$.  Using the Schur complement formula for the matrix $G$, we have
\begin{align}
\label{sk6}
\operatorname{det} G & =\operatorname{det}\left(G_{* *}\right)\left(z_{1}^{\prime}-\xi^{\top} G_{* *}^{-1} \xi\right) \notag\\
&\stackrel{d}=\left[-4 D^{\prime \prime}(0)\right]^{(N-1) / 2} z_{1}^{\prime} \prod_{j=1}^{N-1}\left(\sqrt{\frac{N-1}{N}}\lambda_{j}-z_{3}^{\prime}\right) \\
&\quad -\frac{\left[-4 D^{\prime \prime}(0)\right]^{N / 2}}{2 N} \sum_{k=1}^{N-1} Z_{k}^{2} \prod_{j \neq k}^{N-1}\left(\sqrt{\frac{N-1}{N}}\lambda_{j}-z_{3}^{\prime}\right) \notag,
\end{align}
where $\lambda_{j}, \  j=1,\dots, N-1,$ are the eigenvalues of $\GOE_{N-1}$ defined in (\ref{eq:gu}), and $Z_k$ are i.i.d.~standard normal random variables. Since the eigenvalues and eigenvectors of GOE matrices are independent, the distributional identity in (\ref{sk6}) follows from rotating the vector $\xi$ in matrix $G$. We remark that this identity is crucial to get many results, including the main results (\ref{gg1}) and \eqref{zx}.
  
The next two results are taken from \cite{AZ20}, and we list them here for the reader's convenience.
\begin{theorem}[{\cite[Theorem 1.2]{AZ20}}]
\label{x-5}
Let $0 \leq R_{1}<R_{2} \leq \infty$ and $E$ be an open set of $\mathbb{R}$. Assume Assumptions I and II. Then
\begin{align*}
\lim _{N \rightarrow \infty} \frac{1}{N} \log \mathbb{E} \operatorname{Crt}_{N}\left(E, B_{N}\left(R_{1}, R_{2}\right)\right)&=\frac{1}{2} \log \left[-4 D^{\prime \prime}(0)\right]-\frac{1}{2} \log D^{\prime}(0)+\frac{1}{2}\\&\quad+\sup _{(\rho, u, y) \in F} \psi_{*}(\rho, u, y),
\end{align*}
where $F=\left\{(\rho, u, y): y \in \mathbb{R}, \rho \in (R_{1}, R_{2}), u \in E\right\}$, and the function $\psi_{*}$ is given by
\begin{align}
\label{ab3}
&\psi_{*}(\rho, u, y)=\Psi_{*}(y)-\frac{\left(u-\frac{\mu \rho^{2}}{2}+\frac{\mu D^{\prime}\left(\rho^{2}\right) \rho^{2}}{D^{\prime}(0)}\right)^{2}}{2\left(D\left(\rho^{2}\right)-\frac{D^{\prime}\left(\rho^{2}\right)^{2} \rho^{2}}{D^{\prime}(0)}\right)}-\frac{\mu^{2} \rho^{2}}{2 D^{\prime}(0)}-\frac{-2 D^{\prime \prime}(0)}{-2 D^{\prime \prime}(0)-\frac{\left[D^{\prime}\left(\rho^{2}\right)-D^{\prime}(0)\right]^{2}}{D\left(\rho^{2}\right)-\frac{D^{\prime}\left(\rho^{2}\right)^{2} \rho^{2}}{D^{\prime}(0)}}}\\&\quad\times\left(y+\frac{1}{\sqrt{-4 D^{\prime \prime}(0)}}\left[\mu+\frac{\left(u-\frac{\mu \rho^{2}}{2}+\frac{\mu D^{\prime}\left(\rho^{2}\right) \rho^{2}}{D^{\prime}(0)}\right)\left(D^{\prime}\left(\rho^{2}\right)-D^{\prime}(0)\right)}{D\left(\rho^{2}\right)-\frac{D^{\prime}\left(\rho^{2}\right)^{2} \rho^{2}}{D^{\prime}(0)}}\right]\right)^{2}+\log \rho \notag .
\end{align}
\end{theorem}
We will use this result in a slightly more general setting, namely besides confining the Hamiltonian $H_N$ to shell domains with constrained critical values, we also restrict the value of $y$ which corresponds to the value of $z_3'$ in the representation \pref{eq:gu}. This is an LRC version of \pref{le:stech}, and the proof for this more general fact follows the same strategy as there, which is a simple modification of the proof for \cite[Theorem 1.2]{AZ20}. We omit the details here.
In the following, we use the same notations $\rho_*$, $u_*$ and $y_*$ for the maximizers of the complexity functions for both the SRC and LRC fields. It should be clear from context which set of values we are using, and the readers should not confuse them.
\begin{lemma}[{\cite[Example 2]{AZ20}}]
\label{x-6}
Assume  $\mu>\sqrt{-2 D^{\prime \prime}(0)}$. Then $\psi_{*}(\rho,u,y)$ attains its unique maximum in $\rz_+\times \rz\times \rz$ at $(\rho_*,u_*,y_*)$, where
\begin{align}
\label{eq:ruy}
&\rho_*=\frac{\sqrt{D^{\prime}(0)}}{\mu},\ \  u_{*}=\frac{D^{\prime}\left(\rho_{*}^{2}\right)-D^{\prime}(0)}{\mu}+\frac{\mu \rho_{*}^{2}}{2}-\frac{\mu D^{\prime}\left(\rho_{*}^{2}\right) \rho_{*}^{2}}{D^{\prime}(0)},\\& y_{*}=-\frac{\mu}{\sqrt{-4 D^{\prime \prime}(0)}}-\frac{\sqrt{-D^{\prime \prime}(0)}}{\mu} \notag.
\end{align}
Moreover, the maximum value is
\begin{align}\label{eq:cpx2}
\psi_{*}\left(\rho_{*}, u_{*}, y_{*}\right)=-\log \sqrt{-4 D^{\prime \prime}(0)}-\frac{1}{2}+\frac{1}{2} \log D^{\prime}(0) .
\end{align}
\end{lemma}
The following lemma is a direct consequence of the above results.

\begin{lemma}
\label{x-7}
Assume  $\mu>\sqrt{-2 D^{\prime \prime}(0)}$ and recall the values of $\rho_*$ and $u_*$ as in \pref{eq:ruy}. For $0\le R_{1} < R_{2}\le \8$, and open set $E\subset\mathbb{R}$,  if $\left(\rho_{*}, u_{*}\right) \notin  \overline{(R_1,R_2)} \times \bar E$, then
\begin{align*}
\lim _{N \rightarrow \infty}  \mathbb{E} [\operatorname{Crt}_{N}\left(E, B_{N}\left(R_{1}, R_{2}\right)\right)]=0.\end{align*}
\end{lemma}
\begin{proof}
By Lemma \ref{x-6}, if $\left(\rho_{*}, u_{*}\right) \notin  \overline{(R_1,R_2)} \times\bar E$, the value of $\sup _{(\rho, u, y) \in F} \psi_{*}(\rho, u, y)$ is  strictly smaller than $\psi_{*}\left(\rho_{*}, u_{*}, y_{*}\right)=-\log \sqrt{-4 D^{\prime \prime}(0)}-\frac{1}{2}+\frac{1}{2} \log D^{\prime}(0)$ since $\overline{(R_1,R_2)} \times \bar E$ is a closed set. Therefore, the right-hand side of \pref{eq:cpx2} is strictly negative, which directly yields the assertion.
\end{proof}
Now let us prove the main results of this section.

\begin{proof}[Proof of Theorem \ref{h-2}]
By Proposition \ref{z-1}, we know this model has at least  one global minimum in $\rz^N$ almost surely. The first triviality conclusion directly follows from (\ref{R1}) and Markov's inequality.

For any $\varepsilon>0$, taking $E=[u_*-\varepsilon,u_*+\varepsilon]^{c}$, $R_1=0$ and $R_2=\infty$, we see that $\left(\rho_{*}, u_{*}\right)$ does not belong to   $\rz_+ \times \bar E$. By Lemma \ref{x-7},
$$
\begin{aligned}
\mathbb{P}\left(\left|\frac{1}{N} H_{N}\left(x^{*}\right)-u_{*}\right|>\varepsilon\right)
&\leq\mathbb{E} [\operatorname{Crt}_{N}\left(E, B_{N}\left(0, \infty\right)\right)] \rightarrow 0
\end{aligned}
$$
as $N\to\8$, which proves \eqref{eq:energy}. Similarly,
$$
\begin{aligned}
&\mathbb{P}\left(\left|\frac{\|x^{*}\|}{\sqrt{N}}-\rho_{*}\right|>\varepsilon\right)=\mathbb{P}\left(\frac{\|x^{*}\|}{\sqrt{N}}\in [\rho_*-\eps,\rho_*+\eps]^c \right) \\
&\leq\mathbb{E} [\operatorname{Crt}_{N}\left(\rz, (0,\rho_*-\eps)\right)]+ \mathbb{E} [\operatorname{Crt}_{N}\left(\rz, (\rho_*+\eps,\8)\right)] \to 0
\end{aligned}
$$
as $N\to\8$, and (\ref{g1}) follows.

Turning to (\ref{gg1}), for any $\delta>0$, let us define
$$
 \Delta_{\delta}=\left\{x \in \mathbb{R}^{N}:L_{N}\left(\nabla^{2}H_{N}(x)\right)\notin B\left(\si_{c_l,r_l}, \delta\right)\right\},
$$
with $c_l=\mu+\frac{-2D''(0)}{\mu}$ and $r_l=\sqrt{-8D''(0)}$. Using the Kac--Rice formula, the spherical coordinates with $\rho=\frac{\|x\|}{\sqrt{N}}$, and conditioning on $Y=u$,
\begin{align}
\label{aa1}
&\mathbb{E}\left[\#\left\{x \in  \Delta_{\delta}: \nabla H_{N}(x)=0\right\} \right] \\
&=\int_{\mathbb{R}^{N}} p_{\nabla H_{N}(x)}(0) \mathbb{E}\left[\left|\operatorname{det}\left(\nabla^{2} H_{N}(x)\right)\right| \indi_{\Delta_{\delta}}(x)\right]\dd  x\nonumber\\
&=S_{N-1} N^{(N-1) / 2} \int_{0}^{\infty} \int_{-\infty}^{\infty} \mathbb{E}[|\operatorname{det} G|\indi_{\left\{L_{N}\left(G\right)\notin B\left(\si_{c_l,r_l}, \delta\right)\right\}}] p_{\nabla H_{N}(x)}(0) f_{Y}(u) \rho^{N-1} \dd  u \dd  \rho \notag,
\end{align}
where $p_{\nabla H_{N}(x)}(0)$ is the density of $\nabla H_{N}(x)$ at 0, and $f_{Y}(u)$ is the density of $Y$.
We want to show the equivalence of the empirical spectral measure of $G$ and that of $G_{**}$ in the large $N$ limit. Enumerating eigenvalues in the ascending order $\la_1(G)\le \dots\le \la_N(G)$, by the interlacement theorem, for all $j \in\{1, \ldots, N-1\}$,
\begin{align*}
\lambda_{j}(G) \leq \lambda_{j}\left(G_{* *}\right) \leq \lambda_{j+1}(G).
\end{align*}
For any function $h$ with $\|h\|_{L} \leq 1, \|h\|_{\infty} \leq 1$, we deduce
\begin{align*}
    &\left|\frac{1}{N} \sum_{i=1}^{N} h\left(\lambda_{i}\left(G\right)\right)-\frac{1}{N-1} \sum_{i=1}^{N-1} h\left(\lambda_{i}\left(G_{* *}\right)\right)\right| \\
    &\le \frac{1}{N(N-1)} \sum_{i=1}^{N-1} \Big[(N-1)|h(\la_{i+1}(G))-h(\la_i(G_{**}))| + |h(\la_1(G))-h(\la_i(G_{**}))| \Big]\\
    &\le \frac1N \sum_{i=1}^{N-2} [\la_{i+1}(G)-\la_i(G_{**})]+\frac1N\left[|h(\la_N(G))-h(\la_{N-1}(G_{**}))|\right]+ \frac2N \\
    &\le \frac{\la_{N-1}(G_{**})-\la_1(G_{**})+4}{N} \stackrel{d}= \frac{\sqrt{\frac{-4(N-1)D''(0)}N}[\la_{N-1}(\GOE_{N-1})-\la_{ 1}(\GOE_{N-1})]+4}{N}.
\end{align*}
For $N$ large enough, we have
\[
    \left\{L_{N}\left(G\right)\notin B\left(\si_{c_l,r_l}, \delta\right),~ \la_{N-1}(G_{**})-\la_1(G_{**})\le 2K\right\}\subset \left\{L_{N-1}\left(G_{**}\right)\notin B\left(\si_{c_l,r_l}, \frac\delta2\right) \right\}.
\]
Note that
\[
\pz(\la_{N-1}(G_{**})-\la_1(G_{**})\ge 2K)    \le \pz\Big(\la^*(\GOE_{N-1})\ge \frac{K}{\sqrt{-4D''(0)}}\Big),
\]
where $\la^*(\GOE_{N-1})$ is the operator norm of $\GOE_{N-1}$. Using the large deviation estimate \pref{eq:opld}, we may argue as in \cite[Lemmas 4.5 and 4.7]{AZ20} and choose a large $K$ to rewrite \eqref{aa1} as
\begin{align*}
    &\mathbb{E}\left[\#\left\{x \in  \Delta_{\delta}: \nabla H_{N}(x)=0\right\} \right]=S_{N-1} N^{(N-1) / 2} \int_{0}^{\infty} \int_{-\infty}^{\infty}  p_{\nabla H_{N}(x)}(0) f_{Y}(u) \rho^{N-1}  \\
&\qquad \times \mathbb{E}\Big[|\operatorname{det} G|\indi_{\left\{L_{N}\left(G\right)\notin B\left(\si_{c_l,r_l}, \delta\right),~ \la_{N-1}(G_{**})-\la_1(G_{**})\le 2K\right\}}\Big] \dd  u \dd  \rho + o(1)\\
&\le S_{N-1} N^{(N-1) / 2} \int_{0}^{\infty} \int_{-\infty}^{\infty}  p_{\nabla H_{N}(x)}(0) f_{Y}(u) \rho^{N-1} \\
&\qquad \times \mathbb{E}\Big[|\operatorname{det} G|\indi_{\left\{L_{N-1}\left(G_{**}\right)\notin B\left(\si_{c_l,r_l}, \frac\delta2\right) \right\}}\Big] \dd  u \dd  \rho + o(1).
\end{align*}
Now, let us show that
\begin{align}
\label{aa2}
&\lim _{N \rightarrow \infty}S_{N-1} N^{(N-1) / 2} \int_{0}^{\infty} \int_{-\infty}^{\infty} \mathbb{E}[|\operatorname{det} G|\indi_{\left\{L_{N-1}\left(G_{* *}\right)\notin B\left(\si_{c_l,r_l}, \frac\delta2\right)\right\}}] \\
&\quad\quad\quad\quad\quad\quad\quad\quad\quad \times p_{\nabla H_{N}(x)}(0) f_{Y}(u) \rho^{N-1} \dd  u \dd  \rho=0 \notag,
\end{align}
which would imply \eqref{gg1} by Markov's inequality.
To prove \eqref{aa2}, we first observe the following decomposition
\begin{align}
\label{ab1}
&S_{N-1} N^{(N-1) / 2} \int_{0}^{\infty} \int_{-\infty}^{\infty} \mathbb{E}[|\operatorname{det} G|\indi_{\left\{L_{N-1}\left(G_{* *}\right)\notin B\left(\si_{c_l,r_l}, \frac\delta2\right)\right\}}] p_{\nabla H_{N}(x)}(0) f_{Y}(u) \rho^{N-1} \dd  u \dd  \rho \\
&=S_{N-1} N^{(N-1) / 2} \int_{0}^{\infty} \int_{-\infty}^{\infty} \mathbb{E}[|\operatorname{det} G|\indi_{\left\{L_{N-1}\left(G_{* *}\right)\notin B\left(\si_{c_l,r_l},\frac\delta2\right),|z_{3}^{\prime}-y_*|\leq\de' \right\}}]\nonumber \\
&\quad\quad\quad\quad\quad\quad\quad\quad\quad\times p_{\nabla H_{N}(x)}(0) f_{Y}(u) \rho^{N-1} \dd  u \dd  \rho\nonumber \\
&\quad+S_{N-1} N^{(N-1) / 2} \int_{0}^{\infty} \int_{-\infty}^{\infty} \mathbb{E}[|\operatorname{det} G|\indi_{\left\{L_{N-1}\left(G_{* *}\right)\notin B\left(\si_{c_l,r_l}, \frac\delta2\right),|z_{3}^{\prime}-y_*|> \de'\right\}}]\nonumber \\
&\quad\quad\quad\quad\quad\quad\quad\quad\quad\times p_{\nabla H_{N}(x)}(0) f_{Y}(u) \rho^{N-1} \dd  u \dd  \rho
\nonumber\\
&=:\sfK_N^1+\sfK_N^2 \notag,
\end{align}
where $\de'=\frac{\delta}{4\sqrt{-4D''(0)}}$. For $\sfK_N^2$, conditioning on $z_{3}^{\prime}=y$, we find
\begin{align}
\label{ab2}
& \sfK_N^2 \leq S_{N-1} N^{(N-1) / 2} \int_{0}^{\infty} \int_{-\infty}^{\infty}\left(\int_{-\infty}^{y_*-\de'}+\int_{y_*+\de'}^{\infty}\right) \mathbb{E}[|\operatorname{det} G|] \\&\quad\quad\quad\quad\quad\quad\quad\quad\quad \times g_{z_{3}^{\prime}}(y) p_{\nabla H_{N}(x)}(0) f_{Y}(u) \rho^{N-1}\dd  y \dd  u \dd  \rho \notag,
\end{align}
where $g_{z_{3}^{\prime}}(y)$ is the p.d.f.~of $z_{3}^{\prime}$ and we have written $\ez[|\det G|]$ for $\ez[|\det G| | z_3'=y]$, which will be our convention in the following. By the remark after Theorem \ref{x-5}, we may slightly modify the proof of \cite[Theorem 1.2]{AZ20} and deduce
\begin{align}
\label{ab4}
\limsup_{N \rightarrow \infty} \frac{1}{N} \log \sfK_N^2 \le \frac{1}{2} \log \left[-4 D^{\prime \prime}(0)\right]-\frac{1}{2} \log D^{\prime}(0)+\frac{1}{2}+\sup _{(\rho, u, y) \in F} \psi_{*}(\rho, u, y),
\end{align}
where $F=\left\{(\rho, u, y): \rho \in\left(0, \infty\right), u \in \mathbb{R}, y \in \left[y_*-\de',y_*+\de'\right]^{c} \right\}$, and the function $\psi_{*}$ is given in (\ref{ab3}).
Due to Lemma \ref{x-6}, the value of the right-hand side of (\ref{ab4}) is strictly smaller than $0$, which directly implies
\begin{align}
\label{ab5}
\lim _{N \rightarrow \infty} \sfK_N^2=0.
\end{align}
For the first term, since $G_{* *}=\sqrt{-4 D^{\prime \prime}(0)}\left(\sqrt{\frac{N-1}{N}} \GOE_{N-1}-z_{3}^{\prime} I_{N-1}\right)$, on the event $\{|z_{3}^{\prime}-y_*|\leq \de'\}$, we can choose $N$ large enough so that
$$\left\{L_{N-1}\left(G_{* *}\right)\notin B\left(\si_{c_l,r_l}, \frac\delta2\right)\right\} \subset \left\{L_{N-1}\left(\GOE_{N-1}\right)\notin B\left(\si_{\rm{sc}}, \de'\right)\right\}.$$
It follows that
\begin{align}
\label{ab6}
& \sfK_N^1 \leq S_{N-1} N^{(N-1) / 2} \int_{0}^{\infty} \int_{-\infty}^{\infty} \mathbb{E}[|\operatorname{det} G|\indi_{\left\{L_{N-1}\left(\GOE_{N-1}\right)\notin B\left(\si_{\rm{sc}}, \de'\right)\right\}}] \\
&\quad\quad\quad\quad\quad\quad\quad\quad\quad \times p_{\nabla H_{N}(x)}(0) f_{Y}(u) \rho^{N-1} \dd  u \dd  \rho \notag.
\end{align}
Combining \eqref{sk6} and (\ref{ab6}) gives
\begin{align}
\label{aa3}
\sfK_N^1 &\leq S_{N-1} N^{(N-1) / 2} \int_{0}^{\infty} \int_{-\infty}^{\infty} \mathbb{E}[|\operatorname{det} G|\indi_{\left\{L_{N-1}\left(\GOE_{N-1}\right)\notin B\left(\si_{\rm{sc}}, \de' \right)\right\}}] \\
&\quad\quad\quad\quad\quad\quad\quad\quad\quad \times p_{\nabla H_{N}(x)}(0) f_{Y}(u) \rho^{N-1} \dd  u \dd  \rho\nonumber \\
&\leq S_{N-1} N^{(N-1) / 2} \int_{0}^{\infty}\int_{-\infty}^{\infty}p_{\nabla H_{N}(x)}(0) f_{Y}(u) \rho^{N-1} \nonumber\\
&\quad\times\mathbb{E}\left[\Big|\left[-4 D^{\prime \prime}(0)\right]^{(N-1) / 2} z_{1}^{\prime} \prod_{j=1}^{N-1}\left(\sqrt{\frac{N-1}{N}}\lambda_{j}-z_{3}^{\prime}\right)\Big|\indi_{\left\{L_{N-1}\left(\GOE_{N-1}\right)\notin B\left(\si_{\rm{sc}}, \de'\right)\right\}}\right] \dd  u \dd  \rho \nonumber\\
&\quad+S_{N-1} N^{(N-1) / 2} \int_{0}^{\infty}\int_{-\infty}^{\infty} p_{\nabla H_{N}(x)}(0) f_{Y}(u) \rho^{N-1}\nonumber\\
&\quad\times\mathbb{E}\left[\Big|\frac{\left[-4 D^{\prime \prime}(0)\right]^{N / 2}}{2 N} \sum_{k=1}^{N-1} Z_{k}^{2} \prod_{j \neq k}^{N-1}\left(\sqrt{\frac{N-1}{N}}\lambda_{j}-z_{3}^{\prime}\right)\Big|\indi_{\left\{L_{N-1}\left(\GOE_{N-1}\right)\notin B\left(\si_{\rm{sc}}, \de'\right)\right\}}\right]  \dd  u \dd  \rho \nonumber\\
&=: S_{N-1} N^{(N-1) / 2}\left(\sfK_{N}^{11} +\sfK_{N}^{12} \right) \notag.
\end{align}
By \cite[Lemmas 4.6 and 4.8]{AZ20}, we have
$$
\limsup _{N \rightarrow \infty} \frac{1}{N} \log \sfK_{N}^{11} =-\infty \text{ and } \limsup _{N \rightarrow \infty} \frac{1}{N} \log \sfK_{N}^{12} =-\infty.
$$
Recalling \pref{eq:snlim}, we find $\lim_{N\to\8} \sfK_N^1=0$. Together with \eqref{ab5}, we have proved \eqref{aa2}.


Let us proceed to prove \eqref{zx}. Recall that $c_l=\mu+\frac{-2D''(0)}{\mu}$ and $r_l=\sqrt{-8D''(0)}$. For any $\delta>0$, we set
$$
\Theta_{\delta}=\left\{x \in \mathbb{R}^{N}:\left|\lambda_{\min }\left(\nabla^{2} H_{N}(x)\right)+r_l-c_l\right|>\sqrt{-4D^{\prime\prime}}\delta\right\}.
$$
Clearly, $\Theta_{\delta}=\Theta_{\delta}^1\cup \Theta_{\delta}^2$, where
\begin{align}
\Theta_{\delta}^1&=\left\{x \in \mathbb{R}^{N}:\lambda_{\min }\left(\nabla^{2} H_{N}(x)\right)+r_l-c_l>\sqrt{-4D^{\prime\prime}}\delta\right\},\nonumber\\
\Theta_{\delta}^2&=\left\{x \in \mathbb{R}^{N}:\lambda_{\min }\left(\nabla^{2} H_{N}(x)\right)+r_l-c_l<-\sqrt{-4D^{\prime\prime}}\delta\right\}.\nonumber
\end{align}
We will show that both $\mathbb{E}\left[\#\left\{x \in \Theta_{\delta}^1: \nabla H_{N}(x)=0\right\}\right]$ and $\mathbb{E}\left[\#\left\{x \in \Theta_{\delta}^2: \nabla H_{N}(x)=0\right\}\right]$ converge to $0$ as $N\to \8$, which would complete the proof of \eqref{zx}.

We first deal with $\Theta_\de^1$. Arguing as in (\ref{aa1}), and using the same notations as
there, we find
\begin{align}
\label{sk3}
&\mathbb{E}\left[\#\left\{x \in \Theta_{\delta}^1: \nabla H_{N}(x)=0\right\} \right] \\
&=S_{N-1} N^{(N-1) / 2} \int_{0}^{\infty} \int_{-\infty}^{\infty} \mathbb{E}[|\operatorname{det} G|\indi_{\left\{\lambda_{\min }\left(G\right)+r_l-c_l>\sqrt{-4D^{\prime\prime}}\delta\right\}}]]\nonumber \\
&\quad\quad\quad\quad\quad\quad\quad\quad\quad \times p_{\nabla H_{N}(x)}(0) f_{Y}(u) \rho^{N-1} \dd  u \dd  \rho\nonumber \\
&=S_{N-1} N^{(N-1) / 2} \int_{0}^{\infty} \int_{-\infty}^{\infty} \mathbb{E}[|\operatorname{det} G|\indi_{\left\{\lambda_{\min }\left(G\right)+r_l-c_l>\sqrt{-4D^{\prime\prime}}\delta,|z_{3}^{\prime}-y_*|\leq\frac{\delta}{2}\right\}}]\nonumber \\
&\quad\quad\quad\quad\quad\quad\quad\quad\quad \times p_{\nabla H_{N}(x)}(0) f_{Y}(u) \rho^{N-1} \dd  u \dd  \rho\nonumber \\
&\ \ +S_{N-1} N^{(N-1) / 2} \int_{0}^{\infty} \int_{-\infty}^{\infty} \mathbb{E}[|\operatorname{det} G|\indi_{\left\{\lambda_{\min }\left(G\right)+r_l-c_l>\sqrt{-4D^{\prime\prime}}\delta,|z_{3}^{\prime}-y_*|>\frac{\delta}{2}\right\}}]\nonumber \\
&\quad\quad\quad\quad\quad\quad\quad\quad\quad \times p_{\nabla H_{N}(x)}(0) f_{Y}(u) \rho^{N-1} \dd  u \dd  \rho
\nonumber\\
&=:\sfH_{N}^{1} +\sfH_{N}^{2} \nonumber ,
\end{align}
where $y_{*}=-\frac{\mu}{\sqrt{-4 D^{\prime \prime}(0)}}-\frac{\sqrt{-D^{\prime \prime}(0)}}{\mu}$ is defined in (\ref{eq:ruy}). For $\sfH_{N}^{2} $, conditioning on $z_{3}^{\prime}=y$, we find
\begin{align}
\label{sk1}
&\sfH_{N}^{2} \leq S_{N-1} N^{(N-1) / 2} \int_{0}^{\infty} \int_{-\infty}^{\infty}\left(\int_{-\infty}^{y_*-\frac{\delta}{2}}+\int_{y_*+\frac{\delta}{2}}^{\infty}\right) \mathbb{E}[|\operatorname{det} G|] \\
&\quad\quad\quad\quad\quad\quad\quad\quad\quad \times g_{z_{3}^{\prime}}(y) p_{\nabla H_{N}(x)}(0) f_{Y}(u) \rho^{N-1}\dd  y \dd  u \dd  \rho \nonumber,
\end{align}
where $g_{z_{3}^{\prime}}(y)$ is the p.d.f.~of $z_{3}^{\prime}$. By the same argument as for (\ref{ab4}) and (\ref{ab5}), we deduce that
\begin{align}
\label{sk2}
\lim _{N \rightarrow \infty} \sfH_{N}^{2} =0.
\end{align}
Due to the interlacement theorem, we have
$
\lambda_{\min}(G) \leq \lambda_{\min}\left(G_{* *}\right),
$
which implies that
\begin{align}
\label{sk4}
&\sfH_{N}^{1} \leq S_{N-1} N^{(N-1) / 2} \int_{0}^{\infty} \int_{-\infty}^{\infty} \mathbb{E}[|\operatorname{det} G|\indi_{\left\{\lambda_{\min }\left(G_{* *}\right)+r_l-c_l>\sqrt{-4D^{\prime\prime}}\delta,|z_{3}^{\prime}-y_*|\leq\frac{\delta}{2}\right\}}] \\
&\quad\quad\quad\quad\quad\quad\quad\quad\quad\times p_{\nabla H_{N}(x)}(0) f_{Y}(u) \rho^{N-1} \dd  u \dd  \rho \nonumber.
\end{align}
Recalling that $G_{* *}=\sqrt{-4 D^{\prime \prime}(0)}\left(\sqrt{\frac{N-1}{N}} \GOE_{N-1}-z_{3}^{\prime} I_{N-1}\right)$, when $N$ is large enough, if $|z_{3}^{\prime}-y_*|\leq \frac{\delta}{2}$, then we have
$$\Big\{\lambda_{\min }\left(G_{* *}\right)+r_l-c_l>\sqrt{-4D^{\prime\prime}}\delta\Big\}\subset \Big\{\lambda_{\min }\left(\GOE_{N-1}\right)+\sqrt{2}>\frac{\delta}{4}\Big\}.$$
Together with (\ref{sk6}), we deduce
\begin{align}
    \label{sk7}
&\sfH_{N}^{1} \leq S_{N-1} N^{(N-1) / 2} \int_{0}^{\infty} \int_{-\infty}^{\infty} \mathbb{E}[|\operatorname{det} G|\indi_{\left\{\lambda_{\min }\left(\GOE_{N-1}\right)+\sqrt{2}>\frac{\delta}{4}\right\}}] \notag\\
&\quad\quad \times p_{\nabla H_{N}(x)}(0) f_{Y}(u) \rho^{N-1} \dd  u \dd  \rho\\
&\leq S_{N-1} N^{(N-1) / 2} \int_{0}^{\infty}\int_{-\infty}^{\infty} \mathbb{E}\Big[\Big|\left[-4 D^{\prime \prime}(0)\right]^{(N-1) / 2} z_{1}^{\prime} \prod_{j=1}^{N-1}\Big(\sqrt{\frac{N-1}{N}}\lambda_{j}-z_{3}^{\prime}\Big) \nonumber \\
&\quad-\frac{\left[-4 D^{\prime \prime}(0)\right]^{N / 2}}{2 N} \sum_{k=1}^{N-1} Z_{k}^{2} \prod_{j \neq k}^{N-1}\Big(\sqrt{\frac{N-1}{N}}\lambda_{j}-z_{3}^{\prime}\Big)\Big|\indi_{\left\{\lambda_{\min }\left(\GOE_{N-1}\right)+\sqrt{2}>\frac{\delta}{4}\right\}}\Big]\notag\\
&\quad\quad \times p_{\nabla H_{N}(x)}(0) f_{Y}(u) \rho^{N-1} \dd  u \dd  \rho \nonumber\\
&\leq S_{N-1} N^{(N-1) / 2} \int_{0}^{\infty}\int_{-\infty}^{\infty} \mathbb{E}\Big[\Big|\left[-4 D^{\prime \prime}(0)\right]^{(N-1) / 2} z_{1}^{\prime} \prod_{j=1}^{N-1}\Big(\sqrt{\frac{N-1}{N}}\lambda_{j}-z_{3}^{\prime}\Big)\Big|\nonumber\\
&\quad\quad\times  \indi_{\left\{\lambda_{\min }\left(\GOE_{N-1}\right)+\sqrt{2}>\frac{\delta}{4}\right\}}\Big] p_{\nabla H_{N}(x)}(0) f_{Y}(u) \rho^{N-1} \dd  u \dd  \rho \nonumber\\
&\quad+S_{N-1} N^{(N-1) / 2} \int_{0}^{\infty}\int_{-\infty}^{\infty} \mathbb{E}\Big[\Big|\frac{\left[-4 D^{\prime \prime}(0)\right]^{N / 2}}{2 N} \sum_{k=1}^{N-1} Z_{k}^{2} \prod_{j \neq k}^{N-1}\Big(\sqrt{\frac{N-1}{N}}\lambda_{j}-z_{3}^{\prime}\Big)\Big|\nonumber\\
&\quad\quad\times \indi_{\left\{\lambda_{\min }\left(\GOE_{N-1}\right)+\sqrt{2}>\frac{\delta}{4}\right\}}\Big] p_{\nabla H_{N}(x)}(0) f_{Y}(u) \rho^{N-1} \dd  u \dd  \rho \nonumber\\
&=: S_{N-1} N^{(N-1) / 2}\left(\sfH_{N}^{11} +\sfH_{N}^{12} \right) \nonumber.
\end{align}
The LDP for the empirical measures of GOE matrices \eqref{ac2} implies that for any $\delta>0$, there exists a large $N_0$ and a small $c>0$ such that for all $N>N_0$,
$$\mathbb{P}\left(\lambda_{\min}\left(\GOE_{N}\right)+\sqrt{2}>\delta\right) \leq e^{-cN^{2}}.
$$
Therefore, following the same argument as for \cite[Lemmas 4.6 and 4.8]{AZ20}, we have
$$
\limsup _{N \rightarrow \infty} \frac{1}{N} \log \sfH_{N}^{11} =-\infty \text{ and } \limsup _{N \rightarrow \infty} \frac{1}{N} \log \sfH_{N}^{12} =-\infty,
$$
which yields that $\lim _{N \rightarrow \infty} \sfH_{N}^{1} =0$. Combining with (\ref{sk2}), this completes the proof of
\begin{align}
\label{519}
\lim _{N \rightarrow \infty}\mathbb{E}\left[\#\left\{x \in \Theta_{\delta}^1: \nabla H_{N}(x)=0\right\}\right]=0.
\end{align}
It remains to consider $\Theta_\de^2$. This requires harder work. Recall the representation of Hessian as in \pref{eq:gu} and $(\rho_*,u_*,y_*)$ as in Lemma \ref{x-6}. For $\mu>\sqrt{-2D''(0)}$, we may choose $\de>0$ small enough so that $\rho_*-\de>0, y_*+\de<-\sqrt2$. As in \cite[(4.8)]{AZ20}, we have the conditional distribution of $z_1'$ given $z_3'=y$
\begin{align}\label{eq:z13con0}
    z_1'| \{z_3'=y\} \sim N \Big (\bar \sfa, \frac{\sfb^2}{N}\Big),
\end{align}
where
    \begin{align*}
     \bar \sfa &= \bar \sfa(\rho,u,y)
     =\frac{-2D''(0)\al\rho^2( u-\frac{\mu\rho^2}{2}+\frac{\mu D'(\rho^2) \rho^2}{D'(0) } )}{(-2D''(0)-\bt^2) \sqrt{D(\rho^2)-\frac{D'(\rho^2)^2 \rho^2}{D'(0)} }}
       +\frac{\al\bt\rho^2 \mu }{-2D''(0)-\bt^2}\notag \\
       &\qquad  -\frac{(-2D''(0)-\bt^2-\al\bt\rho^2)\sqrt{-4D''(0)} y}{-2D''(0)-\bt^2},\\
     \frac{\sfb^2}{N}&=\frac{\sfb(\rho)^2}{N} 
     =\frac{-4D''(0)}{N} +\frac{2D''(0)\al^2\rho^4}{N(-2D''(0)-\bt^2)}.
    \end{align*}
It follows that
\begin{align*}
    G(\rho,u,y)\stackrel{d}= \begin{pmatrix}
        \sqrt{-2D''(0)}\tilde z_1 &\xi^\sfT\\
        \xi& \sqrt{-4D''(0)} \Big(\sqrt{\frac{N-1}{N}}\GOE_{N-1}- y I_{N-1}\Big)
    \end{pmatrix},
\end{align*}
where $\sqrt{-2D''(0)}\tilde z_1\sim N \big (\bar \sfa, \frac{\sfb^2}{N}\big)$ and we have written $G$ for the conditional distribution $G|\{z_3'=y\}$ for simplicity. Direct computation (see \cite[Example 1]{AZ22}) yields
\begin{align}\label{eq:abbd}
    \bar\sfa(\rho_*,u_*,y_*)=-\sqrt{-4D''(0)}y_*, \ \ \sfb(\rho_*)^2>0.
\end{align}
By continuity, we may choose $\de$ small enough so that
\[
    \Big|-\sqrt2 y -\frac{\bar \sfa(\rho,u,y)}{\sqrt{-2D''(0)}}\Big| \le \frac12
\]
for all $(\rho,u,y)\in [\rho_*-\de,\rho_*+\de]\times [u_*-\de, u_*+\de]\times [y_*-\de,y_*+\de]$.  Using the Gaussian tail we find for $t\ge 1$,
\begin{align}\label{eq:z1tail}
    &\pz\Big(\tilde z_1+\sqrt{2}y\le -t\Big)= \pz\Big(\tilde z_1- \frac{\bar \sfa(\rho,u,y)}{\sqrt{-2D''(0)}}\le  -\sqrt2 y -\frac{\bar \sfa(\rho,u,y)}{\sqrt{-2D''(0)}}-t\Big) \\
    &\le \pz\Big(\sqrt{\frac{\sfb(\rho)^2}{-2D''(0) N}}z_1\le -\frac{t}2\Big)\le e^{-\frac{-2D''(0)Nt^2}{8\sfb_*^2}} \nonumber,
\end{align}
where $\sfb_*=\sup_{\rho_*-\de\le \rho\le \rho_*+\de} \sfb(\rho)$ and $z_1\sim N(0,1)$. Note that $\ez[\sqrt{-2D''(0)}\tilde z_1]=\bar \sfa$, which is close to $-\sqrt{-4D''(0)}y_*$ in a small neighborhood of $(\rho_*,u_*,y_*)$ and $-\sqrt{-4D''(0)}y_*=c_l$ is the center of the limiting semicircle distribution. The next result shows that the first element of $G$ does not change the limit of the smallest eigenvalue at deterministic $(\rho,u,y)$ in a small neighborhood of the maximizer $(\rho_*,u_*,y_*)$.
\begin{lemma}
\label{51}
    For any small $\eps>0$ and small enough $\de>0$, there is a $\de_0>0$ such that for large $N$
    \[
    \pz\Big(\la_{\min}(G(\rho,u,y))\le c_l-r_l-\eps\Big)\le e^{-\de_0 N},
    \]
    uniformly for $(\rho,u,y)\in[\rho_*-\de,\rho_*+\de]\times [u_*-\de, u_*+\de]\times [y_*-\de,y_*+\de]$, where $c_l=\mu+\frac{-2D''(0)}{\mu}$ and $r_l=\sqrt{-8D''(0)}$.
\end{lemma}

\begin{proof}
    Let $V_1$ be an $(N-1)\times (N-1)$ random orthogonal matrix such that $V_1\xi=(\|\xi\|,0,...,0)^\sfT$. By the rotational invariance of GOE matrices,
    \begin{align*}
        &\begin{pmatrix}
            1&0\\
            0&V_1
        \end{pmatrix} (-G)
        \begin{pmatrix}
            1&0\\
            0&V_1^\sfT
        \end{pmatrix}
        \\
        &\stackrel{d}=\sqrt{\frac{{-2D''(0)}}{ N}} \begin{pmatrix}
            -\sqrt{N}(\tilde z_1+\sqrt2 y) & (\chi_{N-1},0,\dots,0)\\
            \begin{pmatrix}
                {\chi_{N-1}}\\
                0\\
                \vdots\\
                0
            \end{pmatrix}& \sqrt{{2(N-1)}}\GOE_{N-1}
        \end{pmatrix}
        +\sqrt{-4D''(0)} y I_N,
    \end{align*}
    where $\chi_{N-1}$ is an independent random variable following the $\chi$ distribution with $N-1$ degrees of freedom. Repeating the above operations with independent random orthogonal matrices $V_2,...,V_{N-1}$, we find a variant of the Dumitriu--Edelman tridiagonal representation \cite{DE02}
    \begin{align*}
        V(-G)V^\sfT &\stackrel{d}= \sqrt{\frac{-2D''(0)}{N}} \begin{pmatrix}
            -\sqrt{N}(\tilde z_1+\sqrt2 y)&\chi_{N-1}& &\\
            \chi_{N-1}& \sqrt2 \eta_2& \chi_{N-2}& &\\
             & \ddots &\ddots&\ddots&\\
             &&& \chi_1& \sqrt2\eta_{N}
        \end{pmatrix}
        +\sqrt{-4D''(0)} y I_N \notag\\
        &=: \sqrt{\frac{-2D''(0)}{N}} W +\sqrt{-4D''(0)} y I_N,
    \end{align*}
    where $V$ is an independent random orthogonal matrix, $\chi_i$ is an independent $\chi$ random variable with $i$ degrees of freedom, and $\eta_2,...,\eta_N$ are i.i.d.~standard normal random variables. We follow an idea from \cite{BV13} to estimate the largest eigenvalue of $W$. Observe that
    \begin{align*}
        W&=-\begin{pmatrix}
            0&\sqrt{\chi_{N-1}}& &&\\
            &-\sqrt{\chi_{N-1}}&\sqrt{\chi_{N-2}}&&\\
            &&&\ddots&\\
            &&&&-\sqrt{\chi_1}
        \end{pmatrix}
    \begin{pmatrix}
        0&\sqrt{\chi_{N-1}}& &&\\
        &-\sqrt{\chi_{N-1}}&\sqrt{\chi_{N-2}}&&\\
        &&&\ddots&\\
        &&&&-\sqrt{\chi_1}
    \end{pmatrix}^\sfT\\
    & \quad   +\mathrm{diag}( \chi_{N-1}-\sqrt{N}(\tilde z_1+\sqrt2 y), \sqrt2 \eta_2 +\chi_{N-1}+\chi_{N-2},...,\sqrt2\eta_N+\chi_1).
\end{align*}
Writing $\chi_0=0$, it follows that
\begin{align*}
    \la_{\max}(W)&\le \max\{ \chi_{N-1}-\sqrt{N}(\tilde z_1+\sqrt2 y), ~\max_{2 \le i\le N}\{\sqrt2\eta_i+\chi_{N-i+1}+\chi_{N-i}\}\}.
\end{align*}
Using the  concentration inequality for the chi-square distribution \cite{ML00}, we have
\begin{align*}
    \pz\Big(\chi_{N-i}\ge (1+\frac\eps3)\sqrt{N}\Big) \le \pz\Big(\chi_{N}^2-N\ge \frac{\sqrt2 \eps N}{3}+\frac{\eps^2 N}{9}\Big)\le e^{-\frac{\eps^2 N}{18}}.
\end{align*}
Together with the Gaussian tail, we deduce
\begin{align*}
    &\pz(\sqrt2\eta_i+\chi_{N-i+1}+\chi_{N-i}\ge (2+\eps)\sqrt{N})\\
    & \le \pz\Big(\sqrt2 \eta_i\ge\frac{\eps\sqrt{N}}{3}\Big)+\pz\Big(\chi_{N-i+1}\ge (1+\frac\eps3)\sqrt{N}\Big)+\pz\Big(\chi_{N-i}\ge (1+\frac\eps3)\sqrt{N}\Big)\\
    &\le e^{-\frac1{36}N\eps^2}+2e^{-\frac{\eps^2 N}{18}}\le 3 e^{-\frac1{36}N\eps^2}.
\end{align*}
Similarly, using \pref{eq:z1tail} we find
\begin{align*}
    &\pz(\chi_{N-1}-\sqrt{N}(\tilde z_1+\sqrt2 y)\ge  (2+\eps)\sqrt{N})\le \pz(\chi_{N-1}\ge (1+\eps)\sqrt{N})+\pz(\tilde z_1+\sqrt2y\le -1) \\
    &\le  e^{-\frac{\eps^2N}2} +e^{-\frac{-2D''(0)N}{8\sfb_*^2}}\le 2e^{-\frac{N\eps^2}{2}},
\end{align*}
where we have chosen $\de,\eps$ small enough so that
$
\frac{-2D''(0)}{4\sfb_*^2}> \eps^2.
$
Combining all these estimates with the union bound yields
\[
    \pz\Big(\la_{\max}(W)\ge (2+\eps)\sqrt{N}\Big)\le 3N e^{-\frac1{36}N\eps^2}+2e^{-\frac{N\eps^2}{2}}\le e^{-\frac1{72}N\eps^2}
\]
for $N$ large enough. By choosing $\de$ small we may assume
\[
    |c_l +\sqrt{-4D''(0)}y|\le \frac{\eps}{2}
\]
for $y\in[y_*-\de,y_*+\de]$. Hence,
\begin{align*}
    &\pz\Big(\la_{\min}(G)\le c_l-r_l-\eps \Big)\\
    &= \pz\Big( \la_{\min}(G)+\sqrt{-4D''(0)}y \le c_l-r_l +\sqrt{-4D''(0)}y  -\eps\Big)\\
    &\le \pz \Big(-\sqrt{\frac{-2D''(0)}{N}}\la_{\max}(W) \le -r_l  -\frac{\eps}2\Big)\\
    &\le \pz \Big(\la_{\max}(W)\ge 2\sqrt{N}+\frac{\eps \sqrt{N}}{-4D''(0)}\Big)\le e^{-\frac{N\eps^2}{1200 \left(D''(0)\right)^2}}.
\end{align*}
It suffices to take $\de_0=\frac{\eps^2}{1200 \left(D''(0)\right)^2}$.
\end{proof}

Turning to the proof of \eqref{zx}, we note that by the Cauchy--Schwarz inequality
\begin{align*}
    &\ez[|\det G| \indi_{\{\la_{\min}(G)\le c_l-r_l-\eps\}}|]\leq (\ez[|\det G |^2])^{1/2} \pz(\la_{\min}(G)\le c_l-r_l-\eps)^{1/2}.
\end{align*}
Recall the construction of matrix $G$ in (\ref{eq:gu}) and the conditional distribution \eqref{eq:z13con0}.
Using the Schur complement and conditioning, we have 
\begin{align*}
    \ez[|\det G |^2]&=  \ez\Big[ \ez[|\det G_{**}|^2 | z_1'-\xi^\sfT G_{**}^{-1} \xi |^2 |G_{**},\xi] \Big]\\
    &=\ez\Big[ (\sfa_N^2+\frac{\sfb^2}{N})|\det G_{**}|^2\Big],
\end{align*}
where writing $G_{**}=G_{**}|\{z_3'=y\}=\sqrt{-4 D^{\prime \prime}(0)}\left(\sqrt{\frac{N-1}{N}} \GOE_{N-1}-y I_{N-1}\right)$,
$$\sfa_N=\sfa_N(\rho,u,y)=\bar \sfa(\rho,u,y)-\xi^\sfT G_{**}^{-1}\xi.$$ 
As in \eqref{sk6} we may apply \eqref{eq:z13con0} to rewrite
\begin{align}\label{eq:abdef}
  (\sfa_N, \det G_{**})&\stackrel{d}= \Big( \bar \sfa (\rho,u,y)-\frac{\sqrt{-D''(0)}}{N}\sum_{i=1}^{N-1} \frac{Z_i^2}{(\frac{N-1}{N})^{1/2}\la_i-y}, \\
  &\qquad [-4D''(0)]^{(N-1)/2}\prod_{j=1}^{N-1} \Big[(\frac{N-1}{N})^{1/2}\la_j-y\Big] \Big),\notag
\end{align}
where $\la_1, \dots, \la_{N-1}$ are eigenvalues of $\GOE_{N-1}$ that are independent of i.i.d.~standard normal random variables $Z_1,\dots,Z_{N-1}$.
For $(\rho,u,y)\in[\rho_*-\de,\rho_*+\de]\times [u_*-\de, u_*+\de]\times [y_*-\de,y_*+\de]$, by continuity, we may choose $\de$ small enough and $N$ large enough so that
\[
|\bar \sfa(\rho,u,y)|\le C_{\mu, D,\de}, \ \ \sfb^2\le N.
\]
From the Cauchy--Schwarz inequality we know
\[
\ez_{Z}\Big(\frac{1}{N} \sum_{i=1}^{N-1} \frac{Z_i^2}{|(\frac{N-1}{N})^{1/2}\la_i-y| }\Big)^2\le \frac1N \sum_{i=1}^{N-1} \frac{\ez(Z_i^4)}{|(\frac{N-1}{N})^{1/2}\la_i-y|^2 },
\]
where $\ez_Z$ only takes expectation of $Z_i$'s.
Conditioning on the eigenvalues again, plugging in the representation of $G_{**}$, we deduce for large $N$
\begin{align*}
    &\ez[|\det G |^2] \le 2\ez\Big[ \frac{-3D''(0)}{N} \sum_{i=1}^{N-1} \frac{1}{|(\frac{N-1}{N})^{1/2}\la_i-y|^2 } |\det G_{**}|^2\Big]\\&\quad\quad\quad\quad\quad\quad+(2C_{\mu, D,\de}^2+1)\ez[|\det G_{**}|^2]\\
    &\le \frac{2[-4D''(0)]^{N}}N\sum_{i=1}^{N-1} \ez\Big[\prod_{j\neq i} |(\frac{N-1}{N})^{1/2}\la_j-y |^2 \Big] \\
    &\quad + 3C_{\mu, D,\de}^2[-4D''(0)]^{N-1}\ez\Big[\prod_{i=1}^{N-1}|(\frac{N-1}{N})^{1/2}\la_j-y |^2\Big].
\end{align*}

\begin{lemma}\label{le:det2}
    For any $\de>0$ and $\sqrt2+\de< x< \frac1\de$, we have for $N$ large enough
    \begin{align*}
        \ez\Big[\prod_{j=1}^{N-1} |(\frac{N-1}{N})^{1/2}\la_j+x |^2 \Big]&\le e^{ N\de} \ez\Big[\prod_{j=1}^{N-1}|(\frac{N-1}{N})^{1/2}\la_j+x |\Big]^2,\\
        \ez\Big[\prod_{j=1,j\neq i}^{N-1} |(\frac{N-1}{N})^{1/2}\la_j+x |^2 \Big]&\le e^{ N\de} \ez\Big[\prod_{j=1}^{N-1}|(\frac{N-1}{N})^{1/2}\la_j+x |\Big]^2, \ \ i=1,\dots, N-1.
    \end{align*}
\end{lemma}
\begin{proof}
    We only prove the second inequality as the proof for the first one is a simplification of the second. We
    denote by
    $$L(\la_{j\neq i}^{N-1})=\frac1{N-2}\sum_{j=1,j\neq i}^{N-1} \de_{\la_j}$$ the empirical measure of the vector $(\la_1,\dots,\la_{i-1},\la_{i+1},\dots,\la_{N-1})$. 
  In short, the result follows from the same argument as for Lemma \ref{x-13} since the empirical measures $L((\frac{N-1}N)^{1/2}\la_{j\neq i}^{N-1})$ and $L(\la_{i=1}^{N-1})$ are exponentially equivalent. Let us indicate the difference when arguing as in the proof of \cite[Lemma 6.4]{Be22}. As shown in \cite[(A.17)]{ABC13}, $L(\la_{i=1}^{N-1})$ and $L((\frac{N-1}N)^{1/2}\la_{j=1}^{N-1})$ are exponentially equivalent. Similarly, using \cite[(5.3)]{AZ20}, we see that $L((\frac{N-1}N)^{1/2}\la_{j=1}^{N-1})$  and  $L((\frac{N-1}N)^{1/2}\la_{j\neq i}^{N-1})$ are exponentially equivalent. It follows that  $L((\frac{N-1}N)^{1/2}\la_{j\neq i}^{N-1})$ satisfies the same LDP as $L(\la_{i=1}^{N-1})$ with speed $N^2$.

  Let $\log_\kappa x=\log x \indi_{\{\frac1\kappa\le x\le \kappa\}}+ \log \kappa \indi_{\{x>\kappa\}}-\log \kappa \indi_{\{x<1/\kappa\}}$. By taking $\kappa$ and $N$ large enough, we have for $\sqrt2+\de< x< \frac1\de$,
  \begin{align*}
    &\ez\Big[\prod_{j=1,j\neq i}^{N-1} |(\frac{N-1}{N})^{1/2}\la_j+x |^2 \Big]\le 4e^{N\de/4}\exp\Big(2N\int_{-\sqrt2}^{\sqrt2} \frac1{\pi} \log_\kappa |x-\la|\sqrt{2-\la^2} \dd \la\Big)\\
    &\le 8e^{N\de/2}\Big[\ez\Big(\exp\Big( N\int \log_\kappa|x-t | L\Big((\frac{N-1}{N})^{1/2}\la_{j=1}^{N-1}\Big)(\dd t)\Big)\\
    &\quad \times \indi_{\Big\{d\Big(L\Big((\frac{N-1}{N})^{1/2}\la_{j=1}^{N-1}\Big),\si_{\rm sc}\Big) \le \frac{\de}{8\kappa}, \la_{\max}\le \sqrt2 +\de/2\Big\}}\Big)\Big]^2\\
    &\le 8 e^{N\de/2} \ez\Big[\prod_{i=1}^{N-1}|(\frac{N-1}{N})^{1/2}\la_j+x |\Big]^2.
  \end{align*}
  The desired result follows by enlarging the exponent to control the constant.
\end{proof}

Using \pref{le:det2} for small $\de\le\frac{\de_0}2$ where $\de_0$ is from Lemma \ref{51}, when $N$ is large enough, we have uniformly for $(\rho,u,y)\in[\rho_*-\de,\rho_*+\de]\times [u_*-\de, u_*+\de]\times [y_*-\de,y_*+\de],$
\begin{align}
\label{528}
    \ez[|\det G |^2]&\le (-8D''(0) +C_{\mu, D,\de}^2)[-4D''(0)]^{N-1} e^{\frac{\de_0 N}{2} } \ez\Big[\prod_{i=1}^{N-1}|(\frac{N-1}{N})^{1/2}\la_j-y |\Big]^2\\
    &\le (-8D''(0) +C_{\mu, D,\de}^2)e^{\frac{\de_0 N}{2} } \ez[|\det G_{**}|]^2\nonumber\\
    &\le \frac{(-8D''(0) +C_{\mu, D,\de}^2)\pi N e^{\frac{\de_0 N}{2} }}{ 2\sfb_0^2} \ez[|\det G |]^2 \nonumber,
\end{align}
where in the last step we have used the fact that
$
    \sqrt{\frac2\pi}|b|\le \ez|a+bz|
$
for $z\sim N(0,1), a,b\in\rz$ with conditioning and $\sfb_0^2=\inf_{\rho_*-\de\le \rho\le \rho_*+\de} \sfb(\rho)^2>0$.

We are now ready to handle $\mathbb{E}\left[\#\left\{x \in \Theta_{\delta}^2: \nabla H_{N}(x)=0\right\}\right]$. Again arguing as in (\ref{aa1}), we have
\begin{align}
\label{kk}
&\mathbb{E}\left[\#\left\{x \in \Theta_{\delta}^2: \nabla H_{N}(x)=0\right\}\right]\\
&=S_{N-1} N^{(N-1) / 2} \int_{0}^{\infty} \int_{-\infty}^{\infty} \mathbb{E}[|\operatorname{det} G|\indi_{\left\{\lambda_{\min }\left(G\right)+r_l-c_l<-\sqrt{-4D^{\prime\prime}}\delta\right\}}]\nonumber \\
&\quad\quad\quad\quad\quad\quad\quad\quad\quad \times p_{\nabla H_{N}(x)}(0) f_{Y}(u) \rho^{N-1} \dd  u \dd  \rho\nonumber \\
&=S_{N-1} N^{(N-1) / 2} \int_{0}^{\infty} \int_{-\infty}^{\infty} \mathbb{E}[|\operatorname{det} G|\indi_{\left\{\lambda_{\min }\left(G\right)+r_l-c_l<-\sqrt{-4D^{\prime\prime}}\delta,|z_{3}^{\prime}-y_*|\leq\delta\right\}}]\nonumber \\
&\quad\quad\quad\quad\quad\quad\quad\quad\quad\times p_{\nabla H_{N}(x)}(0) f_{Y}(u) \rho^{N-1} \dd  u \dd  \rho\nonumber \\
&\quad+S_{N-1} N^{(N-1) / 2} \int_{0}^{\infty} \int_{-\infty}^{\infty} \mathbb{E}[|\operatorname{det} G|\indi_{\left\{\lambda_{\min }\left(G\right)+r_l-c_l<-\sqrt{-4D^{\prime\prime}}\delta,|z_{3}^{\prime}-y_*|>\delta\right\}}]\nonumber \\
&\quad\quad\quad\quad\quad\quad\quad\quad\quad\times p_{\nabla H_{N}(x)}(0) f_{Y}(u) \rho^{N-1} \dd  u \dd  \rho
\nonumber\\
&=:\sfL_{N}^{1} + \sfL_{N}^{2} \nonumber .
\end{align}
By the same argument as for (\ref{sk2}), we know
\begin{align}
\label{kk1}
\lim _{N \rightarrow \infty} \sfL_{N}^{2} =0.
\end{align}
For $\sfL_{N}^{1} $, conditioning on $z_{3}^{\prime}=y$, we make the following decomposition,
\begin{align*}
&\sfL_{N}^{1} =S_{N-1} N^{(N-1) / 2} \int_{0}^{\infty} \int_{-\infty}^{\infty}\int_{y_{*}-\delta}^{y_{*}+\delta} \mathbb{E}[|\operatorname{det} G|\indi_{\left\{\lambda_{\min }\left(G\right)+r_l-c_l<-\sqrt{-4D^{\prime\prime}}\delta\right\}}]\nonumber \\
&\quad\quad\quad\quad\quad\quad\quad\quad\quad\quad\quad\quad\quad\quad\quad\quad\times g_{z_{3}^{\prime}}(y) p_{\nabla H_{N}(x)}(0) f_{Y}(u) \rho^{N-1}\dd  y \dd  u \dd  \rho\nonumber\\
&=S_{N-1} N^{(N-1) / 2} \int_{\rho_{*}-\delta}^{\rho_{*}+\delta} \int_{u_{*}-\delta}^{u_{*}+\delta}\int_{y_{*}-\delta}^{y_{*}+\delta} \mathbb{E}[|\operatorname{det} G|\indi_{\left\{\lambda_{\min }\left(G\right)+r_l-c_l<-\sqrt{-4D^{\prime\prime}}\delta\right\}}]\nonumber \\
&\quad\quad\quad\quad\quad\quad\quad\quad\quad\quad\quad\quad\quad\quad\quad\quad\times g_{z_{3}^{\prime}}(y) p_{\nabla H_{N}(x)}(0) f_{Y}(u) \rho^{N-1}\dd  y \dd  u \dd  \rho\nonumber\\
&\quad+S_{N-1} N^{(N-1) / 2} \int_{[\rho_{*}-\delta,\rho_{*}+\delta]^c}\int_{u_{*}-\delta}^{u_{*}+\delta}\int_{y_{*}-\delta}^{y_{*}+\delta} \mathbb{E}[|\operatorname{det} G|\indi_{\left\{\lambda_{\min }\left(G\right)+r_l-c_l<-\sqrt{-4D^{\prime\prime}}\delta\right\}}]\nonumber \\
&\quad\quad\quad\quad\quad\quad\quad\quad\quad\quad\quad\quad\quad\quad\quad\quad\times g_{z_{3}^{\prime}}(y) p_{\nabla H_{N}(x)}(0) f_{Y}(u) \rho^{N-1}\dd  y \dd  u \dd  \rho\nonumber\\
&\quad+S_{N-1} N^{(N-1) / 2} \int_{\rho_{*}-\delta}^{\rho_{*}+\delta} \int_{[u_{*}-\delta,u_{*}+\delta]^c}\int_{y_{*}-\delta}^{y_{*}+\delta} \mathbb{E}[|\operatorname{det} G|\indi_{\left\{\lambda_{\min }\left(G\right)+r_l-c_l<-\sqrt{-4D^{\prime\prime}}\delta\right\}}]\nonumber \\
&\quad\quad\quad\quad\quad\quad\quad\quad\quad\quad\quad\quad\quad\quad\quad\quad\times g_{z_{3}^{\prime}}(y) p_{\nabla H_{N}(x)}(0) f_{Y}(u) \rho^{N-1}\dd  y \dd  u \dd  \rho\nonumber\\
&\quad+S_{N-1} N^{(N-1) / 2} \int_{[\rho_{*}-\delta,\rho_{*}+\delta]^c} \int_{[u_{*}-\delta,u_{*}+\delta]^c} \int_{y_{*}-\delta}^{y_{*}+\delta} \mathbb{E}[|\operatorname{det} G|\indi_{\left\{\lambda_{\min }\left(G\right)+r_l-c_l<-\sqrt{-4D^{\prime\prime}}\delta\right\}}]\nonumber \\
&\quad\quad\quad\quad\quad\quad\quad\quad\quad\quad\quad\quad\quad\quad\quad\quad\times g_{z_{3}^{\prime}}(y) p_{\nabla H_{N}(x)}(0) f_{Y}(u) \rho^{N-1}\dd  y \dd  u \dd  \rho\nonumber\\
&=:\sfL_{N}^{11} +\sfL_{N}^{12} +\sfL_{N}^{13} +\sfL_{N}^{14},
\end{align*}
where the set  $[\rho_{*}-\delta,\rho_{*}+\delta]^c$ represents $(0,\rho_{*}-\delta)\cup(\rho_{*}+\delta,+\infty)$. 
Similarly to (\ref{ab4}) and (\ref{ab5}), we see all of $\sfL_{N}^{12} $, $\sfL_{N}^{13} $ and $\sfL_{N}^{14} $ converge to $0$ as ${N \rightarrow \infty}$.

It remains to show the convergence of $\sfL_{N}^{11} $.
Using Lemma \ref{51} and (\ref{528}), by the Cauchy--Schwarz inequality, we deduce that
\begin{align*}
\sfL_{N}^{11}&\leq S_{N-1} N^{(N-1) / 2} \int_{\rho_{*}-\delta}^{\rho_{*}+\delta} \int_{u_{*}-\delta}^{u_{*}+\delta}\int_{y_{*}-\delta}^{y_{*}+\delta} (\ez[|\det G |^2])^{1/2} \nonumber \\
&\quad\times\pz(\lambda_{\min }\left(G\right)+r_l-c_l<-\sqrt{-4D^{\prime\prime}}\delta)^{1/2} g_{z_{3}^{\prime}}(y) p_{\nabla H_{N}(x)}(0) f_{Y}(u) \rho^{N-1}\dd  y \dd  u \dd \rho \nonumber \\
&\leq e^{\frac{-\de_0 N}{8} } S_{N-1} N^{(N-1) / 2} \int_{\rho_{*}-\delta}^{\rho_{*}+\delta} \int_{u_{*}-\delta}^{u_{*}+\delta}\int_{y_{*}-\delta}^{y_{*}+\delta} \mathbb{E}[|\operatorname{det} G|]\nonumber \\
&\quad\quad\quad\quad\quad\quad\quad\quad\quad \times g_{z_{3}^{\prime}}(y) p_{\nabla H_{N}(x)}(0) f_{Y}(u) \rho^{N-1}\dd  y \dd  u \dd \rho \notag\\
&\le e^{\frac{-\de_0 N}{8} }  \mathbb{E}\left[\operatorname{Crt}_{N}\left(\mathbb{R}, \mathbb{R}^{N}\right)\right] \to 0
\end{align*}
as $N\to\8$, where the last step follows from Theorem \ref{h-1} (i). We have proved
$
\lim _{N \rightarrow \infty} \sfL_{N}^{1} =0.
$
Combining this with (\ref{kk1}) implies
\begin{align*}
\lim _{N \rightarrow \infty}\mathbb{E}\left[\#\left\{x \in \Theta_{\delta}^2: \nabla H_{N}(x)=0\right\}\right]=0,
\end{align*}
and our proof is complete.
\end{proof}

\section{Elastic manifold}\label{se:em}

 From the physics literature \cite{FLD20}, we know that for any $x\in\Omega$ fixed, the gradient of $\mathcal{H}(\mathbf{u}(x))$
$$
\nabla \mathcal{H}(\mathbf{u}(x))=\sum_{y \in \Omega}\left(\mu_0 I_{L^d}-t_0 \Delta\right)_{x y} \cdot \mathbf{u}(y)+\frac{\partial}{\partial \mathbf{u}(x)} X_N(\mathbf{u}(x), x)
$$ 
is an $N$-dimensional vector, and thus $\nabla \mathcal{H}(\mathbf{u})$ can be regarded as an $NL^d$-dimensional vector. Similarly, the Hessian $\nabla^2 \mathcal{H}(\mathbf{u})$ is an $NL^d\times NL^d$ matrix and its component 
$$
\begin{aligned}
\nabla^2 \mathcal{H}(\mathbf{u})_{ix,jy} 
 =\delta_{i j}\left(\mu_0 I_{L^d}-t_0 \Delta\right)_{x y}+\delta_{x y} \frac{\partial^2}{\partial u_i \partial u_j} X_N(\mathbf{u}(x), x) .
\end{aligned}
$$
Since the rest of this section heavily relies on the remarkable work \cite{BABM21}, we follow the notations from there. Recall the semicircle distribution with variance $t$
$$
\sigma_{0, 2\sqrt{t}}(\mathrm{~d} x)=\frac{\sqrt{4 t-x^2}}{2 \pi t} \indi_{x \in[-2 \sqrt{t}, 2 \sqrt{t}]} \mathrm{d} x.
$$  
For any probability measure $\mu$ compactly supported in $(0,\infty)$ , writing $\ell\left(\mu\right)$ for its left edge, we set
\begin{align*}
\mu_t  =\sigma_{0, 2\sqrt{t}} \boxplus \mu, \quad
\ell_t  =\ell\left(\mu_t\right), 
\end{align*}
for the free convolution of $\mu$ with the semicircle distribution $\sigma_{0, 2\sqrt{t}}$ and its left edge, respectively. 

Unlike the locally isotropic Gaussian random fields, for the elastic manifold model, it seems difficult to show that the expected number of critical points converges to one as $N$ tends to infinity for large potential. This prevents us from obtaining a topology trivialization result for elastic manifold. Fortunately, we know the global minimum exists and is unique almost surely from \pref{se:euminimum}. 


Let us define
\begin{align*}
    F(u, t)&=-\int_{\mathbb{R}} \log (\lambda) \mu(\mathrm{d} \lambda)+\int_{\mathbb{R}} \log |\lambda+u| \mu_t(\lambda) \mathrm{d} \lambda-\frac{u^2}{2 t},\\
    u_t&=-t \int \frac{\mu(\mathrm{d} \lambda)}{\lambda},\quad t_c=\inf \left\{t>0: u_t=\ell_t\right\}.
\end{align*}
From {\cite[Section 5.2]{BABM21}}, we know $$t_c=\left(\int \frac{\mu(\mathrm{d} \lambda)}{\lambda^2}\right)^{-1}.$$
The function $F(u,t)$ has a unique maximizer $-u_t$ for $t<t_c$. From the proof of \cite[Theorem 2.8]{BABM21}, we know $u_t< \ell_t$ for $t<t_c$, which means that $u_t$ is not in the support of $\mu_t$. Thus, the function $F(u, t)$ is strictly concave in a neighborhood of $-u_t$ in this case. We restate \cite[Lemma 5.5]{BABM21} as the following result for later references. 

\begin{lemma}
\label{unimax}
 When $t<t_c$ is fixed, the function $F(u,t)$ attains its maximum at the unique maximizer $-u_t$ and the maximum value $F(-u_t,t)=0$.  
\end{lemma}

Set $J=2 \sqrt{B^{\prime \prime}(0)}$. For any fixed $u \in \mathbb{R}^{L^d}$, we define
\begin{equation*}
 \label{eq:au}
 a(u)=\left(-t_0 \Delta+\operatorname{diag}(u)+\mu_0 I_{L^d}\right) \in \mathbb{R}^{L^d \times L^d}.
\end{equation*}
With $a(u)$ defined above, we set
$$
A_N(u)=a(u) \otimes I_N,
$$
where $\otimes$ denotes the Kronecker product. For every $N$ fixed, let $\left(M_i\right)_{i=1}^{L^d}$ be a sequence of independent $N \times N$ matrices and each has the same distribution as $J$ times a GOE matrix. Write
\begin{align}
\label{WH}
W_N =\sum_{i=1}^{L^d} E_{i i} \otimes M_i, \quad\quad
H_N(u) =A_N(u)+W_N,
\end{align}
where $(E_{i j})_{i,j=1}^{L^d}$ denotes the matrix units. 

From {\cite[Section 4]{BABM21}}, we know that
for any complex number $z \in \mathbb{H}$, the upper half-plane, and $u \in \mathbb{R}^{L^d}$, there exists a unique solution $M(u,z)\in \mathbb{C}^{L^d\times L^d}$ to the equation
\begin{align}\label{M}
    -M^{-1}(u, z)=z \operatorname{Id}-a(u)+J^2\operatorname{diag}[M(u, z)] \quad\quad \text{subject to}\quad {\Im}M(u, z)>\mathbf{0},
\end{align}
where $\mathbf{0}$ is the matrix of all zeros.
Clearly, $M(u,z)$ is symmetric. Moreover, from the above equation, we see that $M(u, z)\neq M(u', z)$ for $u\neq u'$ and any $z\in \mathbb{H}$ fixed.
 Let $\mu_{\infty}(u)$ be the probability measure with Stieltjes transform at $z$ given by
 $$
\int \frac{\mu_{\infty}(u,\mathrm{d} s)}{s-z}=\frac{1}{L^d} \operatorname{Tr}\left(M(u, z)\right).
$$

\begin{proof}[Proof of Theorem \ref{Elas}]
For any $\eps>0$, we define the random set
\begin{align*}
 \Xi_{\varepsilon}&=\left\{\mathbf{u}:\hat{\mu}_{\nabla^2 \mathcal{H}(\mathbf{u})} \notin B\left(\tau_{-u^*}(\sigma_{0, 4\sqrt{B''}} \boxplus \hat{\mu}_{-t_0 \Delta+\mu_0 I_{L^d}}), \varepsilon\right)\right\}.
\end{align*}
As in the previous proofs, by Markov's inequality, it suffices to prove that if $\mu_0 > \mu_c\left(t_0, 4B''(0), L, d\right)$,
\begin{align}
\label{Nu}
    \lim _{N \rightarrow \infty}\mathbb{E}\left[\#\left\{\mathbf{u} \in \Xi_{\varepsilon}:\nabla \mathcal{H}(\mathbf{u})=0\right\}\right]= 0,
\end{align}
which is further implied by
\begin{align}
\label{loge}
  \limsup _{N \rightarrow \infty} \frac{1}{N L^d} \log \mathbb{E}\left[\#\left\{\mathbf{u} \in \Xi_{\varepsilon}:\nabla \mathcal{H}(\mathbf{u})=0\right\}\right]<0. 
\end{align}
The Kac--Rice formula in \cite{FLD20} gives
\begin{align}
\label{logkac}
&\frac{1}{N L^d} \log \mathbb{E}\left[\#\left\{\mathbf{u} \in \Xi_{\varepsilon}:\nabla \mathcal{H}(\mathbf{u})=0\right\}\right]=  -\frac{1}{L^d} \log \left(\operatorname{det}\left(\mu_0I_{L^d}-t_0 \Delta\right)\right)+\frac{1}{N L^d} \\
&\quad\times \log \int_{\mathbb{R}^{L^d}} \frac{e^{-N \frac{\|u\|^2}{2 J^2}}}{\left(\sqrt{2 \pi J^2 / N}\right)^{L^d}} \mathbb{E}\left[\left|\operatorname{det}\left(H_N(u)\right)\right|\indi_{d(\hat{\mu}_{H_N(u)},\tau_{-u^*}(\sigma_{0, 4\sqrt{B''}} \boxplus \hat{\mu}_{-t_0 \Delta+\mu_0 I_{L^d}}))> \varepsilon}\right] \mathrm{d} u.\nonumber
\end{align}
For any $\delta>0$, denote by $O_{\delta}(\hat{u}^*)$ the open ball with center $\hat{u}^*:=(u^*,\dots, u^*)\in\mathbb{R}^{L^d}$ and radius $\delta$, and write $O_{\delta}(\hat{u}^*)^c$ for its complement. We divide the integral in Equation (\ref{logkac}) into the following two parts
\begin{align}
\label{s1s2}
&\int_{\mathbb{R}^{L^d}} \frac{e^{-N \frac{\|u\|^2}{2 J^2}}}{\left(\sqrt{2 \pi J^2 / N}\right)^{L^d}} \mathbb{E}\left[\left|\operatorname{det}\left(H_N(u)\right)\right|\indi_{d(\hat{\mu}_{H_N(u)},\tau_{-u^*}(\sigma_{0, 4\sqrt{B''}} \boxplus \hat{\mu}_{-t_0 \Delta+\mu_0 I_{L^d}}))> \varepsilon}\right] \mathrm{d} u\\ &
=\int_{O_{\frac{\eps}{2}}(\hat{u}^*)} \frac{e^{-N \frac{\|u\|^2}{2 J^2}}}{\left(\sqrt{2 \pi J^2 / N}\right)^{L^d}} \mathbb{E}\left[\left|\operatorname{det}\left(H_N(u)\right)\right|\indi_{d(\hat{\mu}_{H_N(u)},\tau_{-u^*}(\sigma_{0, 4\sqrt{B''}} \boxplus \hat{\mu}_{-t_0 \Delta+\mu_0 I_{L^d}}))> \varepsilon}\right] \mathrm{d} u
\nonumber\\ &\quad + \int_{O_{\frac{\eps}{2}}(\hat{u}^*)^c} \frac{e^{-N \frac{\|u\|^2}{2 J^2}}}{\left(\sqrt{2 \pi J^2 / N}\right)^{L^d}} \mathbb{E}\left[\left|\operatorname{det}\left(H_N(u)\right)\right|\indi_{d(\hat{\mu}_{H_N(u)},\tau_{-u^*}(\sigma_{0, 4\sqrt{B''}} \boxplus \hat{\mu}_{-t_0 \Delta+\mu_0 I_{L^d}}))> \varepsilon}\right] \mathrm{d} u
\nonumber \\ &
=:\mathbf{S}_N^1+\mathbf{S}_N^2.\nonumber
\end{align}
To show \eqref{loge}, we have to prove
\begin{align}
\label{s1}
-\frac{1}{L^d} \log \left(\operatorname{det}\left(\mu_0I_{L^d}-t_0 \Delta\right)\right)+ \limsup _{N \rightarrow \infty} \frac{1}{N L^d} \log \mathbf{S}_N^1 <0  , 
\end{align}
and
\begin{align}
\label{s2}
-\frac{1}{L^d} \log \left(\operatorname{det}\left(\mu_0I_{L^d}-t_0 \Delta\right)\right)+ \limsup _{N \rightarrow \infty} \frac{1}{N L^d} \log \mathbf{S}_N^2 <0 .  
\end{align}

For (\ref{s2}), dropping the indicator function in $\mathbf{S}_N^2$ implies
\begin{align}
\label{dropindi}
&\limsup _{N \rightarrow \infty} \frac{1}{N L^d} \log \mathbf{S}_N^2 \\ &\leq 
\limsup _{N \rightarrow \infty} \frac{1}{N L^d} \log\int_{O_{\frac{\eps}{2}}(\hat{u}^*)^c} \frac{e^{-N \frac{\|u\|^2}{2 J^2}}}{\left(\sqrt{2 \pi J^2 / N}\right)^{L^d}} \mathbb{E}\left[\left|\operatorname{det}\left(H_N(u)\right)\right|\right] \mathrm{d} u.\nonumber  
\end{align}
Together with the arguments as in {\cite[Propsition 4.8]{BABM21}}, we apply {\cite[Theorem 4.1]{BABM22}} to get
\begin{align}
\label{supu}
&\lim_{N \rightarrow \infty} \frac{1}{N L^d} \log\int_{O_{\frac{\eps}{2}}(\hat{u}^*)^c} \frac{e^{-N \frac{\|u\|^2}{2 J^2}}}{\left(\sqrt{2 \pi J^2 / N}\right)^{L^d}} \mathbb{E}\left[\left|\operatorname{det}\left(H_N(u)\right)\right|\right] \mathrm{d} u \\&=\sup _{u \in O_{\frac{\eps}{2}}(\hat{u}^*)^c}\left\{\int \log |\lambda| \mu_{\infty}(u, \lambda) \mathrm{d} \lambda-\frac{\|u\|^2}{2 J^2 L^d}\right\}. \nonumber  
\end{align}
For any $u\in \mathbb{R}^{L^d}$, we define $$\mathbb{S}[u]=\int \log |\lambda| \mu_{\infty}(u, \lambda) \mathrm{d} \lambda-\frac{\|u\|^2}{2 J^2 L^d}.$$ 
From {\cite[Propsition 4.11]{BABM21}} and the proof of {\cite[Theorem 2.2]{BABM21}}, we know
\begin{align}
\label{sups}
\sup _{u \in  \mathbb{R}^{L^d}}\mathbb{S}[u]= \sup _{u \in  \mathbb{R}}\mathbb{S}[(u,\dots,u)],  
\end{align}
and 
\begin{align}
\label{muinf}
    \mu_{\infty}((u, \ldots, u), \lambda)=\mu_{\infty}((0, \ldots, 0), \lambda-u),
\end{align}
with the probability measure $\mu_{\infty}((0, \ldots, 0))=\sigma_{0, 4\sqrt{B''}}\boxplus \hat{\mu}_{-t_0 \Delta+\mu_0 I_{L^d}}$.
Note that the condition $\mu_0 > \mu_c\left(t_0, 4B''(0), L, d\right)$ in Theorem \ref{Elas} is equivalent to the condition $t<t_c$ in Lemma \ref{unimax} with $t=4B''(0)$ and $\mu=\hat{\mu}_{-t_0 \Delta+\mu_0 I_{L^d}}$. In this context we have $-u_t=u^*$. The equality (\ref{sups}) together with Lemma \ref{unimax} implies that $\hat{u}^*\in\mathbb{R}^{L^d}$ is a maximizer of $\mathbb{S}[u]$ with $$-\frac{1}{L^d} \log\left(\operatorname{det}\left(\mu_0I_{L^d}-t_0 \Delta\right)\right)+\mathbb{S}[\hat{u}^*]=0.$$
 Recall from \cite[Proposition 4.9]{BABM21} that $\mathbb{S}[u]$ is concave on $\mathbb{R}^{L^d}$. 
The discussion above Lemma \ref{unimax} tells us that the function $\mathbb{S}[u]$
is strictly concave in the diagonal direction in a neighborhood of $\hat{u}^*$, which implies that $\hat{u}^*$ is the unique maximizer of $\mathbb{S}[u]$ in the diagonal direction. 

We will show that $\hat{u}^*$ is the unique maximizer of $\mathbb{S}[u]$ in $\mathbb{R}^{L^d}$. Recall the unique solution $M(u,z)$ in (\ref{M}). 
Note that $\mu_{\infty}(\hat{u}^*)=\tau_{-u^*}(\sigma_{0, 4\sqrt{B''}} \boxplus \hat{\mu}_{-t_0 \Delta+\mu_0 I_{L^d}})$ and $0$ is not in the support of this measure from the discussion above Lemma \ref{unimax}, which implies that 
\begin{align}\label{trice}
\operatorname{Tr}\left(\Im M(\hat{u}^*, \mathrm{i}0^+)\right)=0,    
\end{align}
where the symbol $\mathrm{i}0^+$ is understood as taking the real part of $z$ to be zero and sending the imaginary part of $z$ to zero. Since $\Im M(u,z)$ is positive definite for any $z\in\mathbb{H}$ by the definition (\ref{M}), the matrix $\Im M(\hat{u}^*,\mathrm{i}0^+)$ is positive semi-definite \cite[(3.11b)]{location}. This together with (\ref{trice}) further gives that for all $1\leq j\leq L^d$,
\begin{align*}
 \Im M_{jj}(\hat{u}^*, \mathrm{i}0^+)=0,\quad\text{and} \quad \Im M(\hat{u}^*,\mathrm{i}0^+)=\mathbf{0}. \end{align*} 
Since $\hat{u}^*$ is a maximizer of $\mathbb{S}[u]$, by \cite[(4.17)]{BABM21} (taking $\eta$ to zero\footnote{Note that the negative sign appeared in this equation should be in front of $E_{kk}$ in the next displayed equation, which makes \cite[(4.18)]{BABM21} valid.}), we obtain for all $1\leq j\leq L^d$,
\begin{align}\label{reM}
\Re M_{jj}(\hat{u}^*,\mathrm{i}0^+)= M_{jj}(\hat{u}^*,\mathrm{i}0^+)= \frac{u^*}{J^2}.    
\end{align}
Substituting the above identity to (\ref{M}) gives that
\begin{align}\label{u8}
\left(M(\hat{u}^*,\mathrm{i}0^+)\right)^{-1}=-t_0\Delta + \mu_0 I_{L^d}.    
\end{align}
To prove the uniqueness of the maximizer of $\mathbb{S}[u]$, it suffices to show that $\hat{u}^*$ is the unique maximizer in a small neighborhood of $\hat{u}^*$, since $\mathbb{S}[u]$ is concave on $\mathbb{R}^{L^d}$. From the proof of \cite[Lemma 4.7]{BABM21}, we know that the left edge of $\mu_{\infty}(u)$ is a concave function in $u$; that is, for any $u,v\in \mathbb{R}^{L^d}$ and $t\in [0,1]$,
$$
\ell\left(\mu_{\infty}(t u+(1-t) v)\right) \geq t \ell\left(\mu_{\infty}(u)\right)+(1-t) \ell\left(\mu_{\infty}(v)\right).
$$
As a result, the left edge of $\mu_{\infty}(u)$ is
continuous with respect to $u$.
Thus, there exists a small enough $\delta$ such that for any $u\in O_{\delta}(\hat{u}^*)$, $0$ is also not in the support of the measure $\mu_{\infty}(u)$. Suppose that $O_{\delta}(\hat{u}^*)\ni u'=(u'_1,\dots,u'_{L^d})\neq \hat{u}^*$ is another maximizer of $\mathbb{S}[u]$. Then by an argument similar to (\ref{reM}) and (\ref{u8}), we have for all $1\leq j\leq L^d$,
\begin{align*}
\Re M_{jj}(u',\mathrm{i}0^+)= M_{jj}(u',\mathrm{i}0^+)= \frac{u'_{j}}{J^2}.    
\end{align*}
and
\begin{align*}
\left(M(u',\mathrm{i}0^+)\right)^{-1}=-t_0\Delta + \mu_0 I_{L^d},    
\end{align*}
which leads to a contradiction, since the  matrix $-t_0\Delta + \mu_0 I_{L^d}$ cannot have different inverse matrices.
From here, $\hat{u}^*$ is the unique maximizer of $\mathbb{S}[u]$ in $\mathbb{R}^{L^d}$. Then we deduce that for any $\varepsilon>0$
\begin{align}\label{strictsmall0}
    -\frac{1}{L^d} \log\left(\operatorname{det}\left(\mu_0I_{L^d}-t_0 \Delta\right)\right)+\sup _{u \in O_{\frac{\eps}{2}}(\hat{u}^*)^c}\left\{\int \log |\lambda| \mu_{\infty}(u, \lambda) \mathrm{d} \lambda-\frac{\|u\|^2}{2 J^2 L^d}\right\}<0.
\end{align}
Combining (\ref{dropindi}), (\ref{supu}) and (\ref{strictsmall0}) gives the inequality (\ref{s2}). 

Now we turn to the proof of (\ref{s1}). As $u\in O_{\frac{\eps}{2}}(\hat{u}^*)$, we find
\begin{align}
\label{lambdak}
|\lambda_k(H_N(\hat{u}^*))-\lambda_k(H_N(u))|\leq \|\operatorname{diag}(u-\hat{u}^*)\otimes I_N\|\leq \frac{\eps}{2} \quad \text{for all} \quad 1\leq k\leq NL^d,
\end{align}
where we denote by $\|M\|$ the operator norm of the matrix $M$. It follows that
\begin{equation}
\label{disuu}
d(\hat{\mu}_{H_N(u)},\hat{\mu}_{H_N(\hat{u}^*)})<\frac{\varepsilon}{2}. 
\end{equation}
On the other hand, by the triangle inequality we have
\begin{align}
\varepsilon&<d(\hat{\mu}_{H_N(u)},\tau_{-u^*}(\sigma_{0, 4\sqrt{B''}} \boxplus \hat{\mu}_{-t_0 \Delta+\mu_0 I_{L^d}}))\\&\leq d(\hat{\mu}_{H_N(u)},\hat{\mu}_{H_N(\hat{u}^*)})+ d(\hat{\mu}_{H_N(\hat{u}^*)},\tau_{-u^*}(\sigma_{0, 4\sqrt{B''}} \boxplus \hat{\mu}_{-t_0 \Delta+\mu_0 I_{L^d}})).\nonumber
\end{align}
Along with (\ref{disuu}), we obtain
\begin{align}
d(\hat{\mu}_{H_N(\hat{u}^*)},\tau_{-u^*}(\sigma_{0, 4\sqrt{B''}} \boxplus \hat{\mu}_{-t_0 \Delta+\mu_0 I_{L^d}}))>\frac{\varepsilon}{2}.
\end{align}
Note that $\mu_{\infty}(\hat{u}^*)=\tau_{-u^*}(\sigma_{0, 4\sqrt{B''}} \boxplus \hat{\mu}_{-t_0 \Delta+\mu_0 I_{L^d}}).$ 
Combining these facts with the Cauchy–Schwarz inequality, we deduce 
\begin{align}\label{enlars1}
   & \frac{1}{NL^d}\log \mathbf{S}_N^1\\ &\leq
\frac{1}
{NL^d}\log\int_{O_{\frac{\eps}{2}}(\hat{u}^*)} \frac{e^{-N \frac{\|u\|^2}{2 J^2}}}{\left(\sqrt{2 \pi J^2 / N}\right)^{L^d}} \nonumber\\&\quad\quad\quad\quad\quad\times\mathbb{E}\left[\left|\operatorname{det}\left(H_N(u)\right)\right|\indi_{d(\hat{\mu}_{H_N(\hat{u}^*)},\tau_{-u^*}(\sigma_{0, 4\sqrt{B''}} \boxplus \hat{\mu}_{-t_0 \Delta+\mu_0 I_{L^d}}))>\frac{\varepsilon}{2}}\right] \mathrm{d} u\nonumber\\ &\leq
\frac{1}
{NL^d}\log\int_{O_{\frac{\eps}{2}}(\hat{u}^*)} \frac{e^{-N \frac{\|u\|^2}{2 J^2}}}{\left(\sqrt{2 \pi J^2 / N}\right)^{L^d}} \mathbb{E}\left[\left|\operatorname{det}\left(H_N(u)\right)\right|^2\right]^{1/2}\nonumber\\&\quad\quad\quad\quad\quad\times\mathbb{P}\left(d(\hat{\mu}_{H_N(\hat{u}^*)},\mu_{\infty}(\hat{u}^*)>\frac{\varepsilon}{2}\right)^{1/2} \mathrm{d} u.\nonumber
\end{align}
Arguing as in {\cite[Lemma 4.4]{BABM21}}, there exists some constant $C>0$ such that 
\begin{align}
\label{edet}
\mathbb{E}\left[\left|\operatorname{det}\left(H_N(u)\right)\right|^2\right]^{1/2}\leq (C \max (\|u\|, 1))^N.
\end{align}
Using {\cite[Lemma 4.3]{BABM21}} and {\cite[Corollary 1.4 (b)]{GZ00}}, for any $\eps>0$, 
there exist some constants $C_1>0$ and $C_2>0$ (depending on $\eps$, $L$ and $d$) such that for $N$ large enough
\begin{align}
\label{pdist}
\mathbb{P}\left(d(\hat{\mu}_{H_N(\hat{u}^*)},\mu_{\infty}(\hat{u}^*))>\eps\right) \leq C_1\exp\{-C_2N^2\}.
\end{align}
Plugging the estimates (\ref{edet}) and (\ref{pdist}) into (\ref{enlars1}) directly gives
\begin{align}
  \limsup _{N \rightarrow \infty} \frac{1}{NL^d}\log \mathbf{S}_N^1=-\infty,
\end{align}
which completes the proof of 
(\ref{s1}) and thus the assertion (\ref{meau}). 

The second conclusion on the left edge is a direct consequence of the Pastur self-consistency equation \eqref{Stie} as already rigorously justified by Fyodorov and Le Doussal \cite{fyodorov2020manifolds}. We reproduce their argument briefly for the reader's convenience.
Indeed, taking the real and imaginary parts of the Pastur relation (\ref{Stie}) yields the following equations for any $z$ with $\Im z\ge0$
\begin{align}\label{ri}
{\Re}(m(z))&=\int \frac{s+u^*-z-4B''(0){\Re}(m(z))}{(s+u^*-z-4B''(0){\Re}(m(z)))^2+16B''(0)^2({\Im}m(z))^2}\hat{\mu}_{-t_0 \Delta+\mu_0 I_{L^d}}(\mathrm{d} s),\\  
\Im(m(z))&=\int \frac{4B''(0){\Im}(m(z))}{(s+u^*-z-4B''(0){\Re}(m(z)))^2+16B''(0)^2({\Im}m(z))^2}\hat{\mu}_{-t_0 \Delta+\mu_0 I_{L^d}}(\mathrm{d} s)\nonumber.
\end{align}
Writing $\ell=\ell(\tau_{-u^*}(\sigma_{0, 4\sqrt{B''}} \boxplus \hat{\mu}_{-t_0 \Delta+\mu_0 I_{L^d}}))$, we have ${\Im}(m(\ell))=0$. Replacing $z$ by $\ell$ in (\ref{ri}), we obtain
\begin{align}\label{mell}
{\Re}(m(\ell))&=\int \frac{\hat{\mu}_{-t_0 \Delta+\mu_0 I_{L^d}}(\mathrm{d} s)}{s+u^*-\ell-4B''(0){\Re}(m(\ell))},\\  
\frac{1}{4B''(0)}&=\int \frac{\hat{\mu}_{-t_0 \Delta+\mu_0 I_{L^d}}(\mathrm{d}s)}{(s+u^*-\ell-4B''(0){\Re}(m(\ell)))^2}\nonumber.
\end{align}
Combining the second equation in (\ref{mell}) and definition of the Larkin mass (\ref{larkin}) implies
\begin{align}\label{muc}
   \mu_c=\mu_0+u^*-\ell-4B''(0){\Re}(m(\ell)).
\end{align}
Together with the first equation in (\ref{mell}), we find
$${\Re}(m(\ell))=\int \frac{\hat{\mu}_{-t_0 \Delta}(\mathrm{d} s)}{s+\mu_c}.$$
Plugging the above value of ${\Re}(m(\ell))$ and the definition of $u^*$ as in (\ref{ustar}) into (\ref{muc}) yields \eqref{eq:emledge}.
\end{proof}

\begin{remark}
    At the time of this writing, we are unable to prove the convergence of the smallest eigenvalues of Hessians at the global minimum of elastic manifold. This is due to the lack of large deviation estimates for the smallest eigenvalues of block diagonal matrices in \eqref{WH}. Once such estimates become available, the argument leading to \eqref{xx8} and \eqref{zx} should also work here.
\end{remark}

\textbf{Acknowledgments.}
    We are in debt to Antonio Auffinger, who nicely forwarded a question of Fyodorov to us. This question was about a rigorous proof of the prediction made in \cite{FLD18} after the first version of the paper \cite{AZ20} was posted on arXiv. One year later, we eventually recalled Fyodorov's talk in the 2018 workshop on spin glasses at Banff, and started working on this problem. We would like to thank Robert J. Adler for providing us the reference \cite{KP90}, Zhigang Bao for helpful discussions, Yan Fyodorov for his interest on this paper and encouragement, and Benjamin McKenna for conversions on large deviations of extreme eigenvalues of random matrices. As we are finishing this revised version (elastic manifold is treated using the same strategy as before), we learned from Benjamin McKenna that he and collaborators have also obtained the convergence of the empirical spectral measures of Hessians at the global minimum of elastic manifold in the replica symmetric regime.

    We are grateful to the referees for their careful reading and many constructive suggestions which have significantly improved this paper.

    This work was partially supported by SRG 2020-00029-FST and FDCT 0132/2020/A3.

\bibliographystyle{abbrev}
\bibliography{gfic}
\end{document}